\documentclass[11pt,reqno]{amsart}
\usepackage{amsmath}
\usepackage{amsfonts}
\usepackage{amssymb}
\usepackage[dvips]{graphicx}
\usepackage{color}
\usepackage[pagewise]{lineno}
\usepackage{url}
\theoremstyle{plain}

\newtheorem{theorem}{Theorem}[section]
\newtheorem{proposition}[theorem]{Proposition}
\newtheorem{corollary}[theorem]{Corollary}
\newtheorem{lemma}[theorem]{Lemma}

\theoremstyle{definition}

\newtheorem{definition}[theorem]{Definition}

\newtheorem{remark}[theorem]{Remark}

\DeclareMathOperator{\esssup}{ess \ sup}
\DeclareMathOperator{\osc}{osc}
\DeclareMathOperator{\lip}{Lip}
\DeclareMathOperator{\Var}{Var}
\DeclareMathOperator{\var}{var}
\DeclareMathOperator{\diam}{diam}
\input{tcilatex}

\begin{document}

\begin{abstract}
We consider transformations preserving a contracting foliation, such that
the associated quotient map satisfies a Lasota-Yorke inequality. We prove
that the associated transfer operator, acting on suitable normed spaces, has
a spectral gap (on which we have quantitative estimation).

As an application we consider Lorenz-like two dimensional maps (piecewise
hyperbolic with unbounded contraction and expansion rate): we prove that
those systems have a spectral gap and we show a quantitative estimate for
their statistical stability. Under deterministic perturbations of the system
of size $\delta$, the physical measure varies continuously, with a modulus
of continuity $O(\delta \log \delta )$, which is asymptotically optimal for
this kind of piecewise smooth maps.
\end{abstract}

\title[Spectral Gap for 2-dimensional Contracting Fibers Systems]{Spectral
Gap and quantitative statistical stability for systems with contracting
fibers and Lorenz-like maps.}
\author[Stefano Galatolo]{Stefano Galatolo}
\author[Rafael Lucena]{Rafael Lucena}
\date{\today }
\address[Rafael Lucena]{Universidade Federal do Rio de Janeiro - UFRJ and
Universit\`{a} di Pisa - UNIPI}
\email{rafael.lucena@im.ufal.br}
\urladdr{www.mathlucena.blogspot.com}
\address[Stefano Galatolo]{Dipartimento di Matematica, Universit\`{a} di Pisa - UNIPI}
\email{stefano.galatolo@unipi.it}
\urladdr{http://users.dma.unipi.it/galatolo/}
\maketitle


\section{Introduction}

The study of the behaviour of the transfer operator restricted to a suitable
functional space has proven to be a powerful tool for the understanding of
the statistical properties of a dynamical system. This approach gave first
results (see \cite{LY}, \cite{L3} and \cite{RE}) in the study of the dynamics of piecewise expanding maps where the
involved spaces are made of regular, absolutely continuous measures (see 
\cite{Ba}, \cite{L2}, \cite{BG} and \cite{G} for some introductory text). In
recent years the approach was extended to piecewise hyperbolic systems by
the use of suitable anisotropic norms (the expanding and contracting
directions are managed differently), leading to suitable distribution spaces
on which the transfer operator has good spectral properties (see e.g. \cite%
{BT}, \cite{BaG}, \cite{BaG2}, \cite{DL}, \cite{GL} and \cite{B},\cite{D}
for recent papers containing a survey of the topic). \ From these
properties, several limit theorems or stability statements can be deduced.
This approach has proven to be successful in non-trivial classes of systems
like geodesic flows (see \cite{L2}, \cite{BL}) or billiard maps (ess e.g. 
\cite{DZ} \cite{DZ2} where a relatively simple and unified approach to many
limit and perturbative results is given for the Lorentz gas). In these
approaches, usually some condition of boundedness of the derivatives or
transversality between the map's singular set and the contracting directions
is supposed.

In this work, we consider skew product maps preserving a uniformly
contracting foliation. We show how it is possible, in a simple way, to
define suitable spaces of signed measures (with an anisotropic norm) such
that, under small regularity assumptions, the transfer operator associated
to the dynamics has a spectral gap (in the sense given in Theorem \ref{spgap}%
). This shows an exponential convergence to $0$ in a certain norm for the
iteration of a large class of zero average measures by the transfer
operator. In this approach the speed of this convergence can be
quantitatively estimated, and depends on the rate of contraction of the
stable foliation, the coefficients of the Lasota-Yorke inequality and the
rate of convergence to equilibrium of the induced quotient map (see Remark %
\ref{quantitative2}). We also remark that in our approach we can deal with
piecewise continuous maps having piecewise $C^{1+\alpha }$ regularity,
having unbounded derivatives, and where the discontinuity set is parallel to
the contracting direction, as it happen in the Lorenz-like maps we consider
in Section \ref{last}. These results allow to obtain in the second part of
the paper a quantitative statistical stability estimate for deterministic
perturbations of this kind of Lorenz-like systems. The result applies to
deterministic perturbations of skew product maps with a piecewise expanding
map on the base with $C^{2}$ branches and contracting behaviour on the
fibers. Essentially the main theorem of the section states (see Theorem \ref%
{mainstat}) that the physical measure of the system varies with a modulus of
continuity of the type $\delta \log (\delta )$ under perturbations of size $%
\delta $ ( see Section \ref{realast} for precise statements and definitions)
in a strong topology determined by a certain anisotropic space of signed
measures which will be described below. It is worth to remark that this
bound is also asymptotically optimal (see Remark \ref{opti}).

The function spaces we consider are defined by disintegrating signed
measures on the phase space along the contracting foliation. The signed
measure itself is then seen as a family of measures on the contracting
leaves. We can then consider some notion of regularity for this family to
define suitable spaces of more or less \textquotedblleft
regular\textquotedblright measures where to apply our transfer operator. To
give an idea of these function spaces (see section \ref{sec:spaces}), in the
case of skew product maps of the unit square $I\times I$ to itself, the
disintegration gives rise to a one dimensional family (a path) of measures
defined on the contracting leaves, each leaf is isomorphic to the unit
interval $I$, hence a measure on $I\times I$ is seen as a path of measures
on $I$: a path in a metric space. The function spaces are defined by
suitable notions of regularity for these paths. In the case $I\times I$ for
example, the spaces which arise are included in $L^{1}(I,Lip(I)^{\prime })$
\ (the space of $L^{1}$ functions from the interval to the dual of the space
of Lipschitz functions on the interval), imposing some kind of further
regularity. This is a space of distribution valued functions. For simplicity
we will only use normed vector spaces of signed measures in this paper, we
do not need to consider the completion of the space of signed measure, which
would lead to distribution spaces. Similar strong and weak function spaces
have been used in \cite{Gjep} to investigate quantitatively the statistical
stability of slowly mixing toral extensions (skew products with a non
expanding preserved foliation).

\noindent\textbf{Plan of the paper.} The paper is structured as follows:

\begin{itemize}
	\item in Section 2 we introduce the kind of systems we consider in the
	paper. Essentially, these are skew product maps, with a base map satisfying
	a Lasota-Yorke inequality with respect to suitable spaces (piecewise
	expanding maps e.g.) and the fibers are contracted;
	
	\item in Section 3 we introduce the functional spaces used in the paper and
	discussed in the previous paragraphs;
	
	\item in Section 4 we show the basic properties of the transfer operator
	when applied to these spaces. In particular we see that there is an useful
	\textquotedblleft Perron-Frobenius\textquotedblright-like formula (see
	Proposition \ref{niceformulaab}) .
	
	\item In Section 5 we see the basic properties of the iteration of the
	transfer operator on the spaces we consider. In particular we see \emph{%
		Lasota-Yorke inequalities and a convergence to equilibrium statement} (see
	Propositions \ref{lasotaoscilation2} and \ref{5.8}).
	
	\item In Section 6 we use the convergence to equilibrium and the
	Lasota-Yorke inequalities to prove the \emph{spectral gap} for the transfer
	operator associated to the system restricted to a suitable strong space (see
	Theorems \ref{spgap} and \ref{speclinf}).
	
	\item In Section \ref{last} we present an application of our construction,
	showing a \emph{spectral gap for 2-dimensional Lorenz-like maps} (piecewise $%
	C^{1+\alpha }$ hyperbolic maps with unbounded expansion and contraction
	rates).
	
	\item In Section \ref{realast} we consider similar systems with some more
	regularity. We apply our construction to a class of piecewise $C^{2}$,
	two-dimensional Lorenz-like maps. We prove stronger (bounded variation like)
	regularity results for the iteration of probability measures on that
	systems, and use this to prove a \emph{quantitative statistical stability}
	statement with respect to deterministic perturbations: we establish a
	modulus of continuity $\delta \log \delta $ for the stability of the
	physical measure in weak space ($L^{1}(I,Lip(I)^{\prime })$) after \ a
	\textquotedblleft size $\delta $\textquotedblright perturbation (see Theorem %
	\ref{mainstat}). Qualitative statements, for classes of similar maps were
	shown in \cite{AS} and very recently in \cite{BR}.
\end{itemize}

\textbf{Acknowledgments} This work was partially supported by Alagoas Research Foundation - FAPEAL (Brazil) Grants 60030 000587/2016, CNPq (Brazil) Grants 300398/2016-6, CAPES (Brazil) Grants 99999.014021/2013-07 and EU Marie-Curie IRSES Brazilian-European partnership in Dynamical Systems (FP7-PEOPLE- 2012-IRSES 318999 BREUDS).

\section{Contracting Fiber Maps\label{sec2}}

In this section we introduce the kind of systems we are considering in this
paper and show some of its basic properties. Consider $\Sigma =N_{1}\times
N_{2}$, where $N_{1}$ and $N_{2}$ are compact and finite dimensional
Riemannian manifolds such that $\diam(N_{2})=1$, where $\diam(N_{2})$
denotes the diameter of $N_{2}$ with respect to its Riemannian metric $d_{2}$%
. This is not restrictive but will avoid some multiplicative constants.
Denote by $m_{1}$ and $m_{2}$ the Lebesgue measures on $N_{1}$ and $N_{2}$
respectively, generated by their corresponding Riemannian volumes,
normalized so that $m_{1}(N_{1})=m_{2}(N_{2})=1$ and $m=m_{1}\times m_{2}$.
Consider a map $F:\left( \Sigma ,m\right) \longrightarrow \left( \Sigma
,m\right) $, 
\begin{equation*}
	F(x,y)=(T(x),G(x,y)),  \label{1eq}
\end{equation*}%
where $T:N_{1}\longrightarrow N_{1}$ and $G:\Sigma \longrightarrow N_{2}$
are measurable maps. Suppose that these maps satisfy the following conditions

\subsubsection{Properties of $G$}

\begin{description}
	\item[G1] Consider the $F$-invariant foliation 
	\begin{equation*}
		\mathcal{F}^{s}:=\{\{x\}\times N_2\}_{x\in N_1}.  \label{fol}
	\end{equation*}
	We suppose that $\mathcal{F}^{s}$ is contracted: there exists $0<\alpha <1$
	such that for all $x\in N_{1}$ it holds%
	\begin{equation}
		d_2(G(x,y_{1}),G(x,y_{2}))\leq \alpha d_2(y_{1},y_{2}),\ \ \mathnormal{for\
			all}\ \ y_{1},y_{2}\in N_{2}.  \label{contracting1}
	\end{equation}
\end{description}

\subsubsection{Properties of $T$ and of its associated transfer operator.}

Suppose that:

\begin{description}
	\item[T1] $T$ is non-singular with respect to $m_{1}$ ($m_{1}(A)=0%
	\Rightarrow m_{1}(T^{-1}(A)))=0$);
	
	\item[T2] There exists a disjoint collection of open sets $\mathcal{P}%
	=\{P_{1},\cdots ,P_{q}\}$ of $N_{1}$, such that $m_{1}\left(
	\bigcup_{i=1}^{q}{P_{i}}\right) =1$ and $T_{i}:=T|_{P_{i}}$ is a
	diffeomorphism $T_{i}:P_i \to T_{i}(P_i) \subseteq N_1$, with $\det
	DT_i(x)\neq 0$ for all $x \in P_{i}$ and for all $i$, where $DT_i$ is the
	Jacobian matrix of $T _i$ with respect to the Riemannian metric of $N_{1}$;
	
	\item[T3] Let us consider the Perron-Frobenius Operator associated to $T$, $%
	\func{P}_{T}$ \footnote{%
		The unique operator $\func{P}_{T}:L_{m_{1}}^{1}\longrightarrow L_{m_{1}}^{1}$
		such that 
		\begin{equation*}
			\forall \phi \in L_{m_{1}}^{1}\ \ \mathnormal{and}\ \ \forall \psi \in
			L_{m_{1}}^{\infty }\ \ \int {\psi \cdot \func{P}_{T}(\phi )~}dm_1=\int {%
				\left( \psi \circ T\right) \cdot \phi ~}dm_1.
		\end{equation*}%
	}. We will now make some assumptions on the existence of a suitable
	functional analytic setting adapted to ${\func{P}_T}$. Let us hence denote
	the $L_{m_{1}}^{1}$ norm\footnote{\textbf{Notation: }In the following we use 
		$|\cdot| $ to indicate the usual absolute value or norms for signed measures
		on the basis space $N_{1}.$ We will use $|| \cdot||$ for norms defined for
		signed measures on $\Sigma $.
		\par
		{}} by $|\cdot |_{1}$ and suppose that there exists a Banach space $%
	(S_{\_},| \cdot |_{s})$ such that
	
	\item[T3.1] $S_{\_}\subset L_{m_{1}}^{1}$ is $\func{P}_{T}$-invariant, $%
	|\cdot |_{1}\leq |\cdot |_{s}$ and $\func{P}_{T}: S_{\_} \longrightarrow
	S_{\_}$ is bounded;
	
	\item[T3.2] The unit ball of $(S_{\_},|\cdot |_{s})$ is relatively compact
	in $(L_{m_{1}}^{1},|\cdot|_{1})$;
	
	\item[T3.3] (Lasota-Yorke inequality) There exist $k\in \mathbb{N}$, $%
	0<\beta _{0}<1$ and $C>0$ such that, for all $f\in S_{\_}$, it holds 
	\begin{equation*}
		|\func{P}_{T}^{k}f|_{s}\leq \beta _{0}|f|_{s}+C|f|_{1};  \label{LY1}
	\end{equation*}
	
	\item[T3.4] Suppose there is an unique $\psi _{x}\in S_{\_}$ with $\psi
	_{x}\geq 0$ and $|\psi _{x}|_{1}=1$ such that $\func{P}_{T}(\psi _{x})=\psi
	_{x}$, and if $\psi \in S_{\_}$ is another density for a probability
	measure, then ${P}_{T}^{n}(\psi _{x}-\psi )\rightarrow 0$ as $n \rightarrow
	\infty$ in $S_{\_}$.\footnote{%
		This assumption ensures that from our point of view the system is
		indecomposable. For piecewise expanding maps e.g., the assumption follows
		from topological mixing.}
\end{description}

It is known that in this case (\cite{IM}, see also \cite{BG}, \cite{L2} )
the following holds.

\begin{theorem}
	\label{loiub} If \ $T$ satisfy $T3.1,...,T3.4$ then there exist $0<r<1$ and $%
	D>0$ such that for all $f\in S_{\_}$ with $\int {f~}dm_{1}=0$ and for all $%
	n\geq 0,$ it holds 
	\begin{equation}
		|\func{P}_{T}^{n}(f)|_{s}\leq Dr^{n}|f|_{s}.  \label{p2}
	\end{equation}
\end{theorem}

In order to obtain spectral gap on $L^{\infty }$ like spaces, the following
additional property on $|\cdot|_{s}$ will be supposed at some point in the
paper.

\begin{description}
	\item[N1] There is $H_N\geq 0$ such that $|\cdot|_{\infty }\leq H_N
	|\cdot|_{s}$ (where $|\cdot|_{\infty }$ is the usual $L^{\infty } _{m_1}$
	norm on $N_{1}$).
\end{description}

The following is a standard consequence of item T3.3, allowing to estimate
the behaviour of any given power of the transfer operator.

\begin{corollary}
	There exist constants $B_3>0$, $C_2>0$ and $0<\beta_2<1$, such that for all $%
	f \in S_{\_}$, and all $n \geq 1$, it holds
	
	\begin{equation}
		|\func{P}_{T}^{n}f|_{s} \leq B_3 \beta _2 ^n | f|_{s} + C_2|f|_{1}.
		\label{lasotaiiii}
	\end{equation}
\end{corollary}

\section{Weak and strong spaces\label{sec:spaces}}

\subsection{$L^{1}$-like spaces.}

Through this section we construct some function spaces which are suitable
for the systems defined in section \ref{sec2}. The idea is to define spaces
of signed measures, where the norms are provided by disintegrating measures
along the stable foliation. Thus, a signed measure will be seen as a family
of measures on each leaf. For instance, a measure on the square with a
vertical foliation will be seen as a one parameter family (a path) of
measures on the interval (a stable leaf), where this identification will be
done by means of the Rokhlin's Disintegration Theorem. Finally, in the
vertical direction (on the leaves), we will consider a norm which is the
dual of the Lipschitz norm and in the \textquotedblleft
horizontal\textquotedblright direction we will consider essentially the $%
L_{m_1}^{1}$ norm.

\subsubsection*{Rokhlin's Disintegration Theorem}

Now we present a brief recall about disintegration of measures.

Consider a probability space $(\Sigma,\mathcal{B}, \mu)$ and a partition $%
\Gamma$ of $\Sigma$ by measurable sets $\gamma \in \mathcal{B}$. Denote by $%
\pi : \Sigma \longrightarrow \Gamma$ the projection that associates to each
point $x \in M$ the element $\gamma _x$ of $\Gamma$ which contains $x$, i.e. 
$\pi(x) = \gamma _x$. Let $\widehat{\mathcal{B}}$ be the $\sigma$-algebra of 
$\Gamma$ provided by $\pi$. Precisely, a subset $\mathcal{Q} \subset \Gamma$
is measurable if, and only if, $\pi^{-1}(\mathcal{Q}) \in \mathcal{B}$. We
define the \textit{quotient} measure $\mu _x$ on $\Gamma$ by $\mu _x(%
\mathcal{Q})= \mu(\pi ^{-1}(\mathcal{Q}))$.

The proof of the following theorem can be found in \cite{Kva}, Theorem
5.1.11.

\begin{theorem}
	(Rokhlin's Disintegration Theorem) Suppose that $\Sigma $ is a complete and
	separable metric space, $\Gamma $ is a measurable partition 
	of $\Sigma $ and $\mu $ is a probability on $\Sigma $. Then, $\mu $ admits a
	disintegration relative to $\Gamma $, i.e. a family $\{\mu _{\gamma
	}\}_{\gamma \in \Gamma }$ of probabilities on $\Sigma $ and a quotient
	measure $\mu _{x}$ as above, such that:
	
	\begin{enumerate}
		\item[(a)] $\mu _\gamma (\gamma)=1$ for $\mu _x$-a.e. $\gamma \in \Gamma$;
		
		\item[(b)] for all measurable set $E\subset \Sigma $ the function $\Gamma
		\longrightarrow \mathbb{R}$ defined by $\gamma \longmapsto \mu _{\gamma
		}(E), $ is measurable;
		
		\item[(c)] for all measurable set $E\subset \Sigma $, it holds $\mu (E)=\int 
		{\mu _{\gamma }(E)}d\mu _{x}(\gamma )$.
	\end{enumerate}
	
	\label{rok}
\end{theorem}

The proof of the following lemma can be found in \cite{Kva}, proposition
5.1.7.

\begin{lemma}
	Suppose the $\sigma $-algebra $\mathcal{B}$, on $\Sigma $, has a countable
	generator. If $$(\{\mu _{\gamma }\}_{\gamma \in \Gamma },\mu _{x})$$ and $$%
	(\{\mu _{\gamma }^{\prime }\}_{\gamma \in \Gamma },\mu _{x})$$ are
	disintegrations of the measure $\mu $ relative to $\Gamma $, then $\mu
	_{\gamma }=\mu _{\gamma }^{\prime }$, for $\mu _{x}$-almost every $\gamma
	\in \Gamma $. \label{kv}
\end{lemma}

\subsubsection{The $\mathcal{L}^{1}$ and $S^1$ spaces}

Let $\mathcal{SB}(\Sigma )$ be the space of Borel signed measures on $\Sigma 
$. Given $\mu \in \mathcal{SB}(\Sigma )$ denote by $\mu ^{+}$ and $\mu ^{-}$
the positive and the negative parts of its Jordan decomposition, $\mu =\mu
^{+}-\mu ^{-}$ (see remark {\ref{ghtyhh}). Let $\pi _{x}:\Sigma
	\longrightarrow N_{1}$ \ be the projection defined by $\pi (x,y)=x$, denote
	by $\pi _{x\ast }:$}$\mathcal{SB}(\Sigma )\rightarrow \mathcal{SB}(N_{1})${\
	the pushforward map associated to $\pi _{x}$. Denote by $\mathcal{AB}$ the
	set of signed measures $\mu \in \mathcal{SB}(\Sigma )$ such that its
	associated positive and negative marginal measures, $\pi _{x\ast }\mu ^{+}$
	and $\pi _{x\ast }\mu ^{-},$ are absolutely continuous with respect to the
	volume measure $m_{1}$, i.e. 
	\begin{equation*}
		\mathcal{AB}=\{\mu \in \mathcal{SB}(\Sigma ):\pi _{x\ast }\mu ^{+}<<m_{1}\ \ 
		\mathnormal{and}\ \ \pi _{x\ast }\mu ^{-}<<m_{1}\}.  \label{thespace1}
	\end{equation*}%
}Given a \emph{probability measure} $\mu \in \mathcal{AB}$ on $\Sigma $,
theorem \ref{rok} describes a disintegration $\left( \{\mu _{\gamma
}\}_{\gamma },\mu _{x}\right) $ along $\mathcal{F}^{s}$ (see equation (\ref%
{fol})) by a family $\{\mu _{\gamma }\}_{\gamma }$ of probability measures
on the stable leaves\footnote{%
	In the following to simplify notations, when no confusion is possible we
	will indicate the generic leaf or its coordinate with $\gamma $.} and, since 
$\mu \in \mathcal{AB}$, $\mu _{x}$ can be identified with a non negative
marginal density $\phi _{x}:N_{1}\longrightarrow \mathbb{R}$, defined almost
everywhere, with $|\phi _{x}|_{1}=1$. \ For a general (non normalized)
positive measure $\mu \in \mathcal{AB}$ we can define its disintegration in
the same way. In this case $\mu _{\gamma }$ are still probability measures, $%
\phi _{x}$ is still defined and $\ |\phi _{x}|_{1}=\mu (\Sigma )$.


\begin{definition}
	Let $\pi _{y}:\Sigma \longrightarrow N_{2}$ be the projection defined by $%
	\pi _{y}(x,y)=y$. Let $\gamma \in \mathcal{F}^{s}$, let us consider $\pi
	_{\gamma ,y}:\gamma \longrightarrow N_{2}$, the restriction of the map $\pi
	_{y}:\Sigma \longrightarrow N_{2}$ to the vertical leaf $\gamma $ and the
	associated pushforward map $\pi _{\gamma ,y\ast }$. Given a positive measure 
	$\mu \in \mathcal{AB}$ and its disintegration along the stable leaves $%
	\mathcal{F}^{s}$, $\left( \{\mu _{\gamma }\}_{\gamma },\mu _{x}=\phi
	_{x}m_{1}\right) $, we define the \textbf{restriction of $\mu $ on $\gamma $}
	and denote it by $\mu |_{\gamma }$ as the positive measure on $N_{2}$ (not
	on the leaf $\gamma $) defined, for all mensurable set $A\subset N_{2}$, as 
	\begin{equation*}
		\mu |_{\gamma }(A)=\pi _{\gamma ,y\ast }(\phi _{x}(\gamma )\mu _{\gamma
		})(A).
	\end{equation*}%
	For a given signed measure $\mu \in \mathcal{AB}$ and its Jordan
	decomposition $\mu =\mu ^{+}-\mu ^{-}$, define the \textbf{restriction of $%
		\mu $ on $\gamma $} by%
	\begin{equation*}
		\mu |_{\gamma }=\mu ^{+}|_{\gamma }-\mu ^{-}|_{\gamma }.
	\end{equation*}%
	\label{restrictionmeasure}
\end{definition}

\begin{remark}
	\label{ghtyhh}As we will prove in Corollary \ref{lasttttt}, the restriction $%
	\mu |_{\gamma }$ does not depend on the decomposition. Precisely, if $\mu
	=\mu _{1}-\mu _{2}$, where $\mu _{1}$ and $\mu _{2}$ are any positive
	measures, then $\mu |_{\gamma }=\mu _{1}|_{\gamma }-\mu _{2}|_{\gamma }$ $%
	m_{1}$-a.e. $\gamma \in N_{1}$.
\end{remark}

Let $(X,d)$ be a compact metric space, $g:X\longrightarrow \mathbb{R}$ be a
Lipschitz function and let $L(g)$ be its best Lipschitz constant, i.e. 
\begin{equation}\label{lipsc}
	\displaystyle{L(g)=\sup_{x,y\in X,x\neq y}\left\{ \dfrac{|g(x)-g(y)|}{d(x,y)}%
		\right\} }.
\end{equation}

\begin{definition}
	Given two signed measures $\mu $ and $\nu $ on $X,$ we define a \textbf{\
		Wasserstein-Kantorovich Like} distance between $\mu $ and $\nu $ by 
	\begin{equation*}
		W_{1}^{0}(\mu ,\nu )=\sup_{L(g)\leq 1,|g|_{\infty }\leq 1}\left\vert \int {\
			g}d\mu -\int {g}d\nu \right\vert .
	\end{equation*}%
	\label{wasserstein}
\end{definition}

From now, we denote%
\begin{equation}
	||\mu ||_{W}:=W_{1}^{0}(0,\mu ).  \label{WW}
\end{equation}%
As a matter of fact, $||\cdot ||_{W}$ defines a norm on the vector space of
signed measures defined on a compact metric space. It is worth to remark
that this norm is equivalent to the dual of the Lipschitz norm.

\begin{definition}
	Let $\mathcal{L}^{1}\subseteq \mathcal{AB}$ be defined as%
	\begin{equation*}
		\mathcal{L}^{1}=\left\{ \mu \in \mathcal{AB}:\int_{N_{1}}{W_{1}^{0}(\mu
			^{+}|_{\gamma },\mu ^{-}|_{\gamma })}dm_{1}(\gamma )<\infty \right\}
		\label{L1measurewithsign}
	\end{equation*}%
	and define a norm on it, $||\cdot ||_{1}:\mathcal{L}^{1}\longrightarrow 
	\mathbb{R}$, by%
	\begin{equation*}
		||\mu ||_{1}=\int_{N_{1}}{W_{1}^{0}(\mu ^{+}|_{\gamma },\mu ^{-}|_{\gamma })}%
		dm_{1}(\gamma ).  \label{l1normsm}
	\end{equation*}%
	\label{l1likespace} Here the measurability of the integrand follows by the
	measurability of the disintegration established at Item b) of Theorem \ref%
	{rok}.
\end{definition}

Now, we define the following set of signed measures on $\Sigma $,%
\begin{equation}  \label{S1}
	S^{1}=\left\{ \mu \in \mathcal{L}^{1};\phi _{x}\in S_{\_}\right\}.
\end{equation}
Consider the function $||\cdot ||_{S^{1}}:S^{1}\longrightarrow \mathbb{R}$,
defined by%
\begin{equation*}
	||\mu ||_{S^{1}}=|\phi _{x}|_{s}+||\mu ||_{1},
\end{equation*}%
where we denote $\phi _{x}=\phi _{x}^{+}-\phi _{x}^{-}$ with $\phi
_{x}^{\pm} $ being the marginals of $\mu ^{\pm}$ as explained before.
Moreover, $\phi _{x}$ is the marginal density of the disintegration of $\mu $
and we remark that $\phi _{x}^{+}$ is not necessarily equal to the positive
part of $\phi _{x}$.

The proof of the next proposition is straightforward. Details can be found
in \cite{L}.

\begin{proposition}
	$\left( \mathcal{L}^{1},||\cdot ||_{1}\right) $ and $\left( S^{1},||\cdot
	||_{S^{1}}\right) $ are normed vector spaces. \label{propnorm1}
\end{proposition}

In the following $\left( \mathcal{L}^{1},||\cdot ||_{1}\right) $ and $\left(
S^{1},||\cdot ||_{S^{1}}\right) $ will play the role of a strong and weak
space, for which we will prove a Lasota-Yorke inequality and deduce other
important consequences, as the exponential convergence to equilibrium and
spectral gap for the operator considered on the strong space.

\subsection{$L^{\infty }$ like spaces}

Stronger spaces which can be considered with the above approach can be
defined easily, we show an example of a $L^{\infty }$ like space.

\begin{definition}
	Let $\mathcal{L}^{\infty }\subseteq \mathcal{AB}(\Sigma )$ be defined as%
	\begin{equation*}
		\mathcal{L}^{\infty }=\left\{ \mu \in \mathcal{AB}:\esssup ({W_{1}^{0}(\mu
			^{+}|_{\gamma },\mu ^{-}|_{\gamma }))}<\infty \right\},
	\end{equation*}%
	where the essential supremum is taken over $N_{1}$ with respect to $m_{1}$.
	Define the function $||\cdot ||_{\infty }:\mathcal{L}^{\infty
	}\longrightarrow \mathbb{R}$ by%
	\begin{equation*}
		||\mu ||_{\infty }=\esssup ({W_{1}^{0}(\mu ^{+}|_{\gamma },\mu ^{-}|_{\gamma
			}))}.
	\end{equation*}
\end{definition}

Finally, consider the following set of signed measures on $\Sigma $%
\begin{equation}\label{sinfi}
	S^{\infty }=\left\{ \mu \in \mathcal{L}^{\infty };\phi _{x}\in
	S_{\_}\right\},
\end{equation}%
and the function, $||\cdot ||_{S^{\infty }}:S^{\infty }\longrightarrow 
\mathbb{R}$, defined by%
\begin{equation*}
	||\mu ||_{S^{\infty }}=|\phi _{x}|_{s}+||\mu ||_{\infty }.
\end{equation*}

The proof of the next proposition is straightforward and can be found in 
\cite{L}.

\begin{proposition}
	$\left( \mathcal{L}^{\infty },||\cdot ||_{\infty }\right) $ and $\left(
	S^{\infty },||\cdot||_{S^{\infty }}\right) $ are normed vector spaces.
\end{proposition}


\section{The transfer operator associated to $F$}

In this section we consider the transfer operator associated to skew product
maps as defined in Section 2, acting on our disintegrated measures spaces
defined in Section 3. For such transfer operators and measures we prove a
kind of Perron-Frobenius formula, which is somewhat similar to the one used
for one-dimensional maps.

Consider the pushforward map $\func{F}_{\ast }$ associated with $F$, defined
by 
\begin{equation*}
	\lbrack \func{F}_{\ast }\mu ](E)=\mu (F^{-1}(E)),
\end{equation*}%
for each signed measure $\mu \in \mathcal{SB}(\Sigma )$ and for each
measurable set $E\subset \Sigma $. When $\func{F}_{\ast }$ is considered on
the vector space $\mathcal{SB}(\Sigma )$ or on suitable vector subspaces of
more regular measures, $\func{F}_{\ast }$ is a linear map, beacuse of this
we also call it "transfer operator associated to $F$".

\begin{lemma}
	\label{transformula}For all probability $\mu \in \mathcal{AB}$ disintegrated
	by $(\{\mu _{\gamma }\}_{\gamma },\phi _{x})$, the disintegration $(\{(\func{%
		F}_{\ast }\mu )_{\gamma }\}_{\gamma },(\func{F}_{\ast }\mu )_{x})$ of the
	pushforward $\func{F}_{\ast }\mu $ \ satisfies the following relations%
	\begin{equation}
		(\func{F}_{\ast }\mu )_{x}=\func{P}_{T}(\phi _{x})m_{1}  \label{1}
	\end{equation}
	and
	\begin{equation}
		(\func{F}_{\ast }\mu )_{\gamma }=\nu _{\gamma }:=\frac{1}{\func{P}_{T}(\phi
			_{x})(\gamma )}\sum_{i=1}^{q}{\frac{\phi _{x}}{|\det DT_{i}|}\circ
			T_{i}^{-1}(\gamma )\cdot \chi _{T_{i}(P_{i})}(\gamma )\cdot \func{F}_{\ast
			}\mu _{T_{i}^{-1}(\gamma )}}  \label{2}
	\end{equation}
	when $\func{P}_{T}(\phi _{x})(\gamma )\neq 0$. Otherwise, if $\func{P}%
	_{T}(\phi _{x})(\gamma )=0$, then $\nu _{\gamma }$ is the Lebesgue measure
	on $\gamma $ (the expression $\displaystyle{\frac{\phi _{x}}{|\det DT_{i}|}%
		\circ T_{i}^{-1}(\gamma )\cdot \frac{\chi _{T_{i}(P_{i})}(\gamma )}{\func{P}%
			_{T}(\phi _{x})(\gamma )}\cdot \func{F}_{\ast }\mu _{T_{i}^{-1}(\gamma )}}$
	is understood to be zero outside $T_{i}(P_{i})$ for all $i=1,\cdots ,q$).
	Here and above, $\chi _{A}$ is the characteristic function of the set $A$.
\end{lemma}

\begin{proof}
	By the uniqueness of the disintegration (see Lemma \ref{kv} ) is enough to
	prove the following equation 
	\begin{equation*}
		\func{F}_{\ast }\mu (E)=\int_{N_{1}}{\nu _{\gamma }(E\cap \gamma )}\func{P}%
		_{T}(\phi _{x})(\gamma )dm_1(\gamma) ,
	\end{equation*}%
	for a measurable set $E\subset \Sigma $. For this purpose, let us define the
	sets $B_{1}=\left\{ \gamma \in N_{1};T^{-1}(\gamma )=\emptyset \right\} $, $%
	B_{2}=\left\{ \gamma \in B_{1}^{c};\func{P}_{T}(\phi _{x})(\gamma
	)=0\right\} $ and $B_{3}=\left( B_{1}\cup B_{2}\right) ^{c}$. The following
	properties can be easily proven:
	
	\begin{enumerate}
		\item[1.] $B_i \cap B_j = \emptyset$, $T^{-1}(B_i) \cap T^{-1}(B_j) =
		\emptyset$, for all $1\leq i,j \leq 3$ such that $i \neq j$ and $\bigcup
		_{i=1} ^{3} {B_i} = \bigcup _{i=1} ^{3} {T^{-1}(B_i)} = N_1$;
		
		\item[2.] $m_{1}(T^{-1}(B_{1}))=\phi _{x}m_{1}(T^{-1}(B_{2}))=0$;
	\end{enumerate}
	
	Using the change of variables $\gamma =T_{i}(\beta )$ and the definition of $%
	\nu _{\gamma }$ (see (\ref{2})), we have 
	\begin{equation*}
		\begin{split}
			&\int_{N_{1}}{\nu _{\gamma }(E\cap \gamma )}\func{P}_{T}(\phi _{x})(\gamma
			)dm_1(\gamma) \\
			=&\int_{B_{3}}{\sum_{i=1}^{q}{\ {\frac{\phi _{x}}{|\det DT_{i}|}%
						\circ T_{i}^{-1}(\gamma )\func{F}_{\ast }\mu _{T_{i}^{-1}(\gamma )}(E)\chi
						_{T_{i}(P_{i})(\gamma )}}}}dm_{1}(\gamma ) \\
			=&\sum_{i=1}^{q}{\int_{T_{i}(P_{i})\cap B_{3}}{\ {\frac{\phi _{x}}{|\det
							DT_{i}|}\circ T_{i}^{-1}(\gamma )\func{F}_{\ast }\mu _{T_{i}^{-1}(\gamma
							)}(E)}}}dm_{1}(\gamma ) \\
			=&\sum_{i=1}^{q}{\int_{P_{i}\cap T_{i}^{-1}(B_{3})}{\ {\phi _{x}(\beta )\mu
						_{\beta }(F^{-1}(E))}}}dm_{1}(\beta ) \\
			=&{\int_{T^{-1}(B_{3})}{\ {\phi _{x}(\beta )\mu _{\beta }(F^{-1}(E))}}}%
			dm_{1}(\beta ) \\
			=&\int_{\bigcup_{i=1}^{3}{T^{-1}(B_{i})}}{\ {\ \mu _{\beta }(F^{-1}(E))}}%
			d\phi _{x}m_{1}(\beta ) \\
			=&\int_{N_{1}}{\ {\ \mu _{\beta }(F^{-1}(E))}}d\phi _{x}m_{1}(\beta ) \\
			=&\mu (F^{-1}(E)) \\
			=&\func{F}_{\ast }\mu (E).
		\end{split}
	\end{equation*}%
	And the proof is done.
\end{proof}

As said in Remark \ref{ghtyhh}, Corollary \ref{lasttttt} yields that the
restriction $\mu |_{\gamma }$ does not depend on the decomposition. Thus,
for each $\mu \in \mathcal{L}^{1}$, since $\func{F}^{\ast }\mu $ can be
decomposed as $\func{F}_{\ast }\mu =\func{F}_{\ast }(\mu ^{+})-\func{F}%
_{\ast }(\mu ^{-})$, we can apply the above Lemma to $\func{F}_{\ast }(\mu
^{+})$ and $\func{F}_{\ast }(\mu ^{-})$ to get the following.

\begin{proposition}
	\label{niceformulaab}Let $\gamma \in \mathcal{F}^{s}$ be a stable leaf. Let
	us define the map $F_{\gamma }:N_{2}\longrightarrow N_{2}$ by 
	\begin{equation}\label{ritiruwt}
		F_{\gamma }=\pi _{y}\circ F|_{\gamma }\circ \pi _{\gamma ,y}^{-1}.
	\end{equation}%
	Then, for each $\mu \in \mathcal{L}^{1}$ and for almost all $\gamma \in
	N_{1} $ (interpreted as the quotient space of leaves) it holds 
	\begin{equation}
		(\func{F}_{\ast }\mu )|_{\gamma }=\sum_{i=1}^{q}{\dfrac{\func{F}%
				_{T_{i}^{-1}(\gamma )\ast }\mu |_{T_{i}^{-1}(\gamma )}}{|[\det
				DT_{i}](T_{i}^{-1}(\gamma ))|}\chi _{T_{i}(P_{i})}(\gamma )}\ \ m_{1}%
		\mathnormal{-a.e.}\ \ \gamma \in N_{1}  \label{niceformulaa}
	\end{equation}%
	where $\func{F}_{T_{i}^{-1}(\gamma )\ast }$ is the pushforward map
	associated to $\func{F}_{T_{i}^{-1}(\gamma )}.$
\end{proposition}

\section{Basic properties of the norms and convergence to equilibrium}

In this section, we show important properties of the norms and their
behaviour with respect to the transfer operator. In particular, we prove
that the $\mathcal{L}^{1}$ norm is weakly contracted. We prove Lasota-Yorke
like inequalities for the strong norms and exponential convergence to
equilibrium. All these properties will be used in next section to prove the
spectral gap for the transfer operator associated to the system $F:\Sigma
\rightarrow \Sigma $.

\begin{proposition}[The weak norm is weakly contracted by $\func{F}_{\ast }$]

	\label{l1} If $\mu \in \mathcal{L}^{1}$ then 
	\begin{equation*}
		||\func{F}_{\ast }\mu ||_{1}\leq ||\mu ||_{1}.
	\end{equation*}%
	\label{weakcontral11234}
\end{proposition}

In the proof of the proposition we will use the following lemma about the
behaviour of the $||\cdot ||_W$ norm (see equation (\ref{WW})) which says
that a contraction cannot increase the $||\cdot ||_W$ norm.

\begin{lemma}
	\label{niceformulaac} For every $\mu \in \mathcal{AB}$ and a stable leaf $%
	\gamma \in \mathcal{F}^{s}$, it holds 
	\begin{equation}
		||\func{F}_{\gamma \ast }\mu |_{\gamma }||_{W}\leq ||\mu |_{\gamma }||_{W},
		\label{weak1}
	\end{equation}%
	where $F_{\gamma }:N_{2}\longrightarrow N_{2}$ is defined in Proposition \ref%
	{niceformulaab} and $\func{F}_{\gamma \ast }$ is the associated pushforward
	map. Moreover, if $\mu $ is a probability measure on $N_{2}$, it holds 
	\begin{equation}
		||\func{F}_{\gamma \ast }{^{n}}\mu ||_{W}=||\mu ||_{W}=1,\ \ \forall \ \
		n\geq 1.  \label{simples}
	\end{equation}
\end{lemma}

\begin{proof}
	(of Lemma \ref{niceformulaac}) Indeed, since $F_{\gamma }$ is an $\alpha $%
	-contraction, if $|g|_{\infty }\leq 1$ and $Lip(g)\leq 1$ the same holds for 
	$g\circ F_{\gamma }$. Since 
	\begin{equation*}
		\left\vert \int {g~}d\func{F}_{\gamma \ast }\mu |_{\gamma }\right\vert
		=\left\vert \int {g(F_{\gamma })~}d\mu |_{\gamma }\right\vert ,
	\end{equation*}%
	taking the supremum over $g$ \ such that $|g|_{\infty }\leq 1$ and $%
	Lip(g)\leq 1$ we finish the proof of the inequality (\ref{weak1}).
	
	In order to prove equation (\ref{simples}), consider a probability measure $%
	\mu $ on $N_{2}$ and a Lipschitz function $g:N_{2}\longrightarrow \mathbb{R}$%
	, such that $||g||_{\infty }\leq 1$ we get immediately $|\int {g}d\mu |\leq
	||g||_{\infty }\leq 1$, which yields $||\mu ||_{W}\leq 1$. Considering $%
	g\equiv 1$ we get $||\mu ||_{W}=1$.
\end{proof}

\begin{proof}
	(of Proposition \ref{l1} )
	
	In the following, we consider for all $i$, the change of variable $\gamma
	=T_{i}(\alpha )$. Thus, Lemma \ref{niceformulaac} and equation (\ref%
	{niceformulaa}) yield 
	\begin{eqnarray*}
		||\func{F}_{\ast }\mu ||_{1} &=&\int_{N_{1}}{\ ||(\func{F}_{\ast }\mu
			)|_{\gamma }||_{W}}dm_{1}(\gamma ) \\
		&\leq &\sum_{i=1}^{q}{\int_{T(P_{i})}{\ \left\vert \left\vert \dfrac{\func{F}%
					_{T_{i}^{-1}(\gamma )\ast }\mu |_{T_{i}^{-1}(\gamma )}}{|\det
					DT_{i}(T_{i}^{-1}(\gamma ))|}\right\vert \right\vert _{W}}dm_{1}(\gamma )} \\
		&=&\sum_{i=1}^{q}{\int_{P_{i}}{\left\vert \left\vert \func{F}_{\alpha \ast
				}\mu |_{\alpha }\right\vert \right\vert _{W}}dm_{1}(\alpha )} \\
		&=&\sum_{i=1}^{q}{\int_{P_{i}}{\left\vert \left\vert \mu |_{\alpha
				}\right\vert \right\vert _{W}}}dm_{1}(\alpha ) \\
		&=&||\mu ||_{1}.
	\end{eqnarray*}
\end{proof}

The following proposition shows a regularizing action of the transfer
operator with respect to the strong norm. Such inequalities are usually
called Lasota-Yorke or Doeblin-Fortet inequalities.

\begin{proposition}[Lasota-Yorke inequality for $S^{1}$]
	Let $F:\Sigma \longrightarrow \Sigma $ be a map satisfying T1, T2 and T3.
	Then, there exist $A$, $B_{2}>0$ and $\lambda <1$ such that, for all $\mu
	\in S^{1}$, it holds%
	\begin{equation}
		||\func{F}_{\ast }^{n}\mu ||_{S^{1}}\leq A\lambda ^{n}||\mu
		||_{S^{1}}+B_{2}||\mu ||_{1},\ \ \forall n\geq 1.  \label{xx}
	\end{equation}%
	\label{lasotaoscilation2}
\end{proposition}

\begin{proof}
	
	Firstly, we recall that $\phi _{x}$ is the marginal density of the
	disintegration of $\mu $. Precisely, $\phi _{x}=\phi _{x}^{+}-\phi _{x}^{-}$%
	, where $\phi _{x}^{+}=\dfrac{d\pi _{x}^{\ast }\mu ^{+}}{dm_{1}}$ and $\phi
	_{x}^{-}=\dfrac{d\pi _{x}^{\ast }\mu ^{-}}{dm_{1}}$. By the definition of
	the Wasserstein norm it follows that for every $\gamma $ it holds $||\mu
	|_{\gamma }||_{W}\geq \int 1~d(\mu |_{\gamma })=\phi _{x}(\gamma )$. Thus, 
	$|\phi _{x}|_{1}\leq ||\mu ||_{1}.$ By this last remark, equation (\ref%
	{lasotaiiii}) and Proposition \ref{l1} we have%
	\begin{eqnarray*}
		||\func{F}_{\ast }^{n}\mu ||_{S^{1}} &=&|\func{P}_{T}^{n}\phi _{x}|_{s}+||%
		\func{F}_{\ast }^{n}\mu ||_{1} \\
		&\leq &B_{3}\beta _{2}^{n}|\phi _{x}|_{s}+C_{2}|\phi _{x}|_{1}+||\mu ||_{1}
		\\
		&\leq &B_{3}\beta _{2}^{n}||\mu ||_{S^{1}}+(C_{2}+1)||\mu ||_{1}.
	\end{eqnarray*}%
	We finish the proof by setting $\lambda =\beta _{2}$, $A=B_{3}$ and $%
	B_{2}=C_{2}+1$.
\end{proof}

\subsection{Convergence to equilibrium}

Let $X$ be a compact metric space. Consider the space $\mathcal{SB}(X)$ of
signed Borel measures on $X$. \ In the following we consider two further
normed vectors spaces of signed Borel measures on $X.$ The spaces $%
(B_{s},||~||_{s})\subseteq (B_{w},||~||_{w})\subseteq \mathcal{SB}(X)$ with
norms satisfying%
\begin{equation*}
	||~||_{w}\leq ||~||_{s}.
\end{equation*}%
We say that the a Markov operator $\func{L}:B_{w}\rightarrow B_{w}$ has
convergence to equilibrium with speed at least $\Phi $ and with respect to
the norms $||\cdot ||_{s}$ and $||\cdot ||_{w}$, if for each $\mu \in 
\mathcal{V}_{s}$, where 
\begin{equation}
	\mathcal{V}_{s}=\{\mu \in B_{s},\mu (X)=0\}  \label{vs}
\end{equation}%
is the space of zero-average measures, it holds 
\begin{equation*}
	||\func{L}^{n}(\mu )||_{w}\leq \Phi (n)||\mu ||_{s},  \label{wwe}
\end{equation*}%
where $\Phi (n)\longrightarrow 0$ as $n\longrightarrow \infty $.

In this section, we prove that $F_{\ast }$ has exponential convergence to
equilibrium. This is weaker with respect to the spectral gap. However, the
spectral gap follows from the above Lasota-Yorke inequality and the
convergence to equilibrium. Before the main statements we need some
preliminary lemmata. The following is somewhat similar to Lemma \ref%
{niceformulaac} considering the behaviour of the $||\cdot ||_{W}$ norm after
a contraction. It gives a finer estimate for zero average measures. The
following Lemma is useful to estimate the behaviour of our $W$ norms under
contractions.

\begin{lemma}
	For all signed measures $\mu $ on $N_{2}$ and for all $\gamma \in \mathcal{F}%
	^{s}$, it holds%
	\begin{equation*}
		||\func{F}_{\gamma \ast }\mu ||_{W}\leq \alpha ||\mu ||_{W}+\mu (N_{2})
	\end{equation*}%
	($\alpha $ is the rate of contraction of $G$, see \eqref{contracting1}). In
	particular, if $\mu (N_{2})=0$ then%
	\begin{equation*}
		||\func{F}_{\gamma \ast }\mu ||_{W}\leq \alpha ||\mu ||_{W}.
	\end{equation*}%
	\label{quasicontract}
\end{lemma}

\begin{proof}
	If $Lip(g)\leq 1$ and $||g||_{\infty }\leq 1$, then $g\circ F_{\gamma }$ is $%
	\alpha $-Lipschitz. Moreover, since $||g||_{\infty }\leq 1$, then $||g\circ
	F_{\gamma }-\theta ||_{\infty }\leq \alpha $, for some $\theta $ such that $%
	|\theta |\leq 1$. Indeed, let $z\in N_{2}$ be such that $|g\circ F_{\gamma
	}(z)|\leq 1$, set $\theta =g\circ F_{\gamma }(z)$ and let $d_{2}$ be the
	Riemannian metric of $N_{2}$. Since $\diam(N_{2})=1$, we have 
	\begin{equation*}
		|g\circ F_{\gamma }(y)-\theta |\leq \alpha d_{2}(y,z)\leq \alpha
	\end{equation*}%
	and consequently $||g\circ F_{\gamma }-\theta ||_{\infty }\leq \alpha $.
	
	This implies,
	
	\begin{align*}
		\left\vert \int_{N_{2}}{g}d\func{F}_{\gamma \ast }\mu \right\vert &
		=\left\vert \int_{N_{2}}{g\circ F_{\gamma }}d\mu \right\vert \\
		& \leq \left\vert \int_{N_{2}}{g\circ F_{\gamma }-\theta }d\mu \right\vert
		+\left\vert \int_{N_{2}}{\theta }d\mu \right\vert \\
		& =\alpha \left\vert \int_{N_{2}}{\frac{g\circ F_{\gamma }-\theta }{\alpha }}%
		d\mu \right\vert +|\theta ||\mu (N_{2})|.
	\end{align*}%
	And taking the supremum over $g$ such that $|g|_{\infty }\leq 1$ and $%
	Lip(g)\leq 1$ we have $||\func{F}_{\gamma \ast }\mu ||_{W}\leq \alpha ||\mu
	||_{W}+\mu (N_{2})$. In particular, if $\mu (N_{2})=0$, we get the second
	part.
\end{proof}

Now we are ready to show a key estimate regarding the behaviour of our weak $%
|| \ ||_{1} $ norm in Lorenz-like systems, as defined at beginning of
Section \ref{sec2}.

\begin{proposition}
	\label{5.6} For all signed measure $\mu \in \mathcal{L}^{1}$, it holds 
	\begin{equation}
		||\func{F}_{\ast }\mu ||_{1}\leq \alpha ||\mu ||_{1}+(\alpha +1)|\phi
		_{x}|_{1}.  \label{abovv}
	\end{equation}
\end{proposition}

\begin{proof}
	Consider a signed measure $\mu \in \mathcal{L}^{1}$ and its restriction on
	the leaf $\gamma $, $\mu |_{\gamma }=\pi _{\gamma ,y\ast }(\phi _{x}(\gamma
	)\mu _{\gamma })$. Set%
	\begin{equation*}
		\overline{\mu }|_{\gamma }=\pi _{\gamma ,y\ast }\mu _{\gamma }.
	\end{equation*}%
	If $\mu $ is a positive measure then $\overline{\mu }|_{\gamma }$ is a
	probability on $N_{2}$ and $\mu |_{\gamma }=\phi _{x}(\gamma )\overline{\mu }%
	|_{\gamma }$. Then, the expression given by Proposition \ref{niceformulaab}
	yields%
	\begin{equation*}
		\begin{split}
			&||\func{F}_{\ast }\mu ||_{1} \\
			&\leq \sum_{i=1}^{q}{\ \int_{T(P_{i})}{\
					\left\vert \left\vert \frac{\func{F}_{T_{i}^{-1}(\gamma )\ast }\overline{\mu
							^{+}}|_{T_{i}^{-1}(\gamma )}\phi _{x}^{+}(T_{i}^{-1}(\gamma ))}{|\det
						DT_{i}|\circ T_{i}^{-1}(\gamma )}-\frac{\func{F}_{T_{i}^{-1}(\gamma )\ast }%
						\overline{\mu ^{-}}|_{T_{i}^{-1}(\gamma )}\phi _{x}^{-}(T_{i}^{-1}(\gamma ))%
					}{|\det DT_{i}|\circ T_{i}^{-1}(\gamma )}\right\vert \right\vert _{W}}%
				dm_{1}(\gamma )} \\
			&\leq \func{I}_{1}+\func{I}_{2},
		\end{split}
	\end{equation*}%
	where%
	\begin{equation*}
		\func{I}_{1}=\sum_{i=1}^{q}{\ \int_{T(P_{i})}{\ \left\vert \left\vert \frac{%
					\func{F}_{T_{i}^{-1}(\gamma )\ast }\overline{\mu ^{+}}|_{T_{i}^{-1}(\gamma
						)}\phi _{x}^{+}(T_{i}^{-1}(\gamma ))}{|\det DT_{i}|\circ T_{i}^{-1}(\gamma )}%
				-\frac{\func{F}_{T_{i}^{-1}(\gamma )\ast }\overline{\mu ^{+}}%
					|_{T_{i}^{-1}(\gamma )}\phi _{x}^{-}(T_{i}^{-1}(\gamma ))}{|\det
					DT_{i}|\circ T_{i}^{-1}(\gamma )}\right\vert \right\vert _{W}}dm_{1}(\gamma )%
		}
	\end{equation*}%
	and%
	\begin{equation*}
		\func{I}_{2}=\sum_{i=1}^{q}{\ \int_{T(P_{i})}{\ \left\vert \left\vert \frac{%
					\func{F}_{T_{i}^{-1}(\gamma )\ast }\overline{\mu ^{+}}|_{T_{i}^{-1}(\gamma
						)}\phi _{x}^{-}(T_{i}^{-1}(\gamma ))}{|\det DT_{i}|\circ T_{i}^{-1}(\gamma )}%
				-\frac{\func{F}_{T_{i}^{-1}(\gamma )\ast }\overline{\mu ^{-}}%
					|_{T_{i}^{-1}(\gamma )}\phi _{x}^{-}(T_{i}^{-1}(\gamma ))}{|\det
					DT_{i}|\circ T_{i}^{-1}(\gamma )}\right\vert \right\vert _{W}}dm_{1}(\gamma )%
		}.
	\end{equation*}%
	In the following we estimate $\func{I}_{1}$ and $\func{I}_{2}$. By Lemma \ref%
	{niceformulaac} and a change of variable we have 
	\begin{eqnarray*}
		\func{I}_{1} &=&\sum_{i=1}^{q}{\ \int_{T(P_{i})}{\ \left\vert \left\vert 
				\func{F}_{T_{i}^{-1}(\gamma )\ast }\overline{\mu ^{+}}|_{T_{i}^{-1}(\gamma
					)}\right\vert \right\vert _{W}\frac{|\phi _{x}^{+}-\phi _{x}^{-}|}{|\det
					DT_{i}|}\circ T_{i}^{-1}(\gamma )}dm_{1}(\gamma )} \\
		&\leq &\int_{N_{1}}{\ \left\vert \left\vert \func{F}_{\beta \ast }\overline{%
				\mu ^{+}}|_{\beta }\right\vert \right\vert _{W}|\phi _{x}^{+}-\phi
			_{x}^{-}|(\beta )}dm_{1}(\beta ) \\
		&=&\int_{N_{1}}{\ |\phi _{x}^{+}-\phi _{x}^{-}|(\beta )}dm_{1}(\beta ) \\
		&=&|\phi _{x}|_{1},
	\end{eqnarray*}%
	and by Lemma \ref{quasicontract} we have%
	\begin{eqnarray*}
		\func{I}_{2} &=&\sum_{i=1}^{q}{\ \int_{T(P_{i})}{\ \left\vert \left\vert 
				\func{F}_{T_{i}^{-1}(\gamma )\ast }\left( \overline{\mu ^{+}}%
				|_{T_{i}^{-1}(\gamma )}-\overline{\mu ^{-}}|_{T_{i}^{-1}(\gamma )}\right)
				\right\vert \right\vert _{W}\frac{\phi _{x}^{-}}{|\det DT_{i}|}\circ
				T_{i}^{-1}(\gamma )}dm_{1}(\gamma )} \\
		&\leq &\sum_{i=1}^{q}{\ \int_{P_{i}}{\ \left\vert \left\vert \func{F}_{\beta
					\ast }\left( \overline{\mu ^{+}}|_{\beta }-\overline{\mu ^{-}}|_{\beta
				}\right) \right\vert \right\vert _{W}\phi _{x}^{-}(\beta )}dm_{1}(\beta )} \\
		&\leq &\alpha \int_{N_{1}}{\ \left\vert \left\vert \overline{\mu ^{+}}%
			|_{\beta }-\overline{\mu ^{-}}|_{\beta }\right\vert \right\vert _{W}\phi
			_{x}^{-}(\beta )}dm_{1}(\beta ) \\
		&\leq &\alpha \int_{N_{1}}{\ \left\vert \left\vert \overline{\mu ^{+}}%
			|_{\beta }\phi _{x}^{-}(\beta )-\overline{\mu ^{-}}|_{\beta }\phi
			_{x}^{-}(\beta )\right\vert \right\vert _{W}}dm_{1}(\beta ) \\
		&\leq &\alpha \int_{N_{1}}{\ \left\vert \left\vert \overline{\mu ^{+}}%
			|_{\beta }\phi _{x}^{-}(\beta )-\overline{\mu ^{+}}|_{\beta }\phi
			_{x}^{+}(\beta )\right\vert \right\vert _{W}}dm_{1}(\beta )\\&+&\alpha
		\int_{N_{1}}{\ \left\vert \left\vert \overline{\mu ^{+}}|_{\beta }\phi
			_{x}^{+}(\beta )-\overline{\mu ^{-}}|_{\beta }\phi _{x}^{-}(\beta
			)\right\vert \right\vert _{W}}dm_{1}(\beta ) \\
		&=&\alpha |\phi _{x}|_{1}+\alpha ||\mu ||_{1}.
	\end{eqnarray*}%
	Summing the above estimates we finish the proof.
\end{proof}

Iterating (\ref{abovv}) we get the following corollary.

\begin{corollary}
	For all signed measure $\mu \in \mathcal{L}^{1}$ it holds 
	\begin{equation*}
		||\func{F}_{\ast }^{n}\mu ||_{1}\leq \alpha ^{n}||\mu ||_{1}+\overline{%
			\alpha }|\phi _{x}|_{1},
	\end{equation*}%
	where $\overline{\alpha }=\frac{1+\alpha }{1-\alpha }$. \label{nicecoro}
\end{corollary}

Let us consider the set of zero average measures in $S^{1}$ defined by 
\begin{equation}
	\mathcal{V}_{s}=\{\mu \in S^{1}:\mu (\Sigma )=0\}.  \label{mathV}
\end{equation}%
Note that, for all $\mu \in \mathcal{V}_{s}$ we have $\pi _{x\ast }\mu
(N_{1})=0$. Moreover, since $\pi _{x\ast }\mu =\phi _{x}m_{1}$ ($\phi
_{x}=\phi _{x}^{+}-\phi _{x}^{-}$), we have $\displaystyle{\int_{N_{1}}{\phi
		_{x}}dm_{1}=0}$. This allows us to apply Theorem \ref{loiub} in the proof of
the next proposition.

\begin{proposition}[Exponential convergence to equilibrium]
	\label{5.8} There exist $D_{2}\in \mathbb{R}$ and $0<\beta _{1}<1$ such that
	for every signed measure $\mu \in \mathcal{V}_{s}$, it holds 
	\begin{equation*}
		||\func{F}_{\ast }^{n}\mu ||_{1}\leq D_{2}\beta _{1}^{n}||\mu ||_{S^{1}},
	\end{equation*}%
	for all $n\geq 1$. \label{quasiquasiquasi}
\end{proposition}

\begin{proof}
	Given $\mu \in \mathcal{V}_s$ and denoting $\phi _{x}=\phi _{x}^{+}-\phi
	_{x}^{-}$, it holds that $\int {\phi }_{x}dm _1=0$. Moreover, Theorem \ref%
	{loiub} yields $|\func{P}_{T}^{n}(\phi _{x})|_{s}\leq Dr^{n}|\phi _{x}|_{s}$
	for all $n\geq 1$, then $|\func{P}_{T}^{n}(\phi _{x})|_{s}\leq Dr^{n}||\mu
	||_{S^{1}}$ for all $n\geq 1$.
	
	Let $l$ and $0\leq d\leq 1$ be the coefficients of the division of $n$ by $2$%
	, i.e. $n=2l+d$. Thus, $l=\frac{n-d}{2}$ (by Proposition \ref%
	{weakcontral11234}, we have $||\func{F}_{\ast }^{n}\mu ||_{1}\leq ||\mu
	||_{1}$, for all $n$, and $||\mu ||_{1}\leq ||\mu ||_{S^{1}}$) and by
	Corollary \ref{nicecoro}, it holds (below, set $\beta _{1}=\max \{\sqrt{r},%
	\sqrt{\alpha }\}$)
	
	\begin{eqnarray*}
		||\func{F}_{\ast }^{n}\mu ||_{1} &\leq &||\func{F}_{\ast }^{2l+d}\mu ||_{1}
		\\
		&\leq &\alpha ^{l}||\func{F}_{\ast }^{l+d}\mu ||_{1}+\overline{\alpha }%
		\left\vert \dfrac{d(\pi _{x\ast }(\func{F}^{\ast l+d}\mu ))}{dm_{1}}%
		\right\vert _{1} \\
		&\leq &\alpha ^{l}||\mu ||_{1}+\overline{\alpha }|\func{P}_{T}^{l}(\phi
		_{x})|_{1} \\
		&\leq &(1+\overline{\alpha }D)\beta _{1}^{-d}\beta _{1}^{n}||\mu ||_{S^{1}}
		\\
		&\leq &D_{2}\beta _{1}^{n}||\mu ||_{S^{1}},
	\end{eqnarray*}%
	where $D_{2}=\dfrac{1+\overline{\alpha }D}{\beta _{1}}$.
\end{proof}

\begin{remark}
	\label{quantitative} The rate of convergence to equilibrium, $\beta _{1}$,
	for the map $F$ found above, is directly related to the rate of contraction, 
	$\alpha $, of the stable foliation, and to the rate of convergence to
	equilibrium, $r$, of the induced basis map $T$ (see equation \ref{p2}). More
	precisely, $\beta _{1}=\max \{\sqrt{\alpha },\sqrt{r}\}$. Similarly, we have
	an explicit estimate for the constant $D_{2}$, provided we have an estimate
	for $D$ in the basis map\footnote{%
		It can be difficult to find a sharp estimate for $D$. An approach allowing
		to find some useful upper estimates is shown in \cite{GNS}.}.
\end{remark}

Now we show that under the assumptions taken, the system has a unique
invariant measure $\mu _{0}\in S^{1}$.

\begin{lemma}
	\label{existinv}A contracting fiber map $(N_{1}\times N_{2},F)$ satisfying
	assumptions $G1,T1,...,T3.4$ has a unique invariant measure in $S^{1}$.
\end{lemma}

Before the proof of Lemma \ref{existinv} we need a preliminary lemma.

\begin{lemma}
	\label{compact}Let $\mu _{n}$ be a sequence of probability measures which is
	a Cauchy sequence for the Wassertein like norm $||~||_{W}$ on a compact
	manifold $N$. Then this sequence has a limit in the space of probability
	measures $\mathcal{PB}(N)$ on $N.$ In other words $\mathcal{PB}(N)$ is a
	complete metric space with the distance induced by $||~||_{W}.$
\end{lemma}

\begin{proof}
	Consider $\mathcal{PB}(N)$ with the weak* topology, i.e. the topology in
	which $\mu _{n}\rightarrow \mu $ if and only if for each continuous $%
	f:N\rightarrow \mathbb{R}$ it holds $\int f~d\mu _{n}\rightarrow \int f~d\mu 
	$. This space is compact. Then $\mu _{n}$ has subsequences $\mu _{n_{k}}$
	converging to some $\mu _{0}\in \mathcal{PB}(N)$ in this topology. Since $N$
	is compact we can approximate uniformly every contunuous function $f$ with
	Lipschitz functions $g_{i}$. Given $f\in C^{0}(N),\epsilon >0$ let us choose 
	$g_{i}$ such that $||f-g_{i}||_{\infty }\leq \epsilon $ we have%
	\begin{equation*}
		|\int fd(\mu _{n}-\mu _{m})|\leq |\int (f-g_{i})d(\mu _{n}-\mu _{m})|+|\int
		g_{i}d(\mu _{n}-\mu _{m})|\leq \epsilon +o(m,n)
	\end{equation*}%
	with $o(m,n)\rightarrow 0$ \ as $\min (m,n)\rightarrow \infty $ hence $|\int
	fd(\mu _{n}-\mu _{m})|\leq 2\epsilon $ as $\min (m,n)\rightarrow \infty $.
	Since $\epsilon $ is arbitrary we get $|\int fd(\mu _{n}-\mu
	_{m})|\rightarrow 0$ as $\min (m,n)\rightarrow \infty .$ This shows that $%
	\mu _{n}$ is a Cauchy sequence in the weak* topology, and then it converges
	to $\mu _{0}$ in that topology. Now conversely, suppose that this
	convergence was not in the $||~||_{W}$ norm, there is a subsequence $\mu
	_{n_{k}}$ \ such that $\forall k$ $||\mu _{n_{k}}-\mu _{0}||_{W}\geq
	\epsilon ,$ for some $\epsilon <0$. Then it means there are uniformly
	bounded, $1$-Lipschitz functions $g_{i}$ such that for each $i,$ we have 
	\begin{equation}
		\int g_{i}~d[\mu _{n_{i}}-\mu _{0}]\geq \frac{\epsilon }{2}.  \label{AscArz}
	\end{equation}%
	By Ascoli-Arzel\`{a} theorem however a subsequence $g_{i_{j}}$converges
	uniformly to some continuous function $\hat{g}$, for which $\int g_{i}~d[\mu
	_{n_{i}}-\mu _{0}]\rightarrow 0$, contraddicting $($\ref{AscArz}$)$. Then $%
	||\mu _{n_{k}}-\mu _{0}||_{W}\rightarrow 0$, proving the stetement.
\end{proof}

\begin{proof}[Proof of Lemma \protect\ref{existinv}]
	By assumption $T3.4$, the base map $T$ \ has a unique invariant measure $%
	\psi _{x}\in $ $S_{\_}\subseteq L^{1}$. Let us consider the following set of
	measures having $\psi _{x}$ as a marginal: 
	\begin{equation*}
		M_{\psi }=\{\mu \in S^{1},\pi _{x\ast }(\mu )=\psi _{x}\}.
	\end{equation*}%
	By Proposition \ref{5.8} $F_{\ast }$ is a contraction on $M_{\psi }$, thus if
	we prove that there is a fixed point in $M_{\psi }$ this is unique. Let us
	consider the measure $\nu _{0}:=\psi _{x}\times m_{2}\in S^{1}$ and let us
	iterate this by $F.$ Every iterate $\nu _{n}:=F_{\ast }^{n}(\nu )$ is a
	positive measure and because of Corollary \ref{nicecoro} of $\nu _{n}\in
	S^{1}$. Furthermore, for each $n$, $\pi _{x\ast }(F_{\ast }^{n}(\nu ))=\psi
	_{x}.$ By Proposition \ref{5.8} $\nu _{n}$ is a Cauchy sequence in $M_{\psi }
	$, for the $||~||_{1}$ norm. Let us consider the completion $\overline{%
		M_{\psi }}$ of $M_{\psi }$. Being a contraction $F_{\ast }$ can be extended
	continuously to $\overline{M_{\psi }}$. Let $\mu _{0}$ be hence the limit of 
	$\nu _{n}$ in $\overline{M_{\psi }}$. $\mu _{0}$ is then a fixed point of
	the contraction $F_{\ast }$. We now prove that $\mu _{0}$ is a Borel
	probability measure.
	
	Let us consider the set of Borel probability measures $\mathcal{PB}%
	(N_{1}\times N_{2})$ equipped with the Wassertein distance $d_{W}$ defined
	by $d_{W}(\mu ,\nu )=\underset{Lip(g)\leq 1}{\sup }|\mu (g)-\nu (g)|.$ $%
	M_{\psi }$ \ is a closed subset of $\mathcal{PB}(N_{1}\times N_{2})$ for
	this topology. Indeed for each $\mu \in \mathcal{PB}(N_{1}\times N_{2}),$
	the projection $\pi _{x\ast }(\mu )\in \mathcal{PB}(N_{1})$ can be also
	characterized by its action on suitable Lipschitz observables: let $f\in
	Lip(N_{1})$, consider $\hat{f}\in Lip(N_{1}\times N_{2})$ be defined by $%
	\hat{f}(x,y)=f(x).$ The projection $\pi _{x\ast }(\mu )$ can be also defined
	by the measure on $N_{1}$ for which%
	\begin{equation*}
		\int_{N_{1}}f~d\pi _{x\ast }(\mu )=\int_{N_{1}\times N_{2}}\hat{f}d\mu .
	\end{equation*}
	
	If $\mu _{n}\rightarrow \mu $ in the $d_{W}$ topology and $\mu _{n}\in
	M_{\psi }$ for such a function $\hat{f}$ we have $\int_{N_{1}\times N_{2}}%
	\hat{f}d\nu _{n}\rightarrow \int_{N_{1}\times N_{2}}\hat{f}d\mu $ \ this
	shows that $\pi _{x\ast }(\mu )=\psi _{x}$.
	
	Furthermore, we have that if $\mu ,\nu \in M_{\psi }$ it holds $d_{W}(\mu
	,\nu )\leq ||\mu -\nu ||_{1}$. Indeed for every $g$ such that $Lip(g)\leq 1$%
	, disintegrating the two measures on the stable foliation it holds 
	\begin{equation*}
		\int gd[\mu -\nu ]=\int_{\gamma \in N_{1}}~\int_{N_{2}}g(\gamma ,\cdot
		)d[\mu _{\gamma }-\nu _{\gamma }]~d\psi _{x}.
	\end{equation*}%
	For every $\gamma $ $\ \ g(\gamma ,\cdot )$ is 1-Lipschitz on the stable
	leaf. Hence 
	\begin{equation*}
		\int gd[\mu -\nu ]\leq \int_{\gamma \in N_{1}}~||\mu _{\gamma }-\nu _{\gamma
		}||_{W}~d\psi _{x}=||\mu -\nu ||_{1}.
	\end{equation*}
	
	By this a Cauchy sequence for the $||~||_{1}$ norm \ is also a Cauchy
	sequence for $d_{W}(\mu ,\nu )$. By Lemma \ref{compact} we have that $\nu
	_{n}$ has a limit in $\mathcal{PB}(N_{1}\times N_{2})$ in the $d_{W}$
	topology. Since $M_{\psi }$ is closed in this topology, we get $\mu _{0}\in
	M_{\psi }\subseteq S^{1}.$ Since this invariant measure is the fixed point
	of a contraction, it is unique.
\end{proof}

Another construction \ to show the existence of an invariant measure in the
context of fiber contracting maps can be found in \cite{AP} (subsection
7.3.4.1). If the system satisfies the assumption \textbf{N1} we can also
prove a stronger statement

\begin{proposition}
	\label{belongs} If \textbf{N1} is satisfied, $\mu _{0}$ is the unique $F$%
	-invariant probability in $S^{\infty }$. \footnote{See  $($\ref{sinfi}$)$ for the definition of the space.}
\end{proposition}

\begin{proof}
	Let $\mu _{0}$ be the $F$-invariant measure found in Lemma \ref{existinv}
	such that $\pi _{x\ast }\mu _{0}=\psi _{x}$ where $\psi _{x}$ is the unique $%
	T$-invariant density (see T3.4) in $S_{-}$. If \textbf{N1} is satisfied, we
	have $|\cdot |_{\infty }\leq |\cdot |_{s}$. Suppose that $%
	g:N_{2}\longrightarrow \mathbb{R}$ is a Lipschitz function such that $%
	|g|_{\infty }\leq 1$ and $L(g)\leq 1$. Then, it holds $\left\vert \int {g}%
	d(\mu _{0}|_{\gamma })\right\vert \leq |g|_{\infty }\psi _{x}(\gamma )\leq
	|\psi _{x}|_{\infty }\leq |\psi _{x}|_{s}$. Hence, $\mu _{0}\in S^{\infty }$.
\end{proof}

\subsection{$L^{\infty }$ norms}

In this section we consider an $L^{\infty }$ like anisotropic norm. We show
how a Lasota Yorke inequality can be proved for this norm too.

\begin{lemma}
	Under the assumptions $G1$, $T1,...,T3.3$, for all signed measure $\mu \in
	S^{\infty }$ with marginal density $\phi _{x}$ it holds 
	\begin{equation*}
		||\func{F}_{\ast }\mu ||_{\infty }\leq \alpha |\func{P}_{T}1|_{\infty }||\mu
		||_{\infty }+|\func{P}_{T}\phi _{x}|_{\infty }.
	\end{equation*}
\end{lemma}

\begin{proof}
	Let $T_{i}$ be the branches of $T$, for all $i=1\cdots q$. Applying Lemma %
	\ref{quasicontract} on the third line below, we have 
	\begin{eqnarray*}
		||(\func{F}_{\ast }\mu )|_{\gamma }||_{W} &=&\left\vert \left\vert
		\sum_{i=1}^{q}\frac{\func{F}_{T_{i}^{-1}(\gamma )\ast }\mu
			|_{T_{i}^{-1}(\gamma )}}{|\det DT_{i}(T_{i}^{-1}(\gamma ))|}\chi
		_{T(P_{i})}(\gamma )\right\vert \right\vert _{W} \\
		&\leq &\sum_{i=1}^{q}\frac{||\func{F}_{T_{i}^{-1}(\gamma )\ast }\mu
			|_{T_{i}^{-1}(\gamma )}||_{W}}{|\det DT_{i}(T_{i}^{-1}(\gamma ))|}\chi
		_{T(P_{i})}(\gamma ) \\
		&\leq &\sum_{i=1}^{q}\frac{\alpha ||\mu |_{T_{i}^{-1}(\gamma )}||_{W}+\phi
			_{x}(T_{i}^{-1}(\gamma ))}{|\det DT_{i}(T_{i}^{-1}(\gamma ))|}\chi
		_{T(P_{i})}(\gamma ) \\
		&\leq &\alpha ||\mu ||_{\infty }\sum_{i=1}^{q}\frac{\chi _{T(P_{i})}(\gamma )%
		}{|\det DT_{i}(T_{i}^{-1}(\gamma ))|}+\sum_{i=1}^{q}\frac{\phi
			_{x}(T_{i}^{-1}(\gamma ))}{|\det DT_{i}(T_{i}^{-1}(\gamma ))|}\chi
		_{T(P_{i})}(\gamma ).
	\end{eqnarray*}%
	Hence, taking the supremum on $\gamma $, we finish the proof of the
	statement.
\end{proof}

Applying the last lemma to $\func{F}^{\ast n}$ instead of $\func{F}$ one
obtains.

\begin{lemma}
	\label{inftycontraction}Under the assumptions $G1$, $T1,...,T3.4$, for all
	signed measure $\mu \in S^{\infty }$ it holds%
	\begin{equation*}
		||\func{F}_{\ast }^{n}\mu ||_{\infty }\leq \alpha ^{n}|\func{P}%
		_{T}^{n}1|_{\infty }||\mu ||_{\infty }+|\func{P}_{T}^{n}\phi _{x}|_{\infty },
	\end{equation*}%
	where $\phi _{x}$ is the marginal density of $\mu $.
\end{lemma}

\begin{proposition}[Lasota-Yorke inequality for $S^{\infty }$]
	Suppose $F$ satisfies the assumptions $G1$, $T1,...,T3.4$ and $N1$. Then,
	there are $0<\alpha _{1}<1$ and $A_{1},B_{4}\in \mathbb{R}$ such that for
	all $\mu \in S^{\infty }$, it holds%
	\begin{equation*}
		||\func{F}_{\ast }^{n}\mu ||_{S^{\infty }}\leq A_{1}\alpha _{1}^{n}||\mu
		||_{S^{\infty }}+B_{4}||\mu ||_{1}.
	\end{equation*}%
	\label{LYinfty}
\end{proposition}

\begin{proof}
	By equation (\ref{lasotaiiii}) and (N1) it follows $|\func{P}%
	_{T}^{n}1|_{\infty }\leq H_{N}(B_{3}+C_{2}),$ for each $n$. Then, 
	\begin{eqnarray*}
		||\func{F}_{\ast }^{n}\mu ||_{S^{\infty }} &=&|\func{P}_{T}^{n}\phi
		_{x}|_{s}+||\func{F}\mu ||_{\infty } \\
		&\leq &[B_{3}\beta _{2}^{n}|\phi _{x}|_{s}+C_{2}|\phi _{x}|_{1}]+[\alpha
		^{n}|\func{P}_{T}^{n}1|_{\infty }||\mu ||_{\infty }+|\func{P}_{T}^{n}\phi
		_{x}|_{\infty }] \\
		&\leq &[B_{3}\beta _{2}^{n}|\phi _{x}|_{s}+C_{2}|\phi _{x}|_{1}] \\
		&+&[\alpha ^{n}H_{N}(B_{3}+C_{2})||\mu ||_{\infty }+H_{N}(B_{3}\beta
		_{2}^{n}|\phi _{x}|_{s}+C_{2}|\phi _{x}|_{1})]. \\
		&\leq &[\max (\alpha ,\beta _{2})]^{n}[B_{3}(1+2H_{N})+H_{N}C_{2}]||\mu
		||_{S^{\infty }}+C_{2}(1+H_{N})||\mu ||_{1},
	\end{eqnarray*}%
	where $|\phi _{x}|_{1}\leq ||\mu ||_{1}$ and $|\phi _{x}|_{s}\leq ||\mu
	||_{S^{\infty }}$. We finish the proof, setting $\alpha _{1}=\max (\alpha
	,\beta _{2})$, $A_{1}=[B_{3}(1+2H_{N})+H_{N}C_{2}]$ and $%
	B_{4}=C_{2}(1+H_{N}) $.
\end{proof}

\section{Spectral gap}

In this section, we prove a spectral gap statement for the transfer operator
applied to our strong spaces. For this, we will directly use the properties
proved in the previous section, and this will give a kind of constructive
proof. We remark that, we cannot apply the traditional Hennion, or
Ionescu-Tulcea and Marinescu's approach to our function spaces because there
is no compact immersion of the strong space into the weak one. This comes
from the fact that we are considering the same \textquotedblleft dual of
Lipschitz\textquotedblright distance (see Definition \ref{wasserstein}) in
the contracting direction for both spaces.

\begin{theorem}[Spectral gap on $S^{1}$]
	\label{spgap}If $F$ satisfies \textbf{G1}, \textbf{T1},...,\textbf{T3.4}
	given at beginning of section \ref{sec2}, then the operator $\func{F}_{\ast
	}:S^{1}\longrightarrow S^{1}$ (see \eqref{S1}) can be written as 
	\begin{equation*}
		\func{F}_{\ast }=\func{P}+\func{N},
	\end{equation*}%
	where
	
	\begin{enumerate}
		\item[a)] $\func{P}$ is a projection i.e. $\func{P} ^2 = \func{P}$ and $\dim
		Im (\func{P})=1$;
		
		\item[b)] there are $0<\xi <1$ and $K>0$ such that \footnote{%
			We remark that, the spectral radius of $\func{\overline{N}}$ satisfies $\rho
			(\func{\overline{N}})<1$, where $\func{\overline{N}}$ is the extension of $%
			\func{N}$ to $\overline{S^1}$ (the completion of $S_1$). This gives us
			spectral gap, in the usual sense, for the operator $\func{\overline{F}}:%
			\overline{S_1} \longrightarrow \overline{S_1}$. The same remark holds for
			Theorem \ref{speclinf}.} $\forall \mu \in S^1$ 
		\begin{equation*}
			||\func{N}^{n}(\mu )||_{S^{1}}\leq ||\mu||_{S^{1}} \xi ^{n}K;
		\end{equation*}
		
		\item[c)] $\func{P}\func{N}=\func{N}\func{P}=0$.
	\end{enumerate}
\end{theorem}

\begin{proof}
	First, let us show there exist $0<\xi <1$ and $K_{1}>0$ such that, for all $%
	n\geq 1$, it holds 
	\begin{equation*}
		||\func{F}_{\ast }^{n}||_{{\mathcal{V}_{s}}\rightarrow {\mathcal{V}_{s}}%
		}\leq \xi ^{n}K_{1}  \label{quaselawww}
	\end{equation*}%
	where ${\mathcal{V}_{s}}$ is the zero average space defined in $($\ref{mathV}%
	$)$. Indeed, consider $\mu \in \mathcal{V}_{s}$ (see \eqref{mathV}) s.t. $%
	||\mu ||_{S^{1}}\leq 1$ and for a given $n\in \mathbb{N}$ let $m$ and $0\leq
	d\leq 1$ be the coefficients of the division of $n$ by $2$, i.e. $n=2m+d$.
	Thus $m=\frac{n-d}{2}$. By the Lasota-Yorke inequality (Proposition \ref%
	{lasotaoscilation2}) we have the uniform bound $||\func{F}_{\ast }^{n}\mu
	||_{S^{1}}\leq B_{2}+A$ for all $n\geq 1$. Moreover, by Propositions \ref%
	{quasiquasiquasi} and \ref{weakcontral11234} there is some $D_{2}$ such that
	it holds (below, let $\lambda _{0}$ be defined by $\lambda _{0}=\max \{\beta
	_{1},\lambda \}$)%
	\begin{eqnarray*}
		||\func{F}_{\ast }^{n}\mu ||_{S^{1}} &\leq &A\lambda ^{m}||\func{F}_{\ast
		}^{m+d}\mu ||_{S^{1}}+B_{2}||\func{F}_{\ast }^{m+d}\mu ||_{1} \\
		&\leq &\lambda ^{m}A(A+B_{2})+B_{2}||\func{F}_{\ast }^{m}\mu ||_{1} \\
		&\leq &\lambda ^{m}A(A+B_{2})+B_{2}D_{2}\beta _{1}^{m} \\
		&\leq &\lambda _{0}^{m}\left[ A(A+B_{2})+B_{2}D_{2}\right] \\
		&\leq &\lambda _{0}^{\frac{n-d}{2}}\left[ A(A+B_{2})+B_{2}D_{2}\right] \\
		&\leq &\left( \sqrt{\lambda _{0}}\right) ^{n}\left( \frac{1}{\lambda _{0}}%
		\right) ^{\frac{d}{2}}\left[ A(A+B_{2})+B_{2}D_{2}\right] \\
		&=&\xi ^{n}K_{1},
	\end{eqnarray*}%
	where $\xi =\sqrt{\lambda _{0}}$ and $K_{1}=\left( \frac{1}{\lambda _{0}}%
	\right) ^{\frac{1}{2}}\left[ A(A+B_{2})+B_{2}D_{2}\right] $. Thus, we arrive
	at 
	\begin{equation}
		||(\func{F}_{\ast }|_{_{\mathcal{V}_{s}}}){^{n}}||_{S^{1}\rightarrow
			S^{1}}\leq \xi ^{n}K_{1}.  \label{just}
	\end{equation}
	
	Now, recall that $\func{F}_{\ast }:S^{1}\longrightarrow S^{1}$ has an unique
	fixed point $\mu _{0}\in S^{1}$, which is a probability (see Proposition \ref%
	{belongs}). Consider the operator $\func{P}:S^{1}\longrightarrow \left[ \mu
	_{0}\right] $ ($\left[ \mu _{0}\right] $ is the space spanned by $\mu _{0}$%
	), defined by $\func{P}(\mu )=\mu (\Sigma )\mu _{0}$. By definition, $\func{P%
	}$ is a projection and $\dim Im(P)=1$. Define the operator%
	\begin{equation*}
		\func{S}:S^{1}\longrightarrow \mathcal{V}_{s},
	\end{equation*}%
	by%
	\begin{equation*}
		\func{S}(\mu )=\mu -\func{P}(\mu ),\ \ \ \mathnormal{\forall }\ \ \mu \in
		S^{1}.
	\end{equation*}%
	Thus, we set $\func{N}=\func{F}_{\ast }\circ \func{S}$ and observe that, by
	definition, $\func{P}\func{N}=\func{N}\func{P}=0$ and $\func{F}_{\ast }=%
	\func{P}+\func{N}$. Moreover, $\func{N}^{n}(\mu )=\func{F}_{\ast }{^{n}}(%
	\func{S}(\mu ))$ for all $n\geq 1$. Since $\func{S}$ is bounded and $\func{S}%
	(\mu )\in \mathcal{V}_{s}$, we get by (\ref{just}), $||\func{N}^{n}(\mu
	)||_{S^{1}}\leq \xi ^{n}K||\mu ||_{S^{1}}$, for all $n\geq 1$, where $%
	K=K_{1}||\func{S}||_{S^{1}\rightarrow S^{1}}$.
\end{proof}

In the same way, using the $\mathcal{L}^{\infty }$ Lasota-Yorke inequality
of Proposition \ref{LYinfty}, and Lemma \ref{inftycontraction} it is
possible to obtain exponential convergence to equilibrium (see the proof of
Proposition \ref{5.8}) and spectral gap on the $L^{\infty }$ like strong and
weak spaces $\left( \mathcal{L}^{\infty },||\cdot ||_{\infty }\right) $ and $%
\left( S^{\infty },||\cdot ||_{S^{\infty }}\right) .$ We omit the proof
which is essentially the same as above:

\begin{theorem}[Spectral gap on $S^{\infty }$]
	If $F$ satisfies the assumptions $G1$, $T1,...,T3.4$ and $N1$, then the
	operator $\func{F}_{\ast }:S^{\infty }\longrightarrow S^{\infty }$ can be
	written as 
	\begin{equation*}
		\func{F}_{\ast }=\func{P}+\func{N},
	\end{equation*}%
	where
	
	\begin{enumerate}
		\item[a)] $\func{P}$ is a projection i.e. $\func{P} ^2 = \func{P}$ and $\dim
		Im (\func{P})=1$;
		
		\item[b)] there are $0<\xi _1 <1$ and $K_2>0$ such that $||\func{N}^{n}(\mu
		)||_{S^{\infty }}\leq ||\mu ||_{S^{\infty }}\xi _1 ^{n}K_2$ $\forall \ \ \mu
		\in S^{\infty} $;
		
		\item[c)] $\func{P}\func{N}=\func{N}\func{P}=0$.
	\end{enumerate}
	
	\label{speclinf}
\end{theorem}

\begin{remark}
	\label{quantitative2} The constant $\xi $ for the map $F$, found in Theorem %
	\ref{spgap}, is directly related to the coefficients of the Lasota-Yorke
	inequality and the rate of convergence to equilibrium of $F$ found before
	(see Remark \ref{quantitative}). More precisely, $\xi =\max \{\sqrt{\lambda }%
	,\sqrt{\beta _{1}}\}$. We remark that, from the above proof we also have an
	explicit estimate for $K$ in the exponential convergence, while many
	classical approaches are not suitable for this.
\end{remark}


\subsection{Exponential Decay of Correlations}\label{decay} In this section, we present one of the standard consequences of spectral gap. We will show how Theorems \ref{spgap} and \ref{speclinf} implies an exponential rate of convergence for the limit $$\lim {C_n(f,g)}=0,$$where $$C_n(f,g):=\left| \int{(g \circ F^n )  f}d\mu_0 - \int{g  }d\mu_0 \int{f  }d\mu_0 \right|,$$ $g: \Sigma \longrightarrow \mathbb{R} $ is a Lipschitz function and $f \in \Theta _{\mu _0} ^1$ or $f \in \Theta _{\mu _0} ^\infty$. The sets $\Theta _{\mu _0} ^1$ and $\Theta _{\mu _0} ^\infty$ are defined as $$\Theta _{\mu _0} ^1:= \{ f: \Sigma \longrightarrow \mathbb{R}; f\mu_0 \in S^1\}$$ and $$\Theta _{\mu _0} ^\infty := \{ f: \Sigma \longrightarrow \mathbb{R}; f\mu_0 \in S^\infty \},$$ where the measure $f\mu_0$ is defined by $f\mu_0(E):=\int _E{f}d\mu_0$ for all measurable set $E$.

\begin{proposition}
	For all Lipschitz function $g: \Sigma \longrightarrow \mathbb{R} $ and all $f \in \Theta _{\mu _0} ^1$, it holds $$\left| \int{(g \circ F^n )  f}d\mu_0 - \int{g  }d\mu_0 \int{f  }d\mu_0 \right| \leq ||f \mu _0||_{S^{1}} K |g|_{\lip}  \xi ^{n} \ \ \forall n \geq 1,$$where $\xi$ and $K$ are from Theorem \ref{spgap} and $|g|_{\lip} := |g|_\infty + L(g)$. 
\end{proposition}

\begin{proof}
	
	Let $g: \Sigma \longrightarrow \mathbb{R} $ be a Lipschitz function and $f \in \Theta _{\mu _0} ^1$. By Theorem \ref{spgap}, we have
	
	\begin{eqnarray*}
		\left| \int{(g \circ F^n )  f}d\mu_0 - \int{g  }d\mu_0 \int{f  }d\mu_0 \right| 
		&=& \left| \int{g  }d \func{F^*}{^n} (f\mu_0) - \int{g  }d\func{P}(f\mu_0) \right|
		\\&\leq& \left|\left|  \func{F^*}{^n} (f\mu_0) - \func{P}(f\mu_0) \right|\right|_W \max\{L(g), ||g||_\infty\}
		\\&=& \left|\left|  \func{N}{^n}(f\mu_0) \right|\right|_W \max\{L(g), ||g||_\infty\}
		\\&\leq & \left|\left|  \func{N}{^n}(f\mu_0) \right|\right|_{S^1} \max\{L(g), ||g||_\infty\}
		\\&\leq & ||f \mu _0||_{S^{1}} K |g|_{\lip}  \xi ^{n}.
	\end{eqnarray*}
	
\end{proof}

By the same argument as above and by Theorem \ref{speclinf} it holds the following.

\begin{proposition}
	For all Lipschitz function $g: \Sigma \longrightarrow \mathbb{R} $ and all $f \in \Theta _{\mu _0} ^\infty$, it holds $$\left| \int{(g \circ F^n )  f}d\mu_0 - \int{g  }d\mu_0 \int{f  }d\mu_0 \right| \leq ||f \mu _0||_{S^{1}} K |g|_{\lip}  \xi ^{n} \ \ \forall n \geq 1,$$where $\xi_1$ and $K_2$ are from Theorem \ref{speclinf}.
\end{proposition}

In Proposition \ref{disisisi} we will see that under some further assumptions on the system, the sets $ \Theta _{\mu _0} ^1$ contains the set of Lipschitz functions on $\Sigma$.

\section{Application to Lorenz-like maps \label{last}}

In this section, we apply Theorems \ref{spgap} and \ref{speclinf} to a large
class of maps which are Poincar\'{e} maps for suitable sections of
Lorenz-like flows. In these systems (see e.g \cite{AP}), it can be proved
that there is a two dimensional Poincar\'{e} section $\Sigma$ which can be
supposed to be a rectangle $I^2$, where $I=[0,1]$, whose return map $%
F_{L}:I^2 \rightarrow I^2$, after a suitable change of coordinates, has the
form $F_{L}(x,y)=(T_{L}(x),G_{L}(x,y))$, satisfying the properties, G1 and
T1-T3, of section \ref{sec2}. The map $T_{L}: I \longrightarrow I$, in this
case, can be supposed to be piecewise expanding with $C^{1+\alpha }$
branches.

Hence, we consider a class of skew product maps $F_{L}:I^2 \to I^2$, where $%
I=[0,1]$, satisfying $(G1),(T1),(T2),$ and the following properties on $%
T_{L}:$

\subsubsection{Properties of $T_{L}$ in Lorenz-like systems\label{lorenzk}}

\begin{enumerate}
	\item[(P'1)] $\dfrac{1}{|T_{L}^{\prime }|}$ is of universal bounded $p$%
	-variation, i.e. for $p\geq 1$ 
	\begin{equation*}  \label{pvariation}
		\var_p\left(\dfrac{1}{|T_{L}^{\prime }|}\right):= \sup_{0\leq x_{0}<\cdots
			<x_{n}\leq 1}\left( \sum_{i=0}^{n}{|\dfrac{1}{|T_{L}^{\prime }{(x_{i})}|}-%
			\dfrac{1}{|T_{L}^{\prime }{(x_{i-1})}|}|^{p}}\right) ^{\frac{1}{p}}<\infty;
	\end{equation*}
\end{enumerate}

\begin{enumerate}
	\item[(P'2)] $\inf {|{T_{L}^{n_0}}^{\prime }|} \geq \lambda _1>1$, for some $%
	n_0 \in \mathbb{N}$.
\end{enumerate}

The universal bounded $p$-variation, $\var_{p}$, is a generalization of the
usual bounded variation. It is a weaker notion, allowing piecewise Holder
functions. Indeed, for $p\geq 1$, a $1/p$-Holder function is of universal
bounded $p$-variation. This definition is adapted to maps having $%
C^{1+\alpha }$ regularity.

From properties P'1 and P'2, it follows (see \cite{Gk}) that there exists a
suitable strong space (the space $S_{-}$ in T3.1) for the Perron-Frobenius
operator $P_T$ associated to such a $T_{L}$, in a way that it satisfies the
assumptions $T1,...,T3.3$ and $N1$. In this case, supposing a property like $%
T3.4$ then we can apply our results. Therefore, let us introduce the space
of generalized bounded variation functions with respect to the Lebesgue
measure: $BV_{1,\frac{1}{p}}$. The functions of universal bounded $p$%
-variation are included in this space (for more details and results see \cite%
{Gk}, in particular Lemma 2.7 for a comparison of the two spaces).

A piecewise expanding map satisfying assumptions (P'1) and (P'2) has an
invariant measure with density in $BV_{1,\frac{1}{p}}$, moreover the
transfer operator restricted to this space satisfies a Lasota-Yorke
inequality and other interesting properties, as we will see in the following.

\begin{definition}
	Let $m_1$ be the Lebesgue measure on $I=[0,1]$. For an arbitrary function $%
	h:I\longrightarrow \mathbb{C}$ and $\epsilon >0$ define $\osc(h,B_{\epsilon
	}(x)):I\longrightarrow \lbrack 0,\infty ]$ by 
	\begin{equation*}
		\osc(h,B_{\epsilon }(x))=\esssup\{|h(y_{1})-h(y_{2})|;y_{1},y_{2}\in
		B_{\epsilon }(x)\},
	\end{equation*}%
	where $B_{\epsilon }(x)$ denotes the open ball of center $x$ and radius $%
	\epsilon $ and the essential supremum is taken with respect to the product
	measure $m_1^{2}$ on $I^2$. Also define the real function $\osc %
	_{1}(h,\epsilon )$, on the variable $\epsilon $, by 
	\begin{equation*}
		\osc_{1}(h,\epsilon )=\int {\osc(h,B_{\epsilon }(x))}dm(x).
	\end{equation*}
\end{definition}

\begin{definition}
	Fix $A_1>0$ and denote by $\Phi $ the class of all isotonic maps $\phi
	:(0,A_1]\longrightarrow \lbrack 0,\infty ]$, i.e. such that $x\leq
	y\Longrightarrow \phi (x)\leq \phi (y)$ and $\phi (x)\longrightarrow 0$ if $%
	x\longrightarrow 0$. Set
	
	\begin{itemize}
		\item $R_{1}=\{h:I\longrightarrow \mathbb{C};\osc_{1}(h,.)\in \Phi \}$;
		
		\item For $n\in \mathbb{N}$, define $R_{1,n\cdot p}=\{h\in R_{1};\osc%
		_{1}(h,\epsilon )\leq n\cdot \epsilon ^{\frac{1}{p}}\ \ \forall \epsilon \in
		(0,A_1]\}$;
		
		\item And set $S_{1,p}=\bigcup_{n\in \mathbb{N}}{R_{1,n\cdot p}}$.
	\end{itemize}
	
	\label{AAA}
\end{definition}

\begin{definition}
	Let us consider the following spaces and semi-norms:
	
	\begin{enumerate}
		\item $BV_{1,\frac{1}{p}}$ is the space of $m_1$-equivalence classes of
		functions in $S_{1,p}$;
		
		\item Let $h:I\longrightarrow \mathbb{C}$ be a measurable function. Set 
		\begin{equation*}
			\var _{1,\frac{1}{p}}(h)=\sup_{0\leq \epsilon \leq A_1}{\left( \frac{1}{%
					\epsilon ^{\frac{1}{p}}}\osc_{1}(h,\epsilon )\right) }.
		\end{equation*}
	\end{enumerate}
\end{definition}

Since $BV_{1,1/p}$ was defined using a probability measure, $m_1$, then $\var%
_{1,1/p}(h)\leq 2^{1/p}\var_{p}(h)$ (see \cite{Gk}, Lemma 2.7). 

Let us consider $|\cdot |_{1,\frac{1}{p}}:BV_{1,\frac{1}{p}}\longrightarrow 
\mathbb{R}$ defined by%
\begin{equation*}
	|f|_{1,\frac{1}{p}}=\var_{1,\frac{1}{p}}(f)+|f|_{1},
\end{equation*}%
it holds the following (see \cite{Gk}).

\begin{proposition}
	\label{ultima} $\left( BV_{1,\frac{1}{p}}, | \cdot |_{1, \frac{1}{p}}
	\right) $ is a Banach space.
\end{proposition}

In the above setting, G. Keller has shown (see \cite{Gk}) that there is an $%
A_{1}>0$ (we recall that definition \ref{AAA} depends on $A_{1}$) such that:

\begin{enumerate}
	\item[ (a)] $BV_{1,\frac{1}{p}}\subset L^{1}$ is $\func{P}_{T}$-invariant, $%
	\func{P}_{T}: BV_{1,\frac{1}{p}} \longrightarrow BV_{1,\frac{1}{p}}$ is
	continuous and it holds $|\cdot |_{1}\leq |\cdot |_{1,\frac{1}{p}}$;
	
	\item[ (b)] The unit ball of $(BV_{1,\frac{1}{p}},|\cdot |_{1,\frac{1}{p}})$
	is relatively compact in $(L^{1},|\cdot |_{1})$;
\end{enumerate}

\begin{enumerate}
	\item[(c)] There exists $k\in \mathbb{N}$, $0<\beta _{0}<1$ and $C>0$ such
	that
	
	\begin{equation*}
		|\func{P}_{T}^{k}f|_{1,\frac{1}{p}}\leq \beta _{0}|f|_{1,\frac{1}{p}%
		}+C|f|_{1}.  \label{quasilasotaosc2}
	\end{equation*}
\end{enumerate}

Analogously to the proof of inequality (\ref{lasotaiiii}), we have

\begin{equation}
	|\func{P}_{T}^{n}f|_{1,\frac{1}{p}} \leq B_3 \beta _2 ^n | f|_{1,\frac{1}{p}%
	} + C_2|f|_{1}, \ \ \forall n , \ \ \forall f \in BV_{1,\frac{1}{p}},
	\label{lasotaiiii2}
\end{equation}%
for $B_3, C_2 > 0$ and $0<\beta _2 <1$.

Moreover, in \cite{AGP} (Lemma 2), it was shown that

\begin{enumerate}
	\item[(d)] 
	\begin{equation}  \label{jfhfhlfwe}
		|\cdot |_{\infty }\leq A_1^{\frac{1}{p}-1}|\cdot |_{1,\frac{1}{p}}.
	\end{equation}
\end{enumerate}

Therefore, the properties $T1,T2,T3.1,..,T3.3,N1$ of section \ref{sec2} are
satisfied with $S\_=BV_{1,\frac{1}{p}}$. If $T3.4$ is also satisfied, then we can apply our construction to
such maps.

Thus, for $1 \leq p < \infty$, we set 
\begin{equation*}
	\mathcal{BV}_{1,\frac{1}{p}}:=\left\{ \mu \in \mathcal{L}^{1};\var_{1,\frac{1%
		}{p}}(\phi _{x})<\infty, \ \mathnormal{where} \ \phi_x = \dfrac{d\mu _x}{dm_1}
	\right\}
\end{equation*}
and consider $||\cdot||_{1,\frac{1}{p}}: \mathcal{BV}_{1,\frac{1}{p}%
}\longrightarrow \mathbb{R}$, defined by

\begin{equation*}
	||\mu ||_{1,\frac{1}{p}}=|\phi _{x}|_{1,\frac{1}{p}}+||\mu ||_{1}.
\end{equation*}

Clearly, $\left( \mathcal{BV}_{1,\frac{1}{p}},||\cdot ||_{1,\frac{1}{p}%
}\right) $ is a normed space. If we suppose that the system, $T_L:I
\longrightarrow I$, satisfies $T3.4$, then it has an unique absolutely
continuous invariant probability with density $\varphi _{x}\in BV_{1,\frac{1%
	}{p}}.$

As defined in equation (\ref{mathV}), for $1\leq p<\infty $, consider the
set of zero average measures in $\mathcal{BV}_{1,\frac{1}{p}}$, 
\begin{equation*}
	\mathcal{V}_{s}=\{\mu \in \mathcal{BV}_{1,\frac{1}{p}}:\mu (\Sigma )=0\}.
\end{equation*}%
Directly from the above settings, Proposition \ref{5.8} and from Theorem \ref%
{spgap}, using $\mathcal{BV}_{1,\frac{1}{p}}$ as a strong space (playing the
role of $S^{1}$ in Theorem \ref{spgap}) it follows convergence to
equilibrium and spectral gap for these kind of maps.

\begin{proposition}[Exponential convergence to equilibrium]
	\label{5.10} If $\func{F}_{L}$ satisfies assumptions $G1$, $T1$,$T2$, $T3.4$%
	, $P^{\prime }1$ and $P^{\prime }2$, then there exist $D_2>0$ and $0<\beta
	_{2}<1$ such that, for every signed measure $\mu \in \mathcal{V}_s \subset 
	\mathcal{BV}_{1,\frac{1}{p}}, \ 1 \leq p < \infty$, it holds 
	\begin{equation*}
		||\func{F_L^{\ast n }}\mu ||_{1}\leq D_{2}\beta _{1}^{n}||\mu ||_{1,\frac{1}{%
				p}},
	\end{equation*}
	for all $n\geq 1$.
\end{proposition}

\begin{theorem}[Spectral gap for $\mathcal{BV}_{1,\frac{1}{p}}$ ]
	\label{lorgap}If $\func{F}_{L}$ satisfies assumptions $G1$, $T1$,$T2$, $T3.4$%
	, $P^{\prime }1$ and $P^{\prime }2$, then the operator $\func{F}_{L\ast }:%
	\mathcal{BV}_{1,\frac{1}{p}}\longrightarrow \mathcal{BV}_{1,\frac{1}{p}}$
	can be written as 
	\begin{equation*}
		\func{F}_{L\ast }=\func{P}+\func{N}
	\end{equation*}%
	where
	
	\begin{enumerate}
		\item[a)] $\func{P}$ is a projection i.e. $\func{P} ^2 = \func{P}$ and $\dim
		Im (\func{P})=1$;
		
		\item[b)] there are $0<\xi <1$ and $K>0$ such that for all $\mu \in \mathcal{%
			BV}_{1,\frac{1}{p}}$ 
		\begin{equation*}
			||\func{N}^{n}(\mu )||_{ \mathcal{BV}_{1,\frac{1}{p}}}\leq \xi ^{n}K ||\mu
			||_{ \mathcal{BV}_{1,\frac{1}{p}}};
		\end{equation*}
		
		\item[c)] $\func{P}\func{N}=\func{N}\func{P}=0$.
	\end{enumerate}
	
	\label{spect12}
\end{theorem}

We can get the same kind of results for stronger $L^{\infty }$ like norms.
Let us consider 
\begin{equation*}
	\mathcal{BV}_{1,\frac{1}{p}}^{\infty }:=\left\{ \mu \in \mathcal{L}^{\infty
	};\frac{d(\pi _{x\ast }\mu )}{dm_1}\in BV_{1,\frac{1}{p}}\right\}
\end{equation*}%
and the function, $||\cdot ||_{1,\frac{1}{p}}^{\infty }:\mathcal{BV}_{1,%
	\frac{1}{p}}^{\infty }\longrightarrow \mathbb{R}$, defined by

\begin{equation*}
	||\mu ||_{1,\frac{1}{p}}^{\infty }=|\phi _{x}|_{1,\frac{1}{p}}+||\mu
	||_{\infty }.
\end{equation*}%
Applying Theorem \ref{speclinf} \ using $\mathcal{BV}_{1,\frac{1}{p}%
}^{\infty }$ as a strong space (playing the role of $S^{\infty }$)\ we get

\begin{theorem}[Spectral gap for $\mathcal{BV}_{1,\frac{1}{p}}^{\infty }$]
	\label{lorgapinf}If $\func{F}_{L}$ satisfies assumptions $G1$, 
	$T1$,$T2$,$T3.4$, $P^{\prime }1$ and $P^{\prime }2$, then the operator $%
	\func{F}_{L\ast }:\mathcal{BV}_{1,\frac{1}{p}}^{\infty }\longrightarrow 
	\mathcal{BV}_{1,\frac{1}{p}}^{\infty }$ can be written as 
	\begin{equation*}
		\func{F}_{L\ast }=\func{P}+\func{N},
	\end{equation*}%
	where
	
	\begin{enumerate}
		\item[a)] $\func{P}$ is a projection i.e. $\func{P} ^2 = \func{P}$ and $\dim
		Im (\func{P})=1$;
		
		\item[b)] there are $0<\xi _1 <1$ and $K_2>0$ such that for all $\mu \in 
		\mathcal{BV}_{1,\frac{1}{p}} ^\infty$ 
		\begin{equation*}
			||\func{N}^{n}(\mu )||_{{1,\frac{1}{p}}} ^\infty \leq \xi _1 ^{n} K_2 ||\mu
			||_{{1,\frac{1}{p}}} ^\infty;
		\end{equation*}
		
		\item[c)] $\func{P}\func{N}=\func{N}\func{P}=0$.
	\end{enumerate}
	
\end{theorem}

By Proposition \ref{belongs} we immediately get

\begin{proposition}
	\label{rtyrtsyw} If $\func{F}_{L}$ satisfies assumptions $G1$, 
	$T1$,$T2$,$T3.4$, $P^{\prime }1$ and $P^{\prime }2$, then the unique invariant probability for the system $F_{L}$ in $%
	\mathcal{BV}_{1,\frac{1}{p}}$ is $\mu _{0}$. Moreover, since \textbf{N1} is
	satisfied (equation (\ref{jfhfhlfwe})), $\mu _{0}$ is the unique $F_{L}$%
	-invariant probability in $\mathcal{BV}_{1,\frac{1}{p}}^{\infty }$.
\end{proposition}

\section{Quantitative Statistical  Stability }\label{realast}

Throughout this section, we consider small perturbations of the transfer operator of a particular system of the kind described in the previous sections and study the dependence of the physical invariant measure with respect to the perturbation. 
A classical tool that can be applied for this type of problems is the Keller-Liverani stability theorem \cite{KL}. Since in our setting the strong space is not compactly immersed in the weak one, we cannot directly apply it. We will use another approach giving us precise bounds on the statistical stability. In this section, this approach will be applied to a class of Lorenz-like maps with slightly stronger regularity assumptions than used in Section \ref{last}. We call such a system by \textit{BV Lorenz-like map} (see Definition \ref{lorenzlikemap}) and precisely, we need the additional property stated in item (1) of Definition \ref{lorenzlikemap}.

\subsubsection{Uniform Family of Operators}\label{unifop} In this subsection we present a general {\em quantitative} result relating the {\em stability } of the invariant measure of an \textit{uniform family of operators} (Definition \ref{UF}) and {\em convergence to equilibrium}. 

In the following definition, for all $\delta \in [0,1) $, let $\func{L_\delta}$ be a Markov operator acting on two vector subspaces of signed measures on $X$, $\func{L_\delta}:(B_{s}, ||\cdot||_{s} ) \longrightarrow (B_{s}, ||\cdot||_{s} )$ and $\func{L_\delta}: (B_{w}, ||\cdot||_{w} ) \longrightarrow (B_{w}, ||\cdot||_{w} )$, endowed with two norms, the strong norm $||\cdot||_{s}$ on $B_{s},$ and the weak
norm $||\cdot||_{w}$ on $B_{w}$, such that $||\cdot||_{s}\geq ||\cdot||_{w}$. Suppose that,
\begin{equation*}
	B_{s}\mathcal{\subseteq }B_{w}\mathcal{\subseteq }\mathcal{SB}(X),
\end{equation*} where $\mathcal{SB}(X)$ denotes the space of Borel signed measures on $X$. 

\begin{definition}\label{UF}
	
	A one parameter family of transfer operators $\{\func{L}  _{\delta }\}_{\delta \in \left[  0,1 \right)} $ is said to be an \textbf{uniform family of operators}  with respect to the weak space $(B_{w}, ||\cdot||_{w} )$ and the strong space $(B_{s}, ||\cdot||_{s} )$ if $||\cdot||_{s}\geq ||\cdot||_{w}$ and it satisfies
	\begin{enumerate}
		\item [\textbf{UF1}] Let  $\mu_{\delta }\in B_{s}$ be a probability measure fixed under the operator $\func{L}  _{\delta }$. Suppose there is $M>0$ such that for all $\delta \in [0,1)$, it holds $$||\mu_{\delta}||_{s}\leq M;$$
		\item [\textbf{UF2}] $\func{L}  _{\delta }$ approximates $\func{L}  _{0}$ when $\delta $ is
		small in the following sense: there is $C\in \mathbb{R}^+$ such that:%
		\begin{equation*}
			||(\func{L}  _{0}-\func{L}  _{\delta })\mu_{\delta }||_{w}\leq \delta C;
		\end{equation*}
		\item[\textbf{UF3}] $\func{L} _{0}$ has exponential convergence to equilibrium with
		respect to the norms $||\cdot||_{s}$ and $||\cdot||_{w}$: there exists $0<\rho_2 <1$ and $C_2>0$ such that  $$\forall \ \mu \in \mathcal{V}_s: =\{\mu \in B_{s}: \mu(X)=0 \}$$ it holds $$||\func{L}^{n}_0 \mu||_{w} \leq \rho _2 ^n C_2 ||\mu||_{s};$$
		\item[\textbf{UF4}] The iterates of the operators are uniformly bounded for the weak norm: there exists $M_2 >0$ such that $$\forall \delta ,n,\nu \in B_{s}  \ \textnormal{it holds} \ ||\func{L} _{\delta
		} ^{ n}\nu||_{w}\leq M_{2}||\nu||_{w}.$$
	\end{enumerate}
\end{definition}

Under these assumptions we can ensure that the invariant
measure of the system varies continuously (in the weak norm) when $\func{L} _{0}$ is perturbed to $\func{L}  _{\delta }$, for small values of $\delta$. Moreover, the modulus of continuity can be estimated. We postpone the proof of Proposition \ref{dlogd} to the Appendix 3 (section \ref{jsdhjfnsd}).
\begin{proposition}
	\label{dlogd}Suppose $\{\func{L}_{\delta }\}_{\delta \in \left[0, 1 \right)}$ is a uniform family of operators as in Definition \ref{UF}, where $\mu_{0}$ is the unique fixed point of $\func{L}_{0}$ in $B_{w}$ and $%
	\mu_{\delta }$ is a fixed point of $\func{L} _{\delta }$. Then, there exists $\delta _0 \in (0,1)$ such that for all $\delta \in [0,\delta _0)$, it holds
	
	\begin{equation*}
		||\mu_{\delta }-\mu_{0}||_{w}=O(\delta \log \delta ).
	\end{equation*}
\end{proposition}

\subsection{Quantitative stability of Lorenz-like maps}

In this subsection  we apply the above general result on uniform family of operators (Proposition \ref{dlogd}) to a suitable family of bounded variation Lorenz-like maps.
We consider families of maps as defined in Section \ref{last}, with some further regularity assumptions defining uniform families of  Bounded Variation Lorenz-like maps (see Definitions \ref{lorenzlikemap} and \ref{UFL}).
For these families we prove that the invariant measures associated to a size $\delta$ perturbation varies continuously as the map is perturbed, with modulus of continuity $\delta \log \delta$. Precisely, the aim of this section is to prove the following theorem:

\begin{theorem}[Quantitative stability for deterministic perturbations]\label{mainstat}
	Let $\{F_{\delta }\}_{\delta \in [0,1)} $ be a Uniform BV Lorenz-like family (see definition \ref{UFL}). Denote by $\mu_\delta$ the fixed probability measeres of $\func{F}{_*}{_\delta}$ in $\mathcal{BV} _{1,1}$ (also in $\mathcal{BV}^\infty _{1,1}$), for all $\delta$. Then, there exists $\delta _0 \in (0,1)$ such that for all $\delta \in [0,\delta _0)$, it holds
	\begin{equation*}
		||\mu_{\delta }-\mu_{0}||_{1}=O(\delta \log \delta ).
	\end{equation*}
	\label{htyttigu}
\end{theorem}


The proof will be postponed to the end of the section. 

\begin{remark}
	We believe that using the techniques of \cite{Gjep} in which a sort of generalized bounded variation for disintegrated measures is considered in the spirit of the work \cite{Gk} we could get a similar result removing the  additional Bounded Variation  regularity to the Lorenz-like family. 
\end{remark}

\begin{remark}
	A straightforward computation (see the proof of Lemma \ref{existinv}) yields $||\cdot ||_W \leq ||\cdot||_1$. Then, by Theorem (\ref{htyttigu}), it holds $$||\mu_{\delta }-\mu_{0}||_{W}\leq A\delta \log \delta ,$$for some $A>0$. Therefore, for all Lipschitz function $g:[0,1]^2 \longrightarrow\mathbb{R}$, the following estimate holds $$|\int{g}d\mu_\delta - \int{g}d\mu_0| \leq A ||g||_{Lip} \delta \log \delta,$$where $||g||_{Lip} = ||g||_\infty + L(g)$ (see equation (\ref{lipsc}), for the definition of $L(g)$). Thus, for all Lipschitz functions, $g:[0,1]^2 \longrightarrow\mathbb{R}$, the limit $\displaystyle{\lim _{\delta \longrightarrow 0} {\int{g}d\mu_\delta} = \int{g}d\mu_0}$ holds, with rate of convergence smaller than or equal to $\delta \log \delta$.
\end{remark}

\begin{remark}\label{opti} It is well known (see \cite{G} e.g) that the modulus of continuity $\delta \log(\delta) $ is optimal for suitable deterministic perturbations of piecewise expanding maps (which are the basis maps of our piecewise hyperbolic system). Therefore, the estimate given in Theorem \ref{mainstat} is optimal too. To realize this, consider a sequence of piecewise expanding maps $T_n$ with absolutely continuous invariant measures $\mu_n$,  realizing the modulus of continuity $\delta \log(\delta) $.  Consider $F_n:I^2\to I^2$ given by $F_n (x,y)=(T_n(x), \frac12)$ (the second component contracts  everything to $\frac12$). The sequence $F_n$  has a sequence of invariant measures $\nu_n$ of the kind $\nu_n=\mu_n\times \delta_{\frac12}$ for which is easy to see that $||\nu_n-\nu_0||_1\geq A \delta_n \log(\delta_n)$.
\end{remark}

We now precise the definition of  BV Lorenz-like map and  BV Lorenz-like family considered in the Theorem \ref{mainstat}.

\begin{definition}
	A map $F_L:[0,1]^2 \longrightarrow [0,1]^2$, $F_L(x,y) = (T_{L}(x), G_L(x,y))$, is said to be a \textbf{BV Lorenz-like map} if it satisfies

	\begin{enumerate}
		\item There are $H\geq 0$ and a partition $\mathcal{P}'=\{J_{i}:=(b_{i-1},b_{i}),i=1,
		\cdots ,d\}$ of $I$ such that for all $x_{1},x_{2} \in J_i$ and for all $ y\in I$ the following inequality holds
		\begin{equation}\label{rty}
			|G_L(x_{1},y)-G_L(x_{2},y)|\leq H \cdot |x_{1}-x_{2}|;
		\end{equation}
		
		\item $F_L$ satisfy property $G1$ (hence is
		uniformly contracting on each leaf $\gamma $ with rate of contraction $\alpha$);
		
		\item $T_{L}:I\rightarrow I$ is a piecewise expanding map satisfying the assumptions given in the following definition \ref{defpiece123C1}.
	\end{enumerate}
	
	\label{lorenzlikemap}
\end{definition}

The following definition characterizes a class of piecewise expanding maps of the interval with bounded variation derivative
$T_L:I \longrightarrow I$ which is a subclass of the ones considered in section \ref{lorenzk}.

\begin{definition}[Piecewise expanding functions with bounded variation inverse of the derivative]
	Suppose there exists a partition $\mathcal{P}=\{P_{i}:=(a_{i-1},a_{i}),i=1, \cdots ,q\}$ of $I$ s.t. $T_L:I\longrightarrow I$ satisfies the following conditions. For all $i$

	\begin{enumerate}
		\item[1)] $T_{L_{i}}= T _{L}|_{P _i}$ is of class $C^1$ and $g_{i} = \dfrac{1}{|{T_{L_{i}}}^{\prime }|}$ satisfies (P'1) of section \ref{last}, for $p=1$. 
	\end{enumerate}
	%
	
	\begin{enumerate}
		\item[2)] $T_L$ satisfies (P'2) of section \ref{last}: $\inf {|{T_{L}^{n_0}}^{\prime }|} \geq \lambda _1 >1$ for some $n_0 \in \mathbb{N}$.
	\end{enumerate}
	
	\begin{enumerate}
		\item[3)] $T_{L}$ satisfies T3.4.
	\end{enumerate}
	
	\label{defpiece123C1}
\end{definition}

In particular we assume that $T_{L_i}$ and $g_{i}$ admit a continuous extension to $\overline{P _{i}}=[a_{i-1},a_{i}]$ for all $i=1,\cdots ,q$. 


\begin{remark}
	The definition \ref{defpiece123C1} allows infinite derivative for $T_{L}$ at
	the extreme points of its regularity intervals.
\end{remark}



\begin{definition}Let $T_1$ and $T_2$ be two piecewise expanding maps of definition (\ref{defpiece123C1}). Define the set $Int_n$, by $$Int_n=\{A\subset [0,1], s.t. \ A =I_1\cup...\cup I_n, \ \textnormal{where} \ I_i  \ \textnormal{are intervals}\}$$ the set of subsets of  $[0,1]$ which is the union of at most $n$ intervals. Set $$\mathcal{C}(n,T_1,T_2) = \left\{ 
	\begin{array}{c}
	\epsilon: \exists A_{1}\in Int_n \ \textnormal{ and} \ \exists \  \sigma :I\rightarrow
	I \ \textnormal{a diffeomorphism} \ \textnormal{s.t.}~m_1(A_{1})\geq 1-\epsilon , \\ 
	\ T_{1}|_{A_{1}}=T_{2}\circ \sigma |_{A_{1}} \ \textnormal{and} \ \forall
	x\in A_{1},|\sigma (x)-x|\leq \epsilon ,|\frac{1}{\sigma ^{\prime }(x)}%
	-1|\leq \epsilon%
	\end{array}%
	\right\}$$ and define a distance from $T_{1}$ to $T_{2}$ as:
	
	\begin{equation}\label{shock}
		d_{S,n}(T_{1},T_{2})=\inf \left\{\epsilon| \epsilon \in  \mathcal{C}(n,T_1,T_2)\right\}. 
	\end{equation} 
\end{definition}If we denote by $d_{S}$ the classical notion of Skorokhod distance (see \cite{BG} e.g.), it is obvious that $\forall n \ d_{S,n}\geq d_{S}$.
By \cite{BG}, Lemma 11.2.1, it follows that $\forall n $:
\begin{equation}\label{P}
	|\func{P_{T_{1}}}-\func{P_{T_{2}}}|_{BV\rightarrow L^{1}}\leq 14d_{S,n}(T_{1},T_{2}). 
\end{equation}


\begin{definition}\label{UFL}
	A family of maps $\{F_\delta\}_{\delta \in [0,1)}$ is said to be a \textbf{Uniform BV Lorenz-like family} if $F_\delta$ is a BV Lorenz-like map (see definition \ref{lorenzlikemap}) for all $\delta \in [0,1)$ and $\{F_\delta\}_\delta$ satisfies the following assumptions:
	
	\begin{enumerate}
		\item [(UBV1):] there exist $0 < \lambda < 1$ and $D>0$ s.t. for all $f \in BV _{1,1}$ and for all $\delta \in [0,1)$ it holds $|\func {P}_{T_{\delta }} ^n f |_{1,1} \leq D \lambda ^n |f|_{1,1} + D|f|_1 \ \ \textnormal{for all} \ \ n \geq 1,$ where $\func{P}_{T_{\delta }}$ is the Perron-Frobenius operator of $T_{\delta }$.
	\end{enumerate}
	
	When $\delta $ is small
	\begin{enumerate} 
		\item [(UBV2):] $T_{0}$ and $T_{\delta }$ are close with the above Skorokhod-like distance. For some $n$ independent of $\delta$ it holds $\forall \delta$ $$d_{S,n}(T_{0},T_{\delta })\leq \delta. $$
		
		\item [(UBV3):] For each $\delta$ there is a set $A_{2}$ (depending on $\delta$) such that $A_{2}\in Int_{n_\delta}$ for some ${n_\delta}$ (depending on $\delta$)  furthermore $m_1(A_{2})\geq 1-\delta $ and for all $x\in A_{2},y\in I:$ $$|G_{0}(x,y)-G_{\delta }(x,y)|\leq \delta. $$ Let us furthermore suppose that the number of such intervals during the perturbation remains uniformly bounded: $\sup_\delta n_\delta <\infty $.
	\end{enumerate}
	
	For all $\delta \in [0,1)$, let $n_0 = n_0(\delta) \in \mathbb{N}$ be the first integer such that there exists $\lambda _1(\delta)>0$ satisfying $\left| {T_{\delta , i} ^{n_0} }^\prime (x) \right|  \geq \lambda _1 (\delta) >1$ 
	for all $x \in P_{\delta, i}$ and for each $i=1, \cdots, q(\delta)$, where $T_{\delta , i} ^{n_0}:=T_{\delta}{^{n_0}}|_{P_{\delta ,i}}$. 
	Also set $g_{i,\delta} = \dfrac{1}{|T_{\delta,i} '|}$
	and denote by $H_\delta >0$ and ${\mathcal{P}'}_\delta$ the \textquotedblleft Lipschitz\textquotedblright constant and the regularity partition associated  to $G_\delta$, see item (1) of Definition \ref{lorenzlikemap} and Definition \ref{defpiece123C1}. 
	\begin{enumerate}
		\item [(UBV4):] Suppose that:
		\begin{enumerate}
			\item [(1)] $\inf _\delta {\lambda _1 (\delta) } > 1$, $\sup _\delta {\lambda _1 (\delta) } < \infty $ and $\sup _{\delta \in [0,1)} \{n_0(\delta )\}<\infty $; 
			\item [(2)] there exists $C_4 > 0$ such that $\sup g_{\delta, i}\leq C_4$ and $\var g_{\delta, i} \leq C_4$ for all $i=1,\cdots, q(\delta)$ and all $\delta \in [0,1)$;
			\item [(3)] $\inf_{\delta \in [0,1)}\min _{i=1, \cdots, q(\delta)}\{m_1(P_{i,\delta})\} >0$;
			\item [(4)] $\sup_{\delta \in [0,1)} {H_\delta}  < \infty$, $\sup_{\delta \in [0,1)}\# \mathcal{P}'_\delta <\infty$
		\end{enumerate}

		
	\end{enumerate}
	
	

\end{definition}



\subsubsection{Measures with bounded variation}Here, we introduce a space of measures having bounded variation in some stronger sense, and prove that the invariant measure of a BV Lorenz-like map is in it. We use this fact in the proof of Proposition \ref{rr}, where we prove that the family of transfer operators $\{ \func {F_\delta{_*}}\}_{\delta \in [0,1)}$ induced by a Uniform BV Lorenz-like family $\{F_\delta\}_{\delta \in [0,1)}$ satisfies UF2.

We have seen that a positive measure on the square, $[0,1]^2$, can be disintegrated along the stable
leaves $\mathcal{F}^s$ in a way that we can see it as a family of positive measures on the interval, $\{\mu |_\gamma\}_{\gamma \in \mathcal{F}^s }$. Since there is a one-to-one correspondence between $\mathcal{F}^s$  and $[0,1]$, this defines a  path
in the metric space of positive measures, $[0,1] \longmapsto \mathcal{SB}(I)$, where $\mathcal{SB}(I)$ is endowed with the Wasserstein-Kantorovich like metric (see definition \ref{wasserstein}). 
It will be convenient to use a functional notation and denote such a path by  
$\Gamma_{\mu } : I \longrightarrow \mathcal{SB}(I)$  defined $\mu_x$-a.e. ($\mu_x = {\pi_x}_*\mu$) by $\Gamma_{\mu } (\gamma) = \mu|_\gamma = \pi _{\gamma, y} {_\ast}( \phi _x
(\gamma)\mu _\gamma)$, where $(\{\mu _{\gamma }\}_{\gamma \in I},\phi_{x})$ is some disintegration for $\mu$.
However, since such a disintegration is defined $\mu _x$-a.e. $\gamma \in [0,1]$, the path $\Gamma_\mu$ is not unique.  For this reason we define more precisely $\Gamma_{\mu } $ as the class of almost everywhere equivalent paths corresponding to $\mu$.

\begin{definition}
	Consider a positive Borel measure $\mu$ and a disintegration  $\omega=(\{\mu _{\gamma }\}_{\gamma \in I_\omega},\phi
	_{x})$, where $\{\mu _{\gamma }\}_{\gamma \in I_\omega }$ is a family of
	probabilities on $\Sigma$ defined for all $\gamma \in I_\omega$ (where $\mu _x = \phi _x m_1$), $\mu_x(I_\omega)=1$, and $\phi
	_{x}:I_\omega \longrightarrow \mathbb{R} $ is a non-negative marginal density. Denote by $\Gamma_{\mu }$ the class of equivalent paths associated to $\mu$ 
	\begin{equation*}
		\Gamma_{\mu }=\{ \Gamma^\omega_{\mu }\}_\omega,
	\end{equation*}
	where $\omega$ ranges on all the possible disintegrations of $\mu$ on the stable foliation and $\Gamma^\omega_{\mu }:I_\omega \longrightarrow \mathcal{SB}(I)$ is the path associated to a given disintegration, $\omega$:
	$$\Gamma^\omega_{\mu }(\gamma )=\mu |_{\gamma } = \pi _{\gamma, y} {_\ast} \phi _x
	(\gamma)\mu _\gamma .$$
\end{definition}



\begin{definition}
	Let $\mathcal{P}=\mathcal{P}(\Gamma_{\mu }^\omega)$ be a finite sequence $\mathcal{P}=\{x_{i}\}_{i=1}^{n}\subset I_{\omega }$ 
	and define the \textbf{variation of $\Gamma_{\mu }^\omega $ with respect to $\mathcal{P}$} as
	(denote $\gamma _{i}:=\gamma _{x_{i}}$) 
	\begin{equation*}
		\Var(\Gamma_{\mu }^\omega ,\mathcal{P})=\sum_{j=1}^{n}{||\Gamma_{\mu }^\omega  (\gamma _{j})-\Gamma_{\mu }^\omega (\gamma _{j-1})||_{W}},
	\end{equation*}%
	where we recall $|| \cdot ||_{W}$ is the Wasserstein-like norm defined by equation (\ref{WW}).
	Finally, we define the \textbf{variation of $\Gamma_{\mu }^\omega $} by taking the supremum over the set of finite sequences of any length, as 
	\begin{equation*}
		\Var(\Gamma_{\mu }^\omega ):=\sup_{\mathcal{P}}\Var(\Gamma_{\mu }^\omega ,\mathcal{P}).
	\end{equation*}
	\label{VarGmu}
\end{definition}

\begin{remark}
	For an interval $\eta \subset I$, we define 
	\begin{equation*}
		\displaystyle{\Var _{\overline{\eta}} (\Gamma_{\mu }^\omega ) := \Var (\Gamma_{\mu }^\omega  |_{\overline{%
					\eta}})},
	\end{equation*}where $\overline{\eta}$ is the closure of $\eta$. We also remark that $\displaystyle{\Var _{\overline{\eta}} (\Gamma_{\mu }^\omega ) = \Var (\Gamma_{\mu }^\omega  \cdot \chi_{\overline{%
				\eta}})},$ where $\chi_{\overline{%
			\eta}}$ is the characteristic function of $\overline{%
		\eta}$.
\end{remark}

\begin{remark}
	When no confusion can be done, to simplify the notation, we denote $\Gamma_{\mu }^\omega (\gamma )$ just by $\mu |_{\gamma } $.
\end{remark}

\begin{definition}
	Define the \textbf{variation of a positive measure} $\mu $ by 
	\begin{equation}
		\Var(\mu )=\displaystyle{\inf_{ \Gamma_{\mu }^\omega \in \Gamma_{\mu } }\{\Var(\Gamma_{\mu }^\omega )\}}.
	\end{equation}\label{isufgiueigud}
	\label{represents}
\end{definition}

We remark that,
\begin{equation*}
	||\mu ||_{1}=\int {W^0_{1}(0,\Gamma^\omega _{\mu }(\gamma ))}dm_1(\gamma ),\ \ \textnormal{for any}\ \ \Gamma^\omega_{\mu }\in \Gamma_\mu.
\end{equation*}

\begin{definition}
	From the definition \ref{VarGmu} we define the set of bounded variation positive measures $\mathcal{BV}^{+}$ as
	
	\begin{equation*}
		\mathcal{BV}^{+}=\{\mu \in \mathcal{AB}:\mu \geq 0,\Var(\mu )<\infty \}.
	\end{equation*}
\end{definition}

Now we are ready to state a proposition wich will provide an estimative for the regularity of the iterates $\func{F{_\ast }}^n(m)$. Next inequality (\ref{fljghlfjdgkdg}), is a Lasota-Yorke like inequality, where the variation, $\Var(\mu)$, defined in \ref{isufgiueigud}, plays the role of the strong semi-norm. This is our main tool to estimate the regularity of the invariant measure of a BV Lorenz-like map (Proposition \ref{reg}) and it is an immediate consequence of Theorem \ref{las123456hh} and Remark \ref{jjsdgjf} which are proved in Appendix 1.

\begin{proposition}
	\label{propvar}Let $F_L(x,y)=(T_L(x),G_L(x,y))$ be a BV Lorenz-like map. Then, there are $K_{0}$ and $0< \lambda _{0}< 1$ such that for all $\mu \in \mathcal{BV}^{+}$, all disintegration $\omega$ of $\mu$ and all $n\geq 1$ it holds 
	\begin{equation}\label{jdjfjdjfsds}
		\Var(\Gamma_{\operatorname{\func{F{_\ast }}^n \mu} }^\omega  )\leq K_{0}\lambda _{0}^{n}\Var(\Gamma_ \mu ^\omega)+K_{0}{| \phi _x|_{1,1}}. 
	\end{equation}
\end{proposition}

\begin{remark}
	Taking the infimum over all paths $\Gamma_{ \mu } ^\omega  \in \Gamma_{ \mu }$ on both sides of inequality (\ref{jdjfjdjfsds}), we get 
	
	\begin{equation}\label{fljghlfjdgkdg}
		\Var(\func{F{_\ast }}^n\mu )\leq K_{0}\lambda _{0}^{n}\Var(\mu)+K_{0}{| \phi _x|_{1,1}}. 
	\end{equation}
\end{remark}

A precise estimative for $K_{0}$ can be found  in equation (\ref{hdjfdjhuf}). Remember that, by Proposition \ref{belongs}, a Lorenz-like map has an invariant measure $\mu _{0}\in S^{\infty }$.

\begin{remark}\label{riirorpdf}
	Let $m$ be the Lebesgue measure on $\Sigma=I \times I$, i.e. $m=m_1\times m_1$, where $m_1$ is the Lebesgue measure on $I=[0,1]$. Besides that, consider its trivial disintegration $\omega_0 =(\{m_{\gamma}  \}_{\gamma}, \phi_x)$, given by $m_\gamma = \func{\pi _{y,\gamma}^{-1}{_*}}m_1$, for all $\gamma$ and $\phi _x \equiv 1$. According to this definition, it holds that 
	\begin{equation*}
		m|_\gamma = m_1, \ \ \forall \ \gamma.
	\end{equation*}In other words, the path $\Gamma ^{\omega _0}_m$ is constant: $\Gamma ^{\omega _0}_m (\gamma)= m_1$ for all $\gamma$.  Moreover, for each $n \in \mathbb{N}$, let $\omega_n$ be the particular disintegration for the measure $\func{F{_\ast }}^nm$, defined from $\omega_0$ as an application of Lemma \ref{transformula} and consider the path $\Gamma^{\omega_{n}}_{\func{F{_\ast }}^n m}$ associated with this disintegration. By Proposition \ref{niceformulaab} we have

	\begin{equation}
		\Gamma^{\omega_{n}}_{\func{F{_\ast }}^n m} (\gamma)  =\sum_{i=1}^{q}{\dfrac{\func{F^n%
					_{T_{i}^{-n}(\gamma )}}_{\ast  }m_1}{|\det DT^n_{i}\circ T_{i}^{-n}(\gamma ))|}\chi _{T^n_i(P _{i})}(\gamma )}\ \ \forall \ \ \gamma \in N_{1},  \label{niceformulaaw}
	\end{equation}where $P_i$, $i=1, \cdots, q=q(n)$, ranges over the partition $\mathcal{P}^{(n)}$ defined in the following way: for all $n \geq 1$, let $\mathcal{P}^{(n)}$ be the partition of $I$ s.t. $\mathcal{P}^{(n)}(x) = \mathcal{P}^{(n)}(y)$ if and only if $\mathcal{P}^{(1)}(T^j (x)) = \mathcal{P}^{(1)}(T^j(y))$ for all $j = 0, \cdots , n-1$, where $\mathcal{P}^{(1)} = \mathcal{P}$ (see definition \ref{defpiece123C1}). This path will be used in the proof of the next proposition.

\end{remark}

\begin{proposition}
	Let $F_L(x,y)=(T_L(x),G_L(x,y))$ be BV Lorenz-like map and suppose that $F_L$ has an unique invariant probability measure $\mu _{0}\in \mathcal{BV}^{\infty}_{1,1}$. Then $\mu _{0}\in \mathcal{BV}%
	^{+}$ and 
	\begin{equation*}
		\Var(\mu _{0})\leq 2K_{0}.
	\end{equation*}
	\label{reg}
\end{proposition}

\begin{proof}

	Consider the path $\Gamma^{\omega_n}_{\func{F{_\ast }}^n}m$, defined in Remark \ref{riirorpdf},  which represents the measure $\func{F{_\ast }}^nm$.

	
	According to Proposition \ref{rtyrtsyw}, let $\mu _{0}\in \mathcal{BV}^{\infty}_{1,1}$ be the unique $F_L$-invariant probability measure in $\mathcal{BV}^{\infty}_{1,1}$. Consider the Lebesgue measure $m$ and the iterates $\func{F{_\ast }}^n%
	(m)$. By Theorem \ref{lorgapinf}, these iterates converge to $\mu _{0}$
	in $\mathcal{L}^{\infty }$. It implies that the sequence $\{\Gamma_{\func{F{_\ast }}^n(m)} ^{\omega _n}\}_{n}$ converges $m$-a.e. to $\Gamma_{\mu _{0}}^\omega\in \Gamma_{\mu_0 }$ (in $\mathcal{SB}(I)$ with respect to the metric defined in definition \ref{wasserstein}),  where $\Gamma_{\mu _{0}}^\omega$ is a path given by the Rokhlin Disintegration
	Theorem and $\{\Gamma_{\func{F{_\ast }}^n(m)} ^{\omega_n}\}_{n}$ is given by equation (\ref{niceformulaaw}). It implies that $\{\Gamma_{\func{F{_\ast }}^n(m)} ^{\omega_n}\}_{n}$ converges pointwise to $\Gamma_{\mu _{0}}^\omega$ on a full measure set $\widehat{I}\subset I$. Let us denote $%
	\Gamma_{n}:=\Gamma^{\omega_n}_{\func{F{_\ast }}^n(m)}|_{%
		\widehat{I}}$ and $\Gamma:=\Gamma^\omega _{\mu _{0}}|_{\widehat{I}}$. Since $\{\Gamma_{n} \}_n $ converges pointwise to $\Gamma$ it holds $\Var(\Gamma_{n}, \mathcal{P}) \longrightarrow \Var(\Gamma, \mathcal{P})$ as $n \rightarrow \infty$ for all finite sequences $\mathcal{P} \subset \widehat{I}$. Indeed, let $\mathcal{P}=\{x_1, \cdots, x_k\} \subset \widehat{I}$ be a finite sequence. Then,

	\begin{equation*}
		\Var(\Gamma_{n} ,\mathcal{P})=\sum_{j=1}^{k}{||\Gamma_{n}  (x _{j})-\Gamma_{n}(x _{j-1})||_{W}},
	\end{equation*}taking the limit, we get
	
	\begin{eqnarray*}
		\lim _{n \longrightarrow \infty} {\Var(\Gamma_{n} ,\mathcal{P})} &= & \lim _{n \longrightarrow \infty} {\sum_{j=1}^{k}{||\Gamma_{n}  (x _{j})-\Gamma_{n}(x _{j-1})||_{W}}} \\&= & \sum_{j=1}^{k}{||\Gamma (x _{j})-\Gamma} (x _{j-1})||_{W} \\&= & \Var(\Gamma ,\mathcal{P}).
	\end{eqnarray*} On the other hand, $\Var(\Gamma_{n}, \mathcal{P})\leq \Var (\Gamma_{n}) \leq 2 K_0$  for all $n\geq 1$, where $K_0$ comes from Proposition \ref{propvar}. Then $\Var(\Gamma^\omega_{\mu _0}, \mathcal{P}) \leq  2 K_0$ for all partition $\mathcal{P}$. Thus, $\Var(\Gamma^\omega_{\mu _0}) \leq 2 K_0$ and hence  $\Var(\mu _0) \leq  2 K_0$.

\end{proof}

\begin{remark}
	We remark that, Proposition \ref{reg} is an estimation of the regularity of the disintegration of $\mu _{0}$. Similar results are presented in \cite{GP} and \cite{BM}.
\end{remark}

In Section \ref{decay} we proved exponential decay of corretation for Lorenz like maps and observables in the set $f \in \Theta _{\mu _0} ^1$. In this section we prove that for BV Lorenz like maps, the set $f \in \Theta _{\mu _0} ^1$ contains the set of Lipschitz functions.  Denote the space of the Lipschitz functions. $f:[0,1]^2\longrightarrow \mathbb{R}$ by $\lip(\Sigma)$.
As a consequence of Proposition \ref{reg}, next Proposition \ref{disisisi} yields $\lip(\Sigma) \subset \Theta _{\mu_0} ^1$ (defined in subsection \ref{decay}). In order to prove it, we need the next Lemma \ref{hdgfghddsfg} on disintegration of absolutely continuous measures with respect to a measure $\mu_0 \in \mathcal{AB}$, where its proof was postponed to the Appendix 4.

\begin{lemma}\label{hdgfghddsfg}
	Let $(\{\mu_{0, \gamma}\}_\gamma, \phi_x)$ be the disintegration of $\mu _0$, along the partition $\mathcal{F}^s:=\{\{\gamma\} \times N_2: \gamma \in N_1\}$, and for a $\mu_0$ integrable function $f:N_1 \times N_2 \longrightarrow \mathbb{R}$, denote by $\nu$ the measure $\nu:=f\mu_0$ ($ f\mu_0(E) := \int _E {f}d\mu _0$). If $(\{\nu_{ \gamma}\}_\gamma, \widehat{\nu} )$ is the disintegration of $\nu$, where $\widehat{\nu}:=\pi_x{_*} \nu$, then it holds $\widehat{\nu} \ll m_1$ and $\nu _\gamma \ll \mu_{0, \gamma}$. Moreover, denoting $\overline{f}:=\dfrac{d\widehat{\nu}}{dm_1}$, it holds 
	\begin{equation}\label{fjgh}
		\overline{f}(\gamma)=\int_{N_1}{f(\gamma, y)}d(\mu_0|_\gamma),
	\end{equation} and for $\widehat{\nu}$-a.e. $\gamma \in N_1$

	\begin{equation}\label{gdfgdgf}
		\dfrac{d\nu _{ \gamma}}{d\mu _{0, \gamma}}(y) =
		\begin{cases}
			\dfrac{f|_\gamma (y)}{\int{f|_\gamma(y)}d\mu_{0,\gamma}(y)} , \ \hbox{if} \ \gamma \in B ^c \\
			0, \ \hbox{if} \ \gamma \in B,
		\end{cases} \hbox{for all} \ \ y \in N_2,
	\end{equation}where $B :=  \overline{f} ^{-1}(0)$.
\end{lemma}

\begin{proposition}\label{disisisi}
	Let $F_L:[0,1]^2 \longrightarrow [0,1]^2$, $F_L(x,y) = (T_{L}(x), G_L(x,y))$,  be a BV Lorenz-like map and $\mu_0 \in \mathcal{BV}_{1,1}$ the unique $F_L$-invariant measure in $\mathcal{BV}_{1,1}$. Then, $\lip(\Sigma) \subset \Theta _{\mu_0} ^1$
\end{proposition}

\begin{proof}
	Let $(\{\mu_{0, \gamma}\}_\gamma, \phi_x)$ be the disintegration of $\mu _0$ and denote by $\nu$ the measure $\nu:=f\mu_0$ ($ f\mu_0(E) := \int _E {f}d\mu _0$). If $(\{\nu_{ \gamma}\}_\gamma, \widehat{\nu} )$ is the disintegration of $\nu$, then it holds $\widehat{\nu} \ll m_1$ and $\nu _\gamma \ll \mu_{0, \gamma}$ (see appendix 4, section \ref{disint}). Moreover, denoting $\overline{f}:=\dfrac{d\widehat{\nu}}{dm_1}$, it holds 
	\begin{equation*}
		\overline{f}(\gamma)=\int_{[0,1]}{f(\gamma, y)}d(\mu_0|_\gamma),
	\end{equation*} and

	\begin{equation*}
		\dfrac{d\nu _\gamma}{d\mu_{0, \gamma}} (y)=\dfrac{f(\gamma,y) }{\overline{f}(\gamma)}, \ \ \textnormal{if} \ \  	\overline{f}(\gamma) \neq 0.
	\end{equation*}and
	\begin{equation*}
		\dfrac{d\nu _\gamma}{d\mu_{0, \gamma}} (y) \equiv 0, \ \ \textnormal{if} \ \  	\overline{f}(\gamma) = 0.
	\end{equation*}It is immediate that $\nu \in \mathcal{L}^1$. Let us check that $\overline{f} \in BV_{1,1}$ by estimating the variation of $\overline{f}$. For an arbitrary partition $\mathbb{P} = \{0=\gamma_0, \gamma_1, \cdots, \gamma_n=1\}$ of $[0,1]$, we have

	\begin{eqnarray*}
		|\overline{f}(\gamma_{i}) - \overline{f}(\gamma_{i-1})| &\leq& \left|\int_{[0,1]}{f(\gamma _i, y)}d(\mu_0|_{\gamma_i}) - \int_{[0,1]}{f(\gamma _{i-1}, y)}d(\mu_0|_{\gamma_{i-1}}) \right| 
		\\&\leq& \left|\int_{[0,1]}{f(\gamma _i, y)}d(\mu_0|_{\gamma_i}) - \int_{[0,1]}{f(\gamma _{i}, y)}d(\mu_0|_{\gamma_{i-1}}) \right| 
		\\&+& \left|\int_{[0,1]}{f(\gamma _i, y)}d(\mu_0|_{\gamma_{i-1}}) - \int_{[0,1]}{f(\gamma _{i-1}, y)}d(\mu_0|_{\gamma_{i-1}}) \right| 
		\\&\leq& \left|\int_{[0,1]}{f(\gamma _i, y)}d(\mu_0|_{\gamma_i}-\mu_0|_{\gamma_{i-1}}) \right| 
		\\&+& \left|\int_{[0,1]}{f(\gamma _i, y)-f(\gamma _{i-1}, y)}d(\mu_0|_{\gamma_{i-1}}) \right|
		\\&\leq& ||f||_{\lip} ||\mu_0|_{\gamma_i}-\mu_0|_{\gamma_{i-1}}||_W + L(f)|\gamma_i - \gamma_{i-1}| \left|\phi _x \right|_{\infty}.
	\end{eqnarray*}Thus, $\var{\overline{f}} < \infty$ and $\overline{f} \in BV_{1,1}$ (since $\var_{1,1}{\overline{f}} \leq 2\var{\overline{f}}$). 
\end{proof}

The proof of the following proposition is postponed to the appendix.

\begin{proposition}\label{thshgf}
	Let $\{F_{\delta }\}_{\delta \in [0,1)}$ be a Uniform BV Lorenz-like family (definition (\ref{UFL})) and let $f _{\delta}$ be the unique $F_\delta$-invariant probability in $\mathcal{BV}_{1,1}$ (also in $\mathcal{BV}^{\infty}_{1,1}$). Then, there exists $B_u>0$ such that 
	\begin{equation*}
		\Var(f _{\delta})\leq 2B_u,
	\end{equation*}for all $\delta \in[0,1)$.
\end{proposition}

For the next proposition we will use the following notation. Given a probability measure $f_\delta$ on $I^2$ and a measurable set $E\subset I$, we define the measure $1_Ef_\delta$ on $I^2$, by 
\begin{equation*}
	1_Ef_\delta (A) := f_\delta (A \cap \pi _x ^{-1}(E)) \ \textnormal{for all measurable set} \ A \subset I^2.
\end{equation*}
We remark that, if $(\{f_{\delta, \gamma}\}_\gamma, \phi _{x,\delta})$ is a disintegration of $f_\delta$, then 
\begin{equation}
	\label{jnjjng}
	(\{ f_{\delta, \gamma} \}_\gamma, \chi _{E}\phi _{x,\delta} ),
\end{equation}is a disintegration of $1_Ef_\delta (A)$.

\begin{proposition}[to obtain UF2]\label{UF2ass}
	Let $\{F_\delta \}_{\delta \in [0,1)}$ be a family of BV Lorenz-like maps which  satisfies UBV2, UBV3 and UBV4 of definition \ref{UFL}. Denote by $\func{F_\delta{_\ast}}$ their transfer operators and by $%
	f_{\delta }$ their fixed points (probabilities) in $\mathcal{BV}_{1,1}$ (also in $\mathcal{BV}^{\infty}_{1,1}$). Suppose that $f_{\delta }$ has uniformly bounded variation, 
	$$\Var(f_{\delta })\leq M_{2}, \ \forall \delta.$$
	Then, there is a constant $C_{1}$ such that for $\delta $ small enough, it holds
	$$
	||(\func{F_0{_\ast }}-\func{F_\delta{_\ast }})f_{\delta }||_{{1}}\leq
	C_{1}\delta (M_{2}+1).$$
\end{proposition}

\begin{proof}

	Set $A=A_{1}\cap A_{2}$  where $A_1$ comes from de definition of $d_{S,n}$ (see equation (\ref{shock})) and $A_2$ is from (UBV3) (see definition \ref{UFL}). Remark that these sets depend on $\delta $. Let us
	estimate
	
	\begin{eqnarray*}\label{12112}
		||(\func{F_0{_\ast }}-\func{F_\delta{_\ast }})f_{\delta}||_{{1}}&\leq& \int_{I}{||{\func{F_0{_\ast }}(1}}_{A}{f_{\delta })|}_{\gamma }{-{\func{F_\delta{_\ast }}({1}}_{A}f_{\delta })|_{\gamma}||_{W}}dm_1(\gamma )\\&+&\int_{I}{||{\func{F_0{_\ast }}(1}}_{A^{c}}{f_{\delta })|}_{\gamma }{-{\func{F_\delta{_\ast }}({1}}_{A^{c}}f_{\delta })|_{\gamma }||_{W}}dm_1(\gamma ).
	\end{eqnarray*} 
	

	By the assumptions, for a.e. $\gamma$, $||{f_{\delta }|\gamma ||}_{W}\leq
	(M_{2}+1)$ and  $||{1}_{A^{c}}{f_{\delta }||}_{1}\leq 2\delta(M_{2}+1) .$ Indeed, since $\Var(f_{\delta })\leq M_{2}, \ \forall \delta$, we have (below, we denote $\phi_{x,\delta} =\dfrac{d\pi_x{_*}(f_\delta)}{dm_1}$) 
	
	\begin{eqnarray*}
		||f_{\delta }|_\gamma ||_W &\leq& ||f_{\delta }|_\gamma - f_{\delta }|_{\gamma_2}||_W + ||f_{\delta }|_{\gamma_2}||_W \\&=& ||f_{\delta }|_\gamma - f_{\delta }|_{\gamma_2}||_W + |\phi_{x,\delta}(\gamma_2)|.
	\end{eqnarray*}Integrating with respect to $\gamma_2$ we get 
	\begin{equation}\label{sksdjfn}
		||{f_{\delta }|\gamma ||}_{W}\leq
		(M_{2}+1).
	\end{equation}To prove the inequality $||{1}_{A^{c}}{f_{\delta }||}_{1}\leq 2\delta(M_{2}+1) $ we use the previous equation, $m_1(A^c) \leq 2\delta$ and the fact that (see equation (\ref{jnjjng})) $$||{1}_{A^{c}}{f_{\delta }||}_{1} = \int _{A^c}{||f_\delta |_\gamma||_W}dm_1.$$

	Since ${\func{F_\ast }}$ is a contraction for the weak norm, we have 
	\[
	\int_{I}{||{\func{F_0{_\ast }}(1}}_{A^{c}}{f_{\delta })|}_{\gamma }{-{\func{F_\delta{_\ast }}({1}}_{A^{c}}f_{\delta })|_{\gamma }||_{W}}dm_1(\gamma )\leq
	4\delta(M_{2}+1) .
	\]Now, let us estimate the first summand of (\ref{12112}) by estimating the integral 
	$$\int {||({\func{F_0{_\ast }}}\mu -{\func{F_\delta{_\ast }}}\mu )|_{\gamma }||_{W}}%
	dm_1(\gamma ),$$ where $\mu=1_{A}f_{\delta }$. Denote by $T_{0,i}$, with $0\leq i\leq q$, the branches of $T_{0}$
	defined in the sets $P_{i} \in \mathcal{P}$ and set $T_{\delta,i}=T_{\delta }|_{P_{i}\cap A}$. These functions will play the role of the branches for $T_{\delta }.$ Since in $A,$ $T_{0}=T_{\delta }\circ \sigma_{\delta }$ (where $\sigma _{\delta }$ is the diffeomorphism in the definition of the Skorokhod distance), then\ $T_{\delta ,i}$ are invertible.
	Then%
	\[
	({\func{F_0{_\ast }}}\mu -{\func{F_\delta{_\ast }}}\mu )|_{\gamma }=\sum_{i=1}^{q}%
	\frac{{\func{F}_{0,T_{0,i}^{-1}(\gamma )}{_\ast }}\mu |_{T_{0,i}^{-1}(\gamma )}\chi
		_{T_{0}(P_{i}\cap A)}}{|T_{0,i}^{\prime }(T_{0,i}^{-1}(\gamma ))|}%
	-\sum_{i=1}^{q}\frac{{\func{F}_{\delta ,T_{\delta ,i}^{-1}(\gamma )}{_\ast }}\mu
		|_{T_{\delta ,i}^{-1}(\gamma )}\chi _{T_{\delta }(P_{i}\cap A)}}{|T_{\delta
			,i}^{\prime }(T_{\delta ,i}^{-1}(\gamma ))|}\ \ \mu _{x}-a.e.\ \gamma \in I. 
	\]%
	Let us now consider $T_{0}(P_{i}\cap A)$, $T_{\delta }(P_{i}\cap A)$ and remark that $T_{0}(P_{i}\cap A)=\sigma _{\delta }(T_{\delta }(P_{i}\cap A))$
	where $\sigma _{\delta }$ is a diffeomorphism near to the identity. Let us
	denote $B_{i}=T_{0}(P_{i}\cap A)\cap T_{\delta }(P_{i}\cap
	A)$ and $C_{i}=T_{0}(P_{i}\cap A)\triangle T_{\delta }(P_{i}\cap A)$. Then, we have
	
	\begin{equation*}
		\int_{I}{||({\func{F_0{_\ast }}}\mu -{\func{F_\delta{_\ast }}}\mu )|_{\gamma }||_{W}}%
		dm_1(\gamma ) \leq O_1 + O_2,
	\end{equation*}where $$O_1 = \int_{I}\left\vert \left\vert \sum_{i=1}^{q}\frac{{\func{F}_{0,T_{0,i}^{-1}(\gamma )}{_\ast }}\mu |_{T_{0,i}^{-1}(\gamma )}\chi_{B_{i}}}{|T_{0,i}^{\prime }(T_{0,i}^{-1}(\gamma ))|}-\sum_{i=1}^{q}\frac{{\func{F}_{\delta ,T_{\delta ,i}^{-1}(\gamma )}{_\ast }}\mu |_{T_{\delta,i}^{-1}(\gamma )}\chi _{B_{i}}}{|T_{\delta ,i}^{\prime }(T_{\delta,i}^{-1}(\gamma ))|}\right\vert \right\vert _{W}dm_1$$ and 
	$$O_2 = \int_{I}\left\vert\left\vert \sum_{i=1}^{q}\frac{{\func{F}_{0,T_{0,i}^{-1}(\gamma )}{_\ast }}\mu |_{T_{0,i}^{-1}(\gamma )}\chi _{T_0(P_i\cap A)-B_i}}{|T_{0,i}^{\prime }(T_{0,i}^{-1}(\gamma ))|}-\sum_{i=1}^{q}\frac{{\func{F}_{\delta,T_{\delta ,i}^{-1}(\gamma )}{_\ast }}\mu |_{T_{\delta ,i}^{-1}(\gamma)}\chi _{T_\delta (P_i\cap A)-B_i}}{|T_{\delta ,i}^{\prime }(T_{\delta ,i}^{-1}(\gamma ))|}\right\vert \right\vert _{W}dm _1.$$
	And since  $m_1(C_{i})=O(\delta )$, we \footnote{Remark that $m_1(T_\delta(P_i\cap A)\triangle T_0(P_i\cap A))=O(\delta)$ because $T_\delta(P_i\cap A)=\sigma( T_0(P_i\cap A))$ where $\sigma$ is a diffeomorphism near to the identity as in the definition of the Skhorokod distance and $P_i\cap A$ is a finite  union of intervals whose number is uniformly bounded with respect to $\delta$.} 
	get that there is  $K_1\geq0$ such that $O_2  \leq qK_{1}(M_{2}+1)\delta .$ In order to estimate $O_1$, we note that
	

	\begin{eqnarray*}
		O_1 &=& \int_{I}\left\vert \left\vert \sum_{i=1}^{q}\frac{{\func{F}_{0,T_{0,i}^{-1}(%
					\gamma )}{_\ast }}\mu |_{T_{0,i}^{-1}(\gamma )}\chi _{B_{i}}}{%
			|T_{0,i}^{\prime }(T_{0,i}^{-1}(\gamma ))|}-\sum_{i=1}^{q}\frac{{\func{F}_{\delta
					,T_{\delta ,i}^{-1}(\gamma )}{_\ast }}\mu |_{T_{\delta ,i}^{-1}(\gamma
				)}\chi _{B_{i}}}{|T_{\delta ,i}^{\prime }(T_{\delta ,i}^{-1}(\gamma ))|}%
		\right\vert \right\vert _{W}dm_1 \\
		&\leq &\int_{I}\left\vert \left\vert \sum_{i=1}^{q}\frac{{%
				\func{F}_{0,T_{0,i}^{-1}(\gamma )}{_\ast }}\mu |_{T_{0,i}^{-1}(\gamma )}\chi
			_{B_{i}}}{|T_{0,i}^{\prime }(T_{0,i}^{-1}(\gamma ))|}-\sum_{i=1}^{q}\frac{{%
				\func{F}_{\delta ,T_{\delta ,i}^{-1}(\gamma )}{_\ast }}\mu |_{T_{0,i}^{-1}(\gamma
				)}\chi _{B_{i}}}{|T_{\delta ,i}^{\prime }(T_{\delta ,i}^{-1}(\gamma ))|}%
		\right\vert \right\vert _{W}dm_1 \\
		&+&\int_{I}\left\vert \left\vert \sum_{i=1}^{q}\frac{{\func{F}_{\delta ,T_{\delta
						,i}^{-1}(\gamma )}{_\ast }}\mu |_{T_{0,i}^{-1}(\gamma )}\chi _{B_{i}}}{%
			|T_{\delta ,i}^{\prime }(T_{\delta ,i}^{-1}(\gamma ))|}-\sum_{i=1}^{q}\frac{{%
				\func{F}_{\delta ,T_{\delta ,i}^{-1}(\gamma )}{_\ast }}\mu |_{T_{\delta
					,i}^{-1}(\gamma )}\chi _{B_{i}}}{|T_{\delta ,i}^{\prime }(T_{\delta
				,i}^{-1}(\gamma ))|}\right\vert \right\vert _{W}dm _1 \\
		&=&\int_{I}I(\gamma )~dm_1(\gamma )+\int_{I}II(\gamma )~dm_1(\gamma ).
	\end{eqnarray*}The two summands will be treated separately. Let us denote $\overline{\mu }%
	|_{\gamma }=\pi _{\gamma ,y}{_\ast }\mu _{\gamma }$ (note that $\mu
	|_{\gamma }=\phi _{\mu }(\gamma )\overline{\mu }|_{\gamma }$ and $\overline{%
		\mu }|_{\gamma }$ is a probability measure).%
	\begin{eqnarray*}
		I(\gamma ) &=&\left\vert \left\vert \sum_{i=1}^{q}\frac{{\func{F}_{0,T_{0,i}^{-1}(%
					\gamma )}{_\ast }}\mu |_{T_{0,i}^{-1}(\gamma )}\chi _{B_{i}}}{%
			|T_{0,i}^{\prime }(T_{0,i}^{-1}(\gamma ))|}-\sum_{i=1}^{q}\frac{{\func{F}_{\delta
					,T_{\delta ,i}^{-1}(\gamma )}{_\ast }}\mu |_{T_{0,i}^{-1}(\gamma )}\chi
			_{B_{i}}}{|T_{\delta ,i}^{\prime }(T_{\delta ,i}^{-1}(\gamma ))|}\right\vert
		\right\vert _{W} \\
		&\leq &\left\vert \left\vert \sum_{i=1}^{q}\frac{{\func{F}_{0,T_{0,i}^{-1}(\gamma
					)}{_\ast }}\mu |_{T_{0,i}^{-1}(\gamma )}\chi _{B_{i}}}{|T_{0,i}^{\prime
			}(T_{0,i}^{-1}(\gamma ))|}-\sum_{i=1}^{q}\frac{{\func{F}_{\delta ,T_{\delta
						,i}^{-1}(\gamma )}{_\ast }}\mu |_{T_{0,i}^{-1}(\gamma )}\chi _{B_{i}}}{%
			|T_{0,i}^{\prime }(T_{0,i}^{-1}(\gamma ))|}\right\vert \right\vert _{W} \\
		&+&\left\vert \left\vert \sum_{i=1}^{q}\frac{{\func{F}_{\delta ,T_{\delta
						,i}^{-1}(\gamma )}{_\ast }}\mu |_{T_{0,i}^{-1}(\gamma )}\chi _{B_{i}}}{%
			|T_{0,i}^{\prime }(T_{0,i}^{-1}(\gamma ))|}-\sum_{i=1}^{q}\frac{{\func{F}_{\delta
					,T_{\delta ,i}^{-1}(\gamma )}{_\ast }}\mu |_{T_{0,i}^{-1}(\gamma )}\chi
			_{B_{i}}}{|T_{\delta ,i}^{\prime }(T_{\delta ,i}^{-1}(\gamma ))|}\right\vert
		\right\vert _{W} \\
		&=&I_{a}(\gamma )+I_{b}(\gamma ).
	\end{eqnarray*}%
	Since $f_{\delta }$ is a probability measure it holds, posing $\beta
	=T_{0,i}^{-1}(\gamma )$%
	\begin{eqnarray*}
		\int I_{a}(\gamma )dm_1 &=&\int \left\vert \left\vert \sum_{i=1}^{q}\frac{{%
				\func{F}_{0,T_{0,i}^{-1}(\gamma )}{_\ast }}\mu |_{T_{0,i}^{-1}(\gamma )}\chi
			_{B_{i}}}{|T_{0,i}^{\prime }(T_{0,i}^{-1}(\gamma ))|}-\sum_{i=1}^{q}\frac{{%
				\func{F}_{\delta ,T_{\delta ,i}^{-1}(\gamma )}{_\ast }}\mu |_{T_{0,i}^{-1}(\gamma
				)}\chi _{B_{i}}}{|T_{0,i}^{\prime }(T_{0,i}^{-1}(\gamma ))|}\right\vert
		\right\vert _{W}dm_1(\gamma ) \\
		&\leq &\int \sum_{i=1}^{q}\left\vert \left\vert \frac{{\func{F}_{0,T_{0,i}^{-1}(%
					\gamma )}{_\ast }}\mu |_{T_{0,i}^{-1}(\gamma )}\chi _{B_{i}}}{%
			|T_{0,i}^{\prime }(T_{0,i}^{-1}(\gamma ))|}-\frac{{\func{F}_{\delta ,T_{\delta
						,i}^{-1}(\gamma )}{_\ast }}\mu |_{T_{0,i}^{-1}(\gamma )}\chi _{B_{i}}}{%
			|T_{0,i}^{\prime }(T_{0,i}^{-1}(\gamma ))|}\right\vert \right\vert _{W}dm_1 \\
		&\leq &\sum_{i=1}^{q}\int \left\vert \left\vert \frac{{\func{F}_{0,T_{0,i}^{-1}(%
					\gamma )}{_\ast }}\mu |_{T_{0,i}^{-1}(\gamma )}\chi _{B_{i}}}{%
			|T_{0,i}^{\prime }(T_{0,i}^{-1}(\gamma ))|}-\frac{{\func{F}_{\delta ,T_{\delta
						,i}^{-1}(\gamma )}{_\ast }}\mu |_{T_{0,i}^{-1}(\gamma )}\chi _{B_{i}}}{%
			|T_{0,i}^{\prime }(T_{0,i}^{-1}(\gamma ))|}\right\vert \right\vert _{W}dm_1 \\
		&\leq &\sum_{i=1}^{q}\int_{T_{0,i}^{-1}(B_{i})}\left\vert \left\vert {%
			\func{F}_{0,\beta }{_\ast }}\mu |_{\beta }-{\func{F}_{\delta ,T_{\delta
					,i}^{-1}(T_{0,i}(\beta ))}{_\ast }}\mu |_{\beta }\right\vert \right\vert
		_{W}dm_1(\beta ).
	\end{eqnarray*} We remark $T_{0,i}^{-1}(B_{i})\subseteq P_{i}\cap A$ and $T_{\delta
		,i}^{-1}(T_{0,i}(T_{0,i}^{-1}(B_{i})))\subseteq P_{i}\cap A$. Moreover, since $%
	|T_{\delta ,i}(\beta )-T_{0,i}(\beta )|\leq \delta $ and $T_{0,i}^{-1}$ \ is
	a contraction, then 
	
	\begin{equation}
		\label{3times}
		|T_{0,i}^{-1}\circ T_{\delta ,i}(\beta )-\beta |\leq
		\delta.
	\end{equation}Therefore
	\begin{eqnarray*}
		\left\vert \left\vert {\func{F}_{0,\beta }{_\ast }}\mu |_{\beta }-{\func{F}_{\delta
				,T_{\delta ,i}^{-1}(T_{0,i}(\beta ))}^{_\ast }}\mu |_{\beta }\right\vert
		\right\vert _{W} &\leq &\left\vert \left\vert {\func{F}_{0,\beta }{_\ast }}\mu
		|_{\beta }-{\func{F}_{\delta ,\beta }{_\ast }}\mu |_{\beta }\right\vert \right\vert
		_{W} \\
		&&+\left\vert \left\vert {\func{F}_{\delta ,\beta }{_\ast }}\mu |_{\beta }-{%
			\func{F}_{\delta ,T_{\delta ,i}^{-1}(T_{0,i}(\beta ))}{_\ast }}\mu |_{\beta
		}\right\vert \right\vert _{W}.
	\end{eqnarray*}%
	By (UBV3) and equation (\ref{sksdjfn}),%
	\[
	\left\vert \left\vert {\func{F}_{0,\beta }{_\ast }}\mu |_{\beta }-{\func{F}_{\delta ,\beta
		}{_\ast }}\mu |_{\beta }\right\vert \right\vert _{W}\leq \delta (M_{2}+1). 
	\]%
	
	Then, by (\ref{3times}), we have
	\[
	\left\vert \left\vert {\func{F}_{\delta ,\beta }{_\ast }}\mu |_{\beta }-{\func{F}_{\delta
			,T_{\delta ,i}^{-1}(T_{0,i}(\beta ))}{_\ast }}\mu |_{\beta }\right\vert
	\right\vert _{W}\leq H_\delta\delta (M_{2}+1)
	\]%
	
	when $d(\beta, \cup_i \partial J_i)\geq \delta $. For the other values of $\beta$ we remark that the set of points $\{x \ s.t. \ d(x,\cup_i \partial J_i)\leq \delta \}$ is of measure 
	bounded by
	$\delta (\sup_\delta \# \mathcal{P}'_\delta )$, thus 
	\[
	\int I_{a} dm_1= O (\delta ). 
	\]%
	To estimate $I_{b}(\gamma )$, we have%
	\begin{eqnarray*}
		I_{b}(\gamma ) &=&\left\vert \left\vert \sum_{i=1}^{q}\frac{{\func{F}_{\delta
					,T_{\delta ,i}^{-1}(\gamma )}{_\ast }}\mu |_{T_{0,i}^{-1}(\gamma )}\chi
			_{B_{i}}}{|T_{0,i}^{\prime }(T_{0,i}^{-1}(\gamma ))|}-\sum_{i=1}^{q}\frac{{%
				\func{F}_{\delta ,T_{\delta ,i}^{-1}(\gamma )}{_\ast }}\mu |_{T_{0,i}^{-1}(\gamma
				)}\chi _{B_{i}}}{|T_{\delta ,i}^{\prime }(T_{\delta ,i}^{-1}(\gamma ))|}%
		\right\vert \right\vert _{W} \\
		&\leq &\sum_{i=1}^{q}\left\vert \frac{\chi _{B_{i}}(\gamma )}{%
			|T_{0,i}^{\prime }(T_{0,i}^{-1}(\gamma ))|}-\frac{\chi
			_{B_{i}}(\gamma )}{|T_{\delta ,i}^{\prime }(T_{\delta ,i}^{-1}(\gamma ))|}%
		\right\vert \left\vert \left\vert \func{F}_{\delta ,T_{\delta ,i}^{-1}(\gamma
			)}{_\ast }\mu |_{T_{0,i}^{-1}(\gamma )}\right\vert \right\vert _{W}
	\end{eqnarray*}
	and
	\[
	\int I_{b}(\gamma)~dm_1(\gamma)\leq |({P}_{T_{0}}-{P}_{T_{\delta }}\left) (1)\right\vert
	(M_{2}+1). 
	\]By \cite{BG}, Lemma 11.2.1,
	\[
	\int_{A_{1}}I_{b}(\gamma )~dm_1(\gamma )\leq 14(M_{2}+1)\delta . 
	\]%
	Now, let us estimate the integral of the second summand%
	\[
	II(\gamma )=\left\vert \left\vert \sum_{i=1}^{q}\frac{{\func{F}_{\delta ,T_{\delta
					,i}^{-1}(\gamma )}{_\ast }}\mu |_{T_{0,i}^{-1}(\gamma )}\chi _{B_{i}}}{%
		|T_{\delta ,i}^{\prime }(T_{\delta ,i}^{-1}(\gamma ))|}-\sum_{i=1}^{q}\frac{{%
			\func{F}_{\delta ,T_{\delta ,i}^{-1}(\gamma )}{_\ast }}\mu |_{T_{\delta
				,i}^{-1}(\gamma )}\chi _{B_{i}}}{|T_{\delta ,i}^{\prime }(T_{\delta
			,i}^{-1}(\gamma ))|}\right\vert \right\vert _{W}. 
	\]
	
	Let us make the change of variable $\gamma =T_{\delta ,i}(\beta )$. 
	\begin{eqnarray*}
		\int_{I}II(\gamma )~dm_1(\gamma ) &=&\int_{I}\left\vert \left\vert
		\sum_{i=1}^{q}\frac{{\func{F}_{\delta ,T_{\delta ,i}^{-1}(\gamma )}{_\ast }}\mu
			|_{T_{0,i}^{-1}(\gamma )}\chi _{B_{i}}}{|T_{\delta ,i}^{\prime }(T_{\delta
				,i}^{-1}(\gamma ))|}-\sum_{i=1}^{q}\frac{{\func{F}_{\delta ,T_{\delta
						,i}^{-1}(\gamma )}{_\ast }}\mu |_{T_{\delta ,i}^{-1}(\gamma )}\chi _{B_{i}}}{%
			|T_{\delta ,i}^{\prime }(T_{\delta ,i}^{-1}(\gamma ))|}\right\vert
		\right\vert _{W}~dm_1(\gamma ) \\
		&\leq &\sum_{i=1}^{q}\int_{B_{i}}\frac{1}{|T_{\delta ,i}^{\prime }(T_{\delta
				,i}^{-1}(\gamma ))|}\left\vert \left\vert {\func{F}_{\delta ,T_{\delta
					,i}^{-1}(\gamma )}{_\ast }}\left( \mu |_{T_{0,i}^{-1}(\gamma )}-\mu
		|_{T_{\delta ,i}^{-1}(\gamma )}\right) \right\vert \right\vert _{W}dm_1(\gamma
		) \\
		&\leq &\sum_{i=1}^{q}\int_{B_{i}}\frac{1}{|T_{\delta ,i}^{\prime }(T_{\delta
				,i}^{-1}(\gamma ))|}\left\vert \left\vert \mu |_{T_{0,i}^{-1}(\gamma )}-\mu
		|_{T_{\delta ,i}^{-1}(\gamma )}\right\vert \right\vert _{W}dm_1(\gamma ) \\
		&\leq &\sum_{i=1}^{q}\int_{T_{\delta ,i}^{-1}(B_{i})}\left\vert \left\vert
		\mu |_{T_{0,i}^{-1}(T_{\delta ,i}(\beta ))}-\mu |_{\beta }\right\vert
		\right\vert _{W}dm_1(\beta ).
	\end{eqnarray*} Hence, by (\ref{3times}) 
	\[
	\int_{I}II(\gamma )~dm_1(\gamma )\leq \int \sup_{x,y\in B(\beta ,\delta
		)}(||\mu |_{x}-\mu |_{y}||_{W})dm_1(\beta ) 
	\]%
	and then 
	\[
	\int_{I}II(\gamma )~dm_1(\gamma )\leq 2\delta (M_{2}+1). 
	\]%
	Summing all, the statement is proved.

\end{proof}


\subsubsection{Proof of Theorem \ref{htyttigu}}Before to stablish Theorem \ref{htyttigu}, we need to prove the following proposition.

\begin{proposition}\label{rr}
	Let $\{F_\delta\}_{\delta \in [0,1)}$ be a Uniform BV Lorenz-like family and let $\{\func {F_\delta{_*}}\}_{\delta \in [0,1)}$ be the induced family of transfer operators. Then, $\{\func {F_\delta{_*}}\}_{\delta \in [0,1)}$ is a uniform family of operators with weak space $(\mathcal{L}^{1}, || \cdot ||_1)$ and strong space $(\mathcal{BV}_{1,1}, ||\cdot ||_{1,1})$.
\end{proposition}
\begin{proof}
	To prove UF1, note that, by (UBV1) there exist $0 < \alpha _1 < 1$ and $\overline{D}>0$ s.t. for all $\mu \in \mathcal{BV}_{1,1}$ and for all $\delta $ it holds $||\func{F_\delta{_\ast }}^n  \mu||_{1,1} \leq \overline{D} \alpha _1 ^n ||\mu||_{1,1} + \overline{D}||\mu||_1, \ \textnormal{for all} \ \ n \geq 1.$ Indeed, by Lemma \ref{l1} we have 
	\begin{eqnarray*}
		||\func{F {_\delta} {_*} ^n}  \mu||_{1,1} &=& |\func {P}_{T_{\delta }} ^n \phi _x |_{1,1} + ||\func{F_\delta{_\ast }} ^n \mu||_{1}  \\&\leq & D \lambda ^n |\phi _x|_{1,1} + D|\phi _x|_1 + || \mu||_{1} \\&\leq & D \lambda ^n ||\mu||_{1,1} + (D +1) || \mu||_{1}.
	\end{eqnarray*}Therefore, if $f_{\delta }$ is a fixed probability measure for the operator $\func{F_\delta {_*}}$, by the above inequality we get UF1 with $M=D+1$.
	
	Proposition \ref{UF2ass} and Proposition \ref{thshgf} immediately give UF2. The items UF3 and UF4 follow, respectively, from Proposition \ref{5.10} and Lemma \ref{l1} applied to each $F_\delta$.
\end{proof}

Once this is done, we apply the above result together with Proposition \ref{dlogd} and the proof of Theorem \ref{htyttigu} is established.


\section{Appendix 1: Proof of Propositions \protect\ref{propvar} and \ref{thshgf}}\label{appendix1}

In this section, we obtain Proposition \ref{propvar} as a particular case of Theorem \ref{las123456hh}. We also prove Proposition \ref{thshgf}.

Note that, for all $\mu \in \mathcal{BV}^+$ it holds $||\mu||_1 = |\phi _x|_1$ and $||\mu||_\infty = |\phi _x| _\infty$, where $\phi _x = \dfrac{d \pi _x {_\ast} \mu}{dm}$. We also remark, for each $\mu \in \mathcal{BV}^+$ we have $\phi _x \in BV_{1,1}$.



For a measurable map $F:[0,1]^2 \longrightarrow [0,1]^2 $, of the type $F(x,y)=(T(x),G(x,y))$, and a given $\gamma \in \mathcal{F}^s (\gamma = \{x\}\times [0,1])$, consider the function $F_{\gamma }:[0,1]\longrightarrow [0,1]$, defined by equation (\ref {ritiruwt}).

\begin{definition}
	Consider a function $f:[0,1]^2 \longrightarrow \mathbb{R}$ and let $x_1 \leq \cdots \leq x_n$ and $y_1 \leq \cdots \leq y_n$ be such that $(x_i)_{i=1} ^{n} \subset I$ and $(y_i)_{i=1} ^{n} \subset I$. We define $\var ^\diamond (f,(x_i)_{i=1} ^{n}, (y_i)_{i=1} ^{n})$ by $$\var ^\diamond (f,(x_i)_{i=1} ^{n}, (y_i)_{i=1} ^{n}):= \sum _{i=1} ^{n-1} {|f(x_{i +1},y_i) - f(x_{i},y_i)|},$$and 
	
	\begin{equation*}
		\var ^\diamond (f) := \sup _{(x_i)_{i=1} ^{n}, (y_i)_{i=1} ^{n}}{\var ^\diamond (f,(x_i)_{i=1} ^{n}, (y_i)_{i=1} ^{n})}.
	\end{equation*} If $\eta \subset I$ is an interval, we define $\var _\eta ^\diamond (f) = \var ^\diamond (f|_{\overline{\eta} \times I})$, where $\overline{\eta}$ is the closure of $\eta$.
\end{definition}

Since preliminaries results are necessary, we postponed the proof of the next theorem to the end of the section.

\begin{theorem}
	Let $F(x,y)=(T(x),G(x,y))$ be a measurable transformation such that
	
	\begin{enumerate}
		\item $\var ^\diamond (G) < \infty$
		
		\item $F$ satisfy property $G1$ (hence is
		uniformly contracting on each leaf $\gamma $ with rate of contraction $\alpha$);
		
		\item $T:[0,1]\rightarrow [0,1]$ is a piecewise expanding map satisfying the assumptions given in the definition \ref{defpiece123C1}.
	\end{enumerate}Then, there are $K_{0}$ and $0< \lambda _{0}< 1$ such that for all path $\Gamma_\mu$, where $\mu \in \mathcal{BV}^+$, and all $n\geq 1$ it holds 
	\begin{equation*}
		\Var(\Gamma_{\operatorname{F{_ \ast}^n \mu} } )\leq K_{0}\lambda _{0}^{n}\Var(\Gamma_ \mu )+K_{0}{| \phi _x|_{1,1}}. 
	\end{equation*}
	\label{las123456hh}
\end{theorem}

\begin{remark}\label{jjsdgjf}
	If $F_L$ is a BV Lorenz-like map (definition \ref{lorenzlikemap}), a straightforward computation yields $$\var ^\diamond(G_L)\leq H,$$where $H$ comes from equation (\ref{rty}). This shows that Proposition \ref{propvar} is a direct consequence of Theorem \ref{las123456hh}.
\end{remark}

\subsection{Lasota-Yorke Inequality for positive measures}

Henceforth, we fix a positive measure $\mu \in \mathcal{BV}^+ \subset \mathcal{AB}$ and a path, $\Gamma^\omega_{\mu }$, which represents $\mu $ (i.e. a
pair $(\{\mu _{\gamma }\}_{\gamma },\phi _{x})$ s.t. $\Gamma^\omega_{\mu }(\gamma) := \mu |_\gamma$). To simplify, we will denote the path $\Gamma_{\mu }^\omega \in \Gamma_{\mu }$, just by $\Gamma_{\mu }$.

\begin{remark}
	Consider $T:[0,1] \longrightarrow [0,1]$ a piecewise expanding map from definition \ref{defpiece123C1} and $g_i = \dfrac{1}{|{T_{i}}'|}$.	For all $n \geq 1$, let $\mathcal{P}^{(n)}$ be the partition of $I$ s.t. $\mathcal{P}^{(n)}(x) = \mathcal{P}^{(n)}(y)$ if and only if $\mathcal{P}^{(1)}(T^j (x)) = \mathcal{P}^{(1)}(T^j(y))$ for all $j = 0, \cdots , n-1$, where $\mathcal{P}^{(1)} = \mathcal{P}$ (see definition \ref{defpiece123C1}). Given $P \in \mathcal{P}^{(n)}$, define $g_P ^{(n)} = \frac{1}{|{T^n}{ ^\prime}|_P|}$. Item 2) implies that there exists $C_1 > 0$ and $0 < \theta  <1$ s.t. 
	
	\begin{equation}
		\displaystyle{\sup \{g_P	^{(n)}\}\leq C_1 \theta ^n, \textnormal{ for all} \ P \in \mathcal{P}^{(n)} \ \textnormal{and all} \ n\geq 1.} 
		\label{item2}
	\end{equation} Moreover, equation (\ref{item2}) and some basic properties of real valued $BV$ functions imply (see \cite{V}, page 41, equation (3.1)) there exists $\lambda _{2}\in (\theta,1)$ and $C_{2}>0$ such that
	
	\begin{equation*}
		\var(g_{P }^{(n)})\leq C_{2}\lambda _{2}^{n},\ \textnormal{for all} \ P
		\in \mathcal{P}^{(n)} \ \textnormal{and all} \ n\geq 1.  
	\end{equation*}Then, there is an iterate of $F$, $\widetilde{F}:=F^k$, such that $T^k$ satisfies 
	
	\begin{equation}
		\beta_k: =\var g_{P }^{(k)} + 3 \sup g_{P}^{(k)}   <1,  \ \forall P \in \mathcal{P}^{(k)}. 
		\label{est1}
	\end{equation}We also remark that $G^k:= \pi _y \circ F^k$ also satisfies  
	\begin{equation}\label{finisv}
		\var ^\diamond (G^k) <\infty.
	\end{equation} Next lemma provides equation (\ref{finisv}) and its proof can be found in \cite{AGP}.
	\label{est2}
\end{remark}

\begin{lemma}If $F$ satisfy definition \ref{lorenzlikemap}, then for all $n\geq 1$ and all $f:[0,1]^2\longrightarrow \mathbb{R}$ such that $$\sup _{x,y_1,y_2 \in [0,1]} {\frac{|f(x,y_2) - f(x,y_1)|}{|y_2-y_1|}}< \infty$$ and $$|f|_\infty < \infty,$$ it holds \footnote{$ |f|_{lip'}= |f|_\infty + Lip _y(f)$, where $Lip _y(f) = \sup _{x,y_1,y_2 \in [0,1]} {\frac{|f(x,y_2) - f(x,y_1)|}{|y_2-y_1|}}$.} $$\var ^\diamond (f\circ F^n) \leq q^n \var ^\diamond (f) + \sum _{i=1}^{n-1}{q^i}\left( \var ^\diamond (G)|f|_{lip '} + 2q|f|_\infty\right),$$where $q$ is the number of branches of $T$ ($q:=\# \mathcal{P} $).
	\label{uthd}
\end{lemma}

Recalling equation (\ref{ritiruwt}), set
\begin{equation}
	\Gamma_{\mu_{\func{F}}} (\gamma ):= \func {F _\gamma {_\ast }}\Gamma_\mu(
	\gamma ).  \label{niceformulan1}
\end{equation}
With the above notation and following the strategy of the proof of Lemma \ref{transformula}, the path $\Gamma_{\func{F {_\ast}}\mu}$, defined on a full measure set by
\begin{equation*}
	\Gamma_{\func{F {_\ast}\mu}} (\gamma) =\sum_{i=1} ^{q}{\left(
		g_{i}\cdot {\Gamma_{\mu_{\func{F}}}}\right) \circ T_{L_i}^{-1}%
		(\gamma )\cdot \chi _{T_L(P_i )}(\gamma )}, \ \textnormal{where} \ g_i = \dfrac{1}{{|T_{L_i}'|}}, 
\end{equation*}represents the measure $\func{F{_\ast}} \mu$.

By equations (\ref{weak1}) and (\ref{niceformulan1}), it holds 
\begin{equation*}
	||\Gamma_{\mu_{\func{F}}} (\gamma )||_W \leq ||\Gamma_\mu (\gamma)||_W,
\end{equation*} for $m$-a.e. $\gamma \in I$. Then we have the following.

\begin{lemma}
	Let $\gamma_1$ and $\gamma _2$ be two leaves such that $G(\gamma _i, \cdot ): I \longrightarrow I$ is a contraction, $i=1,2$. Then for every path $\Gamma _\mu
	$, where $\mu \in \mathcal{AB}$, it holds
	
	\begin{equation*}
		||\Gamma_{\mu_{\func{F}}} (\gamma _1 ) - \Gamma_{\mu_{\func{F}}} (\gamma _2) ||_W \leq ||\Gamma_\mu (\gamma _1) - \Gamma_\mu (\gamma _2)
		||_W + |G(\gamma_1 , y_0) - G(\gamma _2,y_0)| |\phi _x|_\infty,
	\end{equation*}
	for some $y_0 \in I$.
	\label{doesnotvary12}
\end{lemma}

\begin{proof}
	Consider $g$ such that $|g|_{\infty }\leq 1$ and $Lip(g)\leq 1$ , and
	observe that since $G _{\gamma _1}-G _{\gamma _2}:I \longrightarrow I$ is continuous, it holds $$\sup _I {\left\vert G(\gamma _1,y)-G (\gamma _2,y)\right\vert} = \left\vert G (\gamma _1,y_0)-G(\gamma _2,y_0)\right\vert,$$ for some $y_0 \in I$. Moreover, by equations (\ref{weak1}) and (\ref{niceformulan1}), we have

	\begin{eqnarray*}
		\left\vert \int {g}d\func \Gamma_{\mu_{\func{F}}} (\gamma _1)-\int {g%
		}d\Gamma_{\mu_{\func{F}}} (\gamma _2)\right\vert 
		&=& \left\vert \int {g}d\func {F  _{\gamma _1 }{_\ast} }\Gamma _\mu(
		\gamma _1 )-\int {g%
		}d\func {F _{\gamma _2} {_\ast} }\Gamma _\mu(
		\gamma _2)\right\vert 
		\\&\leq &\left\vert \int {g}d\func {F  _{\gamma _1 }{_\ast} }\Gamma _\mu(
		\gamma _1 )-  \int {g}d\func {F  _{\gamma _1 }{_\ast} }\Gamma _\mu(
		\gamma _2 )\right\vert \\
		&+&\left\vert \int {g}d\func {F  _{\gamma _1 }{_\ast} }\Gamma _\mu(
		\gamma _2 )-  \int {g}d\func {F  _{\gamma _2 }{_\ast} }\Gamma _\mu(
		\gamma _2 )\right\vert
		\\&\leq &\left\vert \left\vert \func {F  }_{\gamma _1 }{_\ast} (\Gamma _\mu(\gamma _1 )-\Gamma _\mu(\gamma _2 ))\right\vert \right\vert _{W} \\
		&+&\int {\left\vert g(F_{\gamma _{1}})-g(F_{\gamma _{2}})\right\vert }d\mu |_{\gamma _{2}}
		\\&\leq &\left\vert \left\vert \Gamma _\mu(\gamma _1 )-\Gamma _\mu(\gamma _2 )\right\vert \right\vert _{W}
		\\&+&\int {\left\vert G(\gamma _1,y)-G (\gamma _2,y)\right\vert }d\mu |_{\gamma _{2} (y)}
		\\&\leq &\left\vert \left\vert \Gamma _\mu(\gamma _1 )-\Gamma _\mu(\gamma _2 )\right\vert \right\vert _{W}
		\\&+&\sup _I {\left\vert G(\gamma _1, y)-G (\gamma _2, y)\right\vert}\int {1 }d\mu |_{\gamma _{2} (y)}
		\\&= &\left\vert \left\vert \Gamma _\mu(\gamma _1 )-\Gamma _\mu(\gamma _2 )\right\vert \right\vert _{W} \\&+&\left\vert G(\gamma _1,y_0)-G(\gamma _2, y_0)\right\vert |\phi _x|_\infty.
	\end{eqnarray*}%
	Taking the supremum over $g$, we finish the proof.
\end{proof}

The proofs of the next three lemmas are straightforward and analogous to the
one dimensional $BV$ functions. So, we omit them (details can be found in \cite{L}).

\begin{lemma}
	Given paths $\Gamma_{\mu _0}, \Gamma_{\mu _1} \ \mathnormal{and} \ \Gamma_{\mu _2}$ (where $%
	\Gamma_{\mu _i} (\gamma) = \mu _i |_\gamma$) representing the positive measures $\mu_0, \mu _1, \mu _2 \in \mathcal{BV}^+$ respectively, a function 
	$\varphi : I \longrightarrow \mathbb{R}$, an homeomorphism 
	$h:\eta \subset I
	\longrightarrow h(\eta) \subset I$ and a subinterval $\eta \subset I$, then
	the following properties hold 
	
	\begin{enumerate}
		\item[P1)] If $\mathcal{P} $ is a partition of $I$ by intervals $\eta$, then 
		\begin{equation*}
			\Var (\Gamma_{\mu _0}) = \sum _\eta {\Var_{\overline{\eta}} (\Gamma_{\mu_0})};
		\end{equation*}
		
		\item[P2)] $\Var _{\overline{\eta}} (\Gamma_{\mu _1} + \Gamma_{\mu _2}) \leq \Var _{%
			\overline{\eta}} (\Gamma_{\mu _1} ) + \Var _{\overline{\eta}}(\Gamma_{ \mu _2})$
		
		\item[P3)] $\Var_{\overline{\eta}} (\varphi  \Gamma_ {\mu_0}) \leq \left( \sup _{\overline{\eta}%
		} |\varphi| \right) \left( \Var _{\overline{\eta}} (\Gamma_{\mu_0} ) \right)
		+ \left( \sup _ {\gamma \in _{\overline{\eta}}} ||\Gamma_{\mu_0} (\gamma)||_{W}
		\right) \var _{\overline{\eta}} (\varphi )$
		
		\item[P4)] $\Var _{\overline{\eta}}( \Gamma_ {\mu _0} \circ h) = \Var _{\overline{%
				h(\eta) }}( \Gamma_ {\mu_0})$.
	\end{enumerate}
\end{lemma}

\begin{lemma}
	For every path $\Gamma_\mu$, $\mu \in \mathcal{AB}$ and an interval $\eta \subset I$, it holds 
	\begin{eqnarray*}
		\sup _{\gamma \in {\overline{\eta}}} ||\Gamma_\mu ( \gamma)|| _W &\leq &  \Var _{\overline{\eta}}(\Gamma _\mu ) + \dfrac{1}{m({\overline{\eta}})}
		\int _{\overline{\eta}} {||\Gamma_\mu ( \gamma)|| _W }dm_1 (\gamma),
	\end{eqnarray*}where $\overline{\eta}$ is the closure of $\eta$.
	\label{var123}
\end{lemma}

A straightforward application of Lemma \ref{doesnotvary12} yields the following.
\begin{lemma}\label{huthu}
	For all $\Gamma_\mu$, where $\mu \in \mathcal{BV}^+$, and all $P \in \mathcal{P}$ it holds 
	\begin{equation*}
		\Var  _{\overline{P}} (\Gamma_{\mu_{\func{F}}}) \leq \Var  _{\overline{P}} (\Gamma_{\mu}) + \var  _{\overline{P}} ^\diamond (G) |\phi _x|_\infty.
	\end{equation*}
\end{lemma}


\begin{lemma}
	For all path $\Gamma_\mu$, where $\mu \in \mathcal{BV}^+$, it holds
	
	\begin{equation*}
		\Var(\Gamma_{\func{F{_\ast}} \mu}) \leq \sum _{i=1}^q %
		\left[ \var _{\overline{P_i}} (g_i ) + 2\sup _{\overline{P_i}} {g_i}\right]
		\cdot \sup _{\gamma \in {\overline{P_i}}} {||\Gamma _\mu (\gamma) ||_ W} + \sup _{\overline{P_i}}{%
			g_i} \cdot \Var _{\overline{P_i}}(\Gamma_{\mu_{\func{F}}}),
	\end{equation*}where $\Gamma_{\mu_{\func{F}}}$ is defined by equation (\ref{niceformulan1}). 
	\label{lemma123}
\end{lemma}

\begin{proof}
	Using the properties P2, P3, P4, $\displaystyle{\sup _{\gamma\in {\overline{P_i}}} {||
			\Gamma_{\mu_{\func{F}}}(\gamma)||_{W}} \leq \sup _{\gamma \in {\overline{P_i}}} 
		{|| \Gamma_\mu (\gamma)||_{W}} }$ and $\displaystyle{\sup _{\gamma\in {\overline{P_i}}} {|g_i|} = \sup _{\gamma\in {\overline{P_i}}} {g_i}}$,
	we have 
	\begin{eqnarray*}
		\Var (\Gamma_{\operatorname{F{_\ast}  \mu }}) &\leq & \sum _{i=1}^q {\Var _{\overline{T_i(P _i)}} \left[ \left( g_i \cdot \Gamma_{\mu_{\func{F}}} \right) \circ T _i^{-1} \cdot \chi
			_{T(P_i)}\right] } 
		\\&\leq & \sum _{i=1}^q \Var _{\overline{T_i(P_i) }%
		}\left[ \left( g_i \cdot \Gamma_{\mu_{\operatorname{F}}} \right) \circ T_i^{-1} \right] \cdot \sup {|\chi _{T(P_i)}|} \\
		&+& \sum _{i=1}^q \sup _{\overline{{T_i(P_i)}}} 
		{||\left( g_i \cdot \Gamma_{\mu_{\operatorname{F}}} \right) \circ T_i^{-1}||_{W}} \cdot \var (\chi _{T(P_i)}) 
		\\&\leq & \sum _{i=1}^q {\Var _{{\overline{P_i}}} \left(
			g_i \cdot \Gamma_{\mu_{\operatorname{F}}} \right) + 2\cdot \sup _{T_i(P_i)} {||\left( g_i \cdot \Gamma_{\mu_{\operatorname{F}}} \right)
				\circ T_i^{-1}||_{W}}} 
		\\&\leq & \sum _{i=1}^q \var _{\overline{P_i}} \left(
		g_i \right) \cdot \sup _{\overline{P _i}} {||\Gamma_{\mu_{\operatorname{F}}}||_{W}} + \Var
		_{{\overline{P _i}}} (\Gamma_{\mu_{\operatorname{F}}})\cdot \sup _{\overline{P _i}} {%
			|g_i|} \\
		&+& 2\cdot \sum _{i=1}^q \sup _{\overline{P _i}} {%
			|g_i |}  \sup _{\overline{P _i}} {||\Gamma_{\mu_{\operatorname{F}}}  ||_{W}}  		
		\\&\leq & \sum _{i=1}^q \var _{\overline{P _i}} \left(
		g_i \right) \cdot \sup _{\gamma\in {\overline{P _i}}} {||\Gamma_\mu (\gamma) ||_{W} } + \Var _{\overline{P _i}} (\Gamma_{\mu_{\operatorname{F}}} )\cdot \sup _{%
			\overline{P _i}} {|g_i|} \\
		&+& 2\cdot \sum _{i=1}^q \sup_{\gamma \in \overline{P _i}} {|| \Gamma_\mu (\gamma)||_{W}}\cdot \sup _{\overline{P _i}} | g_i| 
		\\&\leq & \sum _{i=1}^q \left[ \var _{\overline{P _i}}
		(g_i ) + 2\sup _{\overline{P _i}} {g_i}\right] \cdot
		\sup _{\gamma \in {\overline{P _i}}} {||\Gamma_\mu (\gamma) ||_ {W}} + \sup _{%
			\overline{P _i}} {g_i} \cdot \Var _{\overline{P _i}} (\Gamma_{\mu_{\operatorname{F}}}).
	\end{eqnarray*}
\end{proof}

\begin{lemma}\label{fdlkfldfl}
	For all path $\Gamma_\mu$, where $\mu \in \mathcal{BV}^+$, it holds
	
	\begin{equation}\label{lasotaingt234}
		\Var(\Gamma_{\func{F {_\ast}\mu}}) \leq \beta  \Var (\Gamma_{\mu}) +
		K_3 |\phi _x|_{1,1}.
	\end{equation}%
	Where $$\beta := \max _{i=1, \cdots, q} \{\var _{\overline{P _i}} (g_i) + 3\sup_{\overline{P_i}} {g_i} \}$$ and $$K_3 = \max_{i=1, \cdots, q} \{\sup _{\overline{P_i}} g_i\} \var  ^\diamond (G) +  \max _{i=1, \cdots, q} \left \lbrace\dfrac{ \var _{\overline{P _i}} (g_i) + 2\sup_{\overline{P_i}} {g_i}}{m(P _i)} \right \rbrace.$$ \label{almost123456}
\end{lemma}

\begin{proof}
	By lemma \ref{huthu}, remark \ref{var123}, lemma \ref{lemma123}, P1, equation (\ref{est1}) of remark \ref{est2} and by $\sum _{i=1} ^{q} {\var _{\overline{P }_i} ^\diamond {G}} = \var ^\diamond (G)$, we get
	
	\begin{eqnarray*}
		\Var(\Gamma_{\operatorname{F{_\ast} \mu} }) 
		&\leq & \sum _{i=1}^q %
		\left[ \var _{\overline{P _i}} (g_i) + 2\sup_{\overline{P_i}} {g_i}\right]
		\sup _{\gamma \in {\overline{P_i}}} {||\mu |_\gamma ||_ W} + \sup_{\overline{P_i}} {g_i} \cdot \Var_{\overline{P_i}}(\Gamma_{\mu_{\operatorname{F}}})
		\\&\leq & \sum _{i=1}^q \left[ \var _{\overline{P _i}} (g_i) + 2\sup_{\overline{P_i}} {g_i}\right] \left( \Var _{\overline{P_i}}(\Gamma_{\mu} ) + \dfrac{1}{m_1({%
				\overline{P_i}})} \int _{\overline{P_i}} {||\mu |_\gamma ||_W }dm_1 (\gamma)
		\right) \\
		&+& \sum _{i=1}^q \sup_{\overline{P_i}} {g_i} \left(\Var  _{\overline{P_i}} (\Gamma_{\mu}) + \var  _{\overline{P_i}} ^\diamond (G) |\phi _x|_\infty \right)
		\\&\leq & \sum _{i=1}^q \left[ \var _{\overline{P _i}} (g_i) + 3\sup_{\overline{P_i}} {g_i}\right]  \Var _{\overline{P_i}}(\Gamma_{\mu} ) 
		\\&+& \sum _{i=1}^q \left[ \var _{\overline{P _i}} (g_i) + 2\sup_{\overline{P_i}} {g_i}\right] \dfrac{1}{m_1({%
				\overline{P_i}})} \int _{\overline{P_i}} {||\mu |_\gamma ||_W }dm_1 (\gamma)
		\\
		&+& |\phi _x|_\infty \max _{i=1, \cdots, q } \{\sup_{\overline{P_i}} {g_i}\}   \var  ^\diamond (G)  
		\\&\leq & \sum _{i=1}^q \left[ \var _{\overline{P _i}} (g_i) + 3\sup_{\overline{P_i}} {g_i}\right]  \Var _{\overline{P_i}}(\Gamma_{\mu} ) 
		\\&+&   \max _{i=1, \cdots, q } \{\dfrac{\var _{\overline{P _i}} (g_i) + 2\sup_{\overline{P_i}} {g_i} }{m_1({%
				\overline{P_i}})} \}|\phi _x|_1
		\\
		&+& |\phi _x|_\infty \max _{i=1, \cdots, q } \{\sup_{\overline{P_i}} {g_i}\}   \var  ^\diamond (G)  
		\\&\leq & \beta\Var ( \Gamma_{\mu} ) + K_3
		|\phi _x|_\infty
		\\&\leq & \beta \Var ( \Gamma_{\mu} ) + K_3
		|\phi _x|_{1,1} .
	\end{eqnarray*}
\end{proof}

\begin{remark}
	Remember that, the coefficients of inequality (\ref{lasotaingt234}) are given by the formulas 
	\begin{equation*}
		\beta = \max _{i} \{\var _{\overline{P _i}} (g_i) + 3\sup_{\overline{P_i}} {g_i} \}
	\end{equation*} and
	\begin{equation*}
		K_3 = \max_{i} \{\sup _{\overline{P_i}} g_i\} \var  ^\diamond (G) +  \max _{i} \left \lbrace\dfrac{ \var _{\overline{P _i}} (g_i) + 2\sup_{\overline{P_i}} {g_i}}{m_1(P _i)} \right \rbrace.
	\end{equation*}We will use these expressions in the next result and later on.
	\label{remmm1}
\end{remark}

From Lemma \ref{fdlkfldfl} and taking the infimum over the paths $\Gamma _\mu$ we have the following. 
\begin{corollary}\label{las123}
	If $F:[0,1]^2\longrightarrow[0,1]^2$ satisfies all the hypothesis of Theorem {\ref{las123456hh}}. Then, for all $\mu \in \mathcal{BV}^+$, it holds
	
	\begin{equation*}
		\Var(\operatorname{F{_\ast} \mu} ) \leq  \beta \Var ( \Gamma_\mu ) + K_3 |\phi _x|_{1,1},  
	\end{equation*}where $\beta$ and $K_3$ were given by Remark \ref{remmm1}.
\end{corollary}

\begin{proposition}\label{rafhssh}
	If $F:[0,1]^2\longrightarrow[0,1]^2$ satisfies all the hypothesis of Theorem {\ref{las123456hh}}. Then, there exist $k \in \mathbb{N}$, $0<\beta_k<1$ and $C_k>0$ such that for all path $\Gamma_\mu$, where $\mu \in \mathcal{BV}^+$, it holds
	
	\begin{equation*}
		\Var(\Gamma_{\operatorname{F{_\ast}^k \mu} }) \leq  \beta_k \Var ( \Gamma_\mu ) + C_k |\phi _x|_{1,1}.  
	\end{equation*}
	\label{las123fjkghdhfg}
\end{proposition}
\begin{proof}
	The proof is a straightforward consequence of the above Remark \ref{remmm1} and Remark \ref{est2}, where $\beta _k$ was defined by equation (\ref{est1}).
\end{proof}

\begin{proposition}
	If $F:[0,1]^2\longrightarrow[0,1]^2$ satisfies all the hypothesis of Theorem {\ref{las123456hh}}. Then, there exist $k \in \mathbb{N}$, $C_{0}$ and $0< \beta _k< 1$
	such that for all path $\Gamma_\mu$, where $\mu \in \mathcal{BV}^+$, and all $n\geq 1$ it holds 
	\begin{equation*}
		\Var(\Gamma_{\operatorname{F{_\ast}^{kn} \mu} })\leq C_{0}\beta _k^{n}\Var(\Gamma_\mu)+C_{0}{%
			| \phi _x|_{1,1}}. 
	\end{equation*}
	\label{las123456hhjdshfkjshkdjfh}
\end{proposition}

\begin{proof} Inequality (\ref{lasotaiiii2}) gives us
	\begin{equation*}
		|\func{P}_{T}^{n}f|_{1,1} \leq B_3 \beta _2 ^n | f|_{1,1}  +  C_2|f|_{1}, \ \ \forall n , \ \ \forall f \in BV_{1,1},
	\end{equation*}for $B_3, C_2 > 0$ and $0<\beta _2 <1$. Then, since $|f|_{1} \leq | f|_{1,1}$,  it holds
	
	\begin{equation}
		|\func{P}_{T}^{n}f|_{1,1} \leq K_2 | f|_{1,1}, \ \ \forall n , \ \ \forall f \in BV_{1,1},
		\label{boundingu} 
	\end{equation}where 
	
	\begin{equation*}
		K_2 = B_3 + C_2. 
	\end{equation*}In particular, inequality (\ref{boundingu}) holds if we replace $f$ by $\phi _x = \dfrac{d(\func {\pi _x {_\ast} \mu })}{dm_1}$ for each $\mu \in \mathcal{BV}^+$.
	
	By inequality (\ref {boundingu}), Proposition \ref{rafhssh} and a straightforward induction we have

	\begin{equation*}
		\Var(\Gamma_{\operatorname{F{_\ast}^{kn} \mu} })\leq \beta_k^{n}\Var(\Gamma_\mu )+ C_k \max \{K_2,1\} \sum _{i=0} ^{n-1}{\beta _k ^i} {%
			|\phi _x|_{1,1}}, \ \ \forall n \geq 0.  
	\end{equation*} We finish the proof by setting 
	
	\begin{equation*}
		C_0 := \max \left \lbrace1,\dfrac{C_k \max \{K_2, 1\}}{1- \beta _k}\right \rbrace.
	\end{equation*}
	
\end{proof}

\begin{proof}(of Theorem \ref{las123456hh})
	
	Let $k \in \mathbb{N}$ be from Proposition \ref{las123456hhjdshfkjshkdjfh}. For a given $n$, we set $n=kq_n + r_n$, where $0 \leq r_n < k$. Applying Proposition \ref{fdlkfldfl} and iterating $r_n$ times the inequality (\ref{lasotaingt234}) we have

	\begin{equation}\label{igjgj}
		\Var(\Gamma_{\operatorname{F{_\ast}^{r_n}\mu} }) \leq \max _{i=0, \cdots, k} \{\beta ^i\} \Var(\Gamma_\mu) + K_3K_2   \sum _{j=0} ^{k} \beta ^j |\phi _x|_{1,1}, 
	\end{equation}where $K_2$ was defined in equation (\ref {boundingu}). Thus, by Proposition \ref{las123456hhjdshfkjshkdjfh} and the above inequality (\ref{igjgj}), we have

	\begin{eqnarray*}
		\Var(\Gamma_{\operatorname{F{_\ast}^n \mu} })&=&\Var(\Gamma_{\operatorname{F{_ \ast}^{k q_n +r_n} \mu} }) \\ &\leq&  C_{0}\beta _k^{q_n}\Var(\Gamma_{\operatorname{F{_\ast}^{r_n} \mu} })+C_{0}{| \phi _x|_{1,1}}
		\\ &\leq&  C_{0} \max _{i=0, \cdots, k} \{\beta ^i\}\beta _k^{q_n}\Var( \Gamma_\mu )+  \left[C_{0} \beta _k^{q_n}K_3K_2 \sum _{j=0} ^{k} \beta ^j  +C_0 \right]{| \phi _x|_{1,1}}
		\\ &\leq &  C_{0} \max _{i=0, \cdots, k} \{\beta ^i\}\beta _k^{\frac{n-r_n}{k}}\Var( \Gamma_\mu )+  \left[C_{0} K_3K_2 \sum _{j=0} ^{k} \beta ^j  +C_0 \right]{| \phi _x|_{1,1}}
		\\ &\leq &  K_{0} \lambda _0 ^n\Var(\Gamma_\mu )+ K_0 | \phi _x|_{1,1},
	\end{eqnarray*}where 
	
	\begin{equation}\label{hdjfdjhuf}
		K_0 = \max \left\{ \frac{C_{0} \max _{i=0, \cdots, k}\{\beta^i\} }{\beta _k}, C_{0} K_3K_2 \sum _{j=0} ^{k} \beta ^j  +C_0\right\}
	\end{equation}and
	
	\begin{equation}
		\lambda _0 = (\beta _k)^{\frac{1}{k}}.
	\end{equation}

\end{proof}

\subsubsection{Uniform Lasota-Yorke like inequality}
\begin{proposition} If $\{F_\delta\}_{\delta \in [0,1)}$ is a Uniform BV Lorenz-like family. Then, there exist uniform constants $\beta_u>0$ and $K_u>0$ such that for every $\mu \in \mathcal{BV}^+$, it holds
	
	\begin{equation}
		\Var(\func{F_\delta{_\ast} \mu} ) \leq  \beta_u \Var ( \mu ) + K_u |\phi _x|_{1,1}, \ \forall \delta \in [0,1).  \label{lasotaingt234dffggdgh}
	\end{equation}

	\label{las123rtryrdfd}
\end{proposition}

\begin{proof}
	Since $\var  ^\diamond (G_\delta) \leq H_\delta$, we can apply Corollary \ref{las123} to each $F_\delta$ to get (see Remark \ref{remmm1})

	\begin{equation*}
		\Var(\func{F_\delta{_\ast} \mu} ) \leq  \beta_\delta \Var ( \mu ) + K_{3, \delta} |\phi _x|_{1,1}, \ \forall \delta \in [0,1),
	\end{equation*}where
	\begin{equation*}
		\beta_\delta = \max _{i=1, \cdots, q} \{\var _{\overline{P _i}} (g_{i\delta}) + 3\sup_{\overline{P_i}} {g_{i\delta}} \}
	\end{equation*} and
	\begin{equation*}K_{3, \delta} = \max_{i} \{\sup _{\overline{P_i}} g_{i,\delta}\} \var  ^\diamond (G_\delta) +  \max _{i} \left \lbrace\dfrac{ \var _{\overline{P _i}} (g_{i,\delta}) + 2\sup_{\overline{P_i}} {g_{i,\delta}}}{m(P _i)} \right \rbrace.
	\end{equation*}Since $\var  ^\diamond (G_\delta) \leq H_\delta$, UBV4 ((2), (3), (4)) yields the existence of uniforms constants $\beta _u:= \sup _{\delta\in [0,1)} \beta _\delta <\infty$ and $K _u:= \sup _{\delta\in [0,1)} K_{3,\delta} <\infty$. 
\end{proof}

Note that, we do not necessarily have $\beta _u<1$. In what follows, we will prove that there exists a uniform $k \in \mathbb{N}$ such that this property is satisfied for the map $F_\delta ^k$, for all $\delta \in [0,1)$. We also remark that, if $\{F_\delta\}_{\delta \in [0,1)}$ is a BV Lorenz-like family, then $F_\delta ^n$ also satisfies the hypothesis of Theorem {\ref{las123456hh}}, for all $n\geq 1$ and all $\delta$, in a way that we can apply Lemma \ref{almost123456} to $F_\delta ^n$, for all $n \geq 1$.

\begin{lemma}\label{jshgd}
	Let $\{T_\delta\}_{\delta \in [0,1)}$ be a family of piecewise expanding maps satisfying Definition \ref{defpiece123C1}, item (1), item (2), item (3) and item (4) of UBV4 (see Definition \ref{UFL}). Then, there is $k$ (which does not depends on $\delta$) such that 
	\begin{equation*}
		\sup _{\delta \in [0,1)}  \max _i \{\var g_{i,\delta} ^{(k)} + 3 \sup g_{i, \delta }^{(k)} \} <1.
	\end{equation*}
\end{lemma}
\begin{proof}(of the Lemma)

	First of all, consider a piecewise expanding map, $T:[0,1]\longrightarrow [0,1]$ satisfying Definition \ref{defpiece123C1}. For all $n \geq 1$, let $\mathcal{P}^{(n)}$ be the partition of $I$ s.t. $\mathcal{P}^{(n)}(x) = \mathcal{P}^{(n)}(y)$ if and only if $\mathcal{P}^{(1)}(T ^j(x)) = \mathcal{P}^{(1)}(T ^j(y))$ for all $j = 0, \cdots , n-1$, where $\mathcal{P}^{(1)} = \mathcal{P}$. For each $n$ define $T^n _i = T^n|{P_i}$ and $g^{(n)}_i = \dfrac{1}{|T _i^{n}|}$, for all $P_i \in \mathcal{P}^{(n)}$.

	Let us consider $n_0$ and $\lambda_1$ from item $2)$ of Definition \ref{defpiece123C1}: $\inf {|{T_{L}^{n_0}}^{\prime }|} \geq \lambda _1 >1$. For a given $n \geq 1$, we write $n=n_0 q_n + r_n$, where $0\leq r_n < n_0$.  Thus, for all $x \in P _i \in \mathcal{P}^{(n)} = \{P_1, \cdots, P_{q(n)}\}$, we have
	
	\begin{eqnarray*}
		\left|{T_i  ^n }'(x)\right| &=&  |\left(T_i^{n_0q_n+r_n} \right)'(x) | 
		\\  &= & |\left(T_i^{n_0q_n} \right)'(T_i^{r_n}(x)) | |(T_i^{r_n})'(x)|
		\\  &\geq & \left(\lambda _1 \right)^{q_n} |(T_i^{r_n})'(x)|.
	\end{eqnarray*}Then,
	
	\begin{eqnarray*}
		g_i ^{(n)} (x) &\leq& \left( \dfrac{1}{\lambda _1}\right)^{q_n} \dfrac{1}{|(T_i^{r_n})'(x)|} 
		\\&\leq& \left( \dfrac{1}{\lambda _1}\right)^{\frac{n}{n_0} -1} \max _{ 0 \leq j \leq n_0}\sup \{g_i\} ^j
		\\&\leq& \lambda _4 ^n C_5,
	\end{eqnarray*}where $\lambda _4 = \dfrac{1}{\sqrt[n_0]{\lambda _1}} <1$ and $C_5 =  \lambda _1\max _{0\leq i \leq q} \{\max _{ 0 \leq j \leq n_0}\sup \{g_i\} ^j \}$. Therefore,
	
	\begin{equation*}
		\sup \{g_i ^{(n)} \} \leq \lambda _4 ^n C_5,
	\end{equation*}for all $n \geq 1$ and all $i$.
	
	Now, set $C_6:=\max \{ C_5, \max _{i}\{\var(g_i) \}\}$. Thus, for all $n \geq 1$ it holds (see \cite{V}, page 41, equation (3.1))

	\begin{equation*}
		\var g_i ^{(n)} \leq \frac{nC_6 ^3}{\lambda _4} \lambda _4^n  ~\forall \delta \in [0,1)~\textnormal{and}~\forall i = 1,\cdots q.
	\end{equation*}Then,
	
	\begin{equation*}
		\var g_i ^{(n)}  \leq C_7 \lambda _5 ^n, \ \forall n \geq 1, \ \forall i,
	\end{equation*}where $\lambda _5 \in (\lambda_4 , 1) $ and $C_7:= \sup _{n \geq 1} \left \{\dfrac{C_ 6 ^3}{\lambda_4 } n \left(\dfrac{\lambda _4 }{\lambda _5} \right) ^n \right \}$.
	
	Now, let us consider a family of piecewise expanding maps, $\{T_\delta\}_{\delta \in [0,1)}$, satisfying Definition \ref{defpiece123C1}, item (1), item (2), item (3) and item (4) of UBV4 (see Definition \ref{UFL}). Applying the above equations to $T_\delta$ we get, for all $i$ and all $\delta$

	\begin{equation*}
		\sup \{g_{i, \delta} ^{(n)} \} \leq \lambda _{4,\delta} ^n C_{5,\delta},
	\end{equation*}

	where $\lambda _{4,\delta}  = \dfrac{1}{\sqrt[n_{0}(\delta)]{\lambda _1(\delta)}}$ and $C_{5,\delta} =  \lambda _1(\delta)\max _{i} \{\max _{ 0 \leq j \leq n_{0}(\delta)}\sup \{g_{i, \delta}\} ^j \}$. By item (1) of UBV4, we get $$\lambda _{4,u}:= \sup _{\delta \in [0,1)} \{ \lambda_{4,\delta}\} = \sup _\delta \{ \dfrac{1}{\sqrt[n_{0}(\delta)]{\lambda _1(\delta)}} \} < 1$$ and by items (1) and (2) of UBV4 it holds $$C_{5,u} := \sup _{\delta \in [0,1)} C_{5,\delta}< \infty.$$Then, we get the uniform estimate 
	
	\begin{equation*}
		\sup \{g_{i, \delta} ^{(n)} \} \leq \lambda _{4,u} ^n C_{5,u},
	\end{equation*}for all $\delta$, all $i$ and all $n \geq 1$.  
	
	By item (2) of UBV4, set $C_{6,u}:=\max \{ C_{5,u}, \sup_{\delta}\max _{i}\{\var(g_{i,\delta}) \}\}$. Thus, for all $n \geq 1$ it holds

	\begin{equation*}
		\var g_{i, \delta} ^{(n)} \leq \frac{nC_{6,u} ^3}{\lambda _{4,u}} \lambda _{4,u}^n  ~\forall i~\textnormal{and}~\forall \delta \in [0,1)~.
	\end{equation*}Then,
	
	\begin{equation*}
		\var g_{i, \delta} ^{(n)}  \leq C_{7,u} \lambda _{5,u} ^n, \ \forall n \geq 1, \ \forall i, \forall \delta 
	\end{equation*}where $\lambda _{5,u} \in (\lambda_{4,u} , 1) $ and $C_{7,u}:= \sup _{n \geq 1} \left \{\dfrac{C_ {6,u} ^3}{\lambda_{4,u} } n \left(\dfrac{\lambda _{4,u} }{\lambda _{5,u}} \right) ^n \right \}$.

\end{proof}

\begin{proposition} If $\{F_\delta\}_{\delta \in [0,1)}$ is a BV Lorenz-like family. Then, there exist uniform constants $0<\lambda_u<1$, $C_u>0$ and $k \in \mathbb{N}$ such that for every $\mu \in \mathcal{BV}^+$, it holds
	
	\begin{equation}
		\Var(\func{F_\delta{_\ast}^k \mu} ) \leq  \lambda_u \Var ( \mu ) + C_u |\phi _x|_{1,1}, \ \forall \delta \in [0,1).  \label{lasotaingt234sdry}
	\end{equation}
	\label{las12efwsef3}
\end{proposition}
\begin{proof}
	
	Consider the iterate $F_\delta ^k$, where $k \in \mathbb{N}$ is from Lemma \ref{jshgd}. Applying Corollary \ref{las123}, we get 
	
	\begin{equation*}
		\Var(\func{F_\delta{_\ast}^k \mu} ) \leq  \beta _\delta \Var ( \mu ) + K_{3,\delta} |\phi _x|_{1,1}  
	\end{equation*}where

	\begin{equation*}
		\beta _\delta :=  \max _i \{\var g_{i,\delta} ^{(k)} + 3 \sup g_{i, \delta }^{(k)} \},
	\end{equation*}and
	
	\begin{equation*}
		K_{3,\delta} := \max_{i} \{\sup _{\overline{P_i}} g_{i,\delta} ^{(k)}\} \var  ^\diamond (G_\delta ^k) +  \max _{i} \left \lbrace\dfrac{ \var _{\overline{P _i}} (g_{i,\delta} ^{(k)}) + 2\sup_{\overline{P_i}} {g_{i,\delta} ^{(k)}}}{m_1(P _i)} \right \rbrace.
	\end{equation*} By Lemma \ref{uthd}, replacing $f$ by $\pi _y$, we have

	\begin{eqnarray*}
		\var  ^\diamond (G_\delta ^k) &\leq& q^k \sum _{j=1} ^{k}{q^j} \{2\var  ^\diamond (G_\delta) + 2q \}
		\\&\leq& q^k \sum _{j=1} ^{k}{q^j} \{2H_\delta + 2q \}. 
	\end{eqnarray*}Since by item (4) of UBV4 we have $\sup _{\delta \in [0,1)}{H_\delta} < \infty$, we get $\sup _{\delta \in [0,1)}{\var  ^\diamond (G_\delta ^k)} < \infty$. By the previous comments, item (2) and item (3) of UBV4, we define

	\begin{equation*}
		C_u := \sup _{\delta \in [0,1)} \{K_{3,\delta}\} <\infty.
	\end{equation*} We also set 
	
	\begin{equation*}
		\lambda _u := \sup _{\delta \in [0,1)} \{\beta _\delta \},
	\end{equation*}where, by Lemma \ref{jshgd}, it holds $\lambda _u <1 $. With these definitions we arrive at inequality (\ref{lasotaingt234sdry}).

\end{proof}

\begin{proposition} If $\{F_\delta\}_{\delta \in [0,1)}$ is a BV Lorenz-like family. Then, there exist uniform constants $0<\xi_u<1$, $B_u>0$ such that for every $\mu \in \mathcal{BV}^+$, all $\delta \in [0,1)$ and all $n \geq 1$, it holds
	
	\begin{equation*}
		\Var(\func{F_\delta{_\ast}^n \mu} ) \leq  \xi_u ^n B_u \Var ( \mu ) + B_u |\phi _x|_{1,1}. 
	\end{equation*}
	\label{dfgdffgb}
\end{proposition}

\begin{proof}
	By UBV1 we have gives us
	\begin{equation*}
		|\func{P}_{T_\delta}^{n}f|_{1,1} \leq D \lambda ^n | f|_{1,1}  +  D|f|_{1}, \ \ \forall n , \ \ \forall f \in BV_{1,1},
	\end{equation*}where $D> 0$ and $0<\lambda <1$. Then, since $|f|_{1} \leq | f|_{1,1}$,  it holds
	
	\begin{equation}
		|\func{P}_{T_\delta}^{n}f|_{1,1} \leq 2D| f|_{1,1}, \ \ \forall n , \ \ \forall f \in BV_{1,1},
		\label{boundingsllldsdu} 
	\end{equation}where $2D \geq 1$. In particular, (\ref{boundingsllldsdu}) holds if we replace $f$ by $\phi _x = \dfrac{d(\func {\pi _x {_\ast} \mu })}{dm_1}$ for each $\mu \in \mathcal{BV}^+$.
	
	By Proposition \ref{las12efwsef3} and a straightforward induction we have

	\begin{equation*}
		\Var(\func{F_\delta{_\ast}^{nk}}\mu )\leq \lambda _u^{n}\Var(\mu )+ 2DC_u \sum _{i=0} ^{n-1}{\lambda _u ^i} {%
			|\phi _x|_{1,1}}, \ \ \forall n \geq 0.  
	\end{equation*}Then,
	
	\begin{equation*}
		\Var(\func{F_\delta {_\ast}^{nk}}\mu )\leq \lambda _u^{n}\Var(\mu )+ \dfrac{2DC_u}{1 - \lambda _u}  {%
			|\phi _x|_{1,1}}, \ \ \forall n \geq 0.  
	\end{equation*}

	Consider $D$ ($2D \geq 1$) from equation (\ref{boundingsllldsdu}) and set $n=kq_n + r_n$, where  $0 \leq r_n < k$. Applying Proposition \ref{las123rtryrdfd} iterating $r_n$ times the inequality (\ref{lasotaingt234dffggdgh}) we get
	
	\begin{equation*}
		\Var(\func{F_\delta {_\ast}^{r_n} \mu} ) \leq \max _{i=0, \cdots, k} \{\beta _u ^i\} \Var(\mu) + 2DK_u   \sum _{j=0} ^{k} \beta _u ^j |\phi _x|_{1,1}. 
	\end{equation*}Thus,

	\begin{eqnarray*}
		\Var(\func{F_\delta {_\ast}^{n} \mu} ) &\leq& \lambda_u ^{q_n} \Var(\func{F_\delta {_\ast}^{r_n} \mu}) + \dfrac{2DC_u}{1-\lambda _u}  |\func{P^{r_n}_{T_\delta }}(\phi _x)|_{1,1}
		\\&\leq& \lambda_u ^{q_n} \left[\max _{i=0, \cdots, k} \{\beta _u ^i\} \Var(\mu) + 2DK_u   \sum _{j=0} ^{k} \beta _u ^j |\phi _x|_{1,1} \right] + \dfrac{4D^2C_u}{1-\lambda _u}  |\phi _x|_{1,1}
		\\&\leq& \lambda_u ^{q_n}\max _{i=0, \cdots, k} \{\beta _u ^i\} \Var(\mu) + \left[ 2DK_u   \sum _{j=0} ^{k} \beta _u ^j |\phi _x|_{1,1}  + \dfrac{4D^2C_u}{1-\lambda _u} \right] |\phi _x|_{1,1}
		\\&\leq& \lambda_u ^{\frac{n}{k} -\frac{r_n}{k}}\max _{i=0, \cdots, k} \{\beta _u ^i\} \Var(\mu) + \left[ 2DK_u   \sum _{j=0} ^{k} \beta _u ^j |\phi _x|_{1,1}  + \dfrac{4D^2C_u}{1-\lambda _u} \right] |\phi _x|_{1,1}
		\\&\leq& \left(\sqrt[k]{\lambda_u}\right) ^{n} \dfrac{\max _{i=0, \cdots, k} \{\beta _u ^i\}}{\lambda_u} \Var(\mu) + \left[ 2DK_u   \sum _{j=0} ^{k} \beta _u ^j |\phi _x|_{1,1}  + \dfrac{4D^2C_u}{1-\lambda _u} \right] |\phi _x|_{1,1}
		\\&\leq& \xi_u ^n B_u \Var ( \mu ) + B_u |\phi _x|_{1,1}, 
	\end{eqnarray*}where $B_u:= \max \left\{ \dfrac{\max _{i=0, \cdots, k} \{\beta _u ^i\}}{\lambda_u} ,  2DK_u   \sum _{j=0} ^{k} \beta _u ^j |\phi _x|_{1,1}  + \dfrac{4D^2C_u}{1-\lambda _u} \right\}$ and $\xi_u: =\sqrt[k]{\lambda_u} $.
	
\end{proof}



With all results established in this section, the proof of Proposition \ref{thshgf} is analogous to the Proposition \ref{reg}, where $B_u$ comes from Proposition \ref{dfgdffgb}.

\section{Appendix 2: Linearity of the restriction \label{remmm}}

Let us consider the measurable spaces $(N_1, \mathcal{N}_1)$ and $(N_2, \mathcal{N}_2)$, where $\mathcal{N}_1$ and $\mathcal{N}_2$ are the Borel's $\sigma$-algebra of $N_1$ and $N_2$ respectively.
Let $\mu \in \mathcal{AB}$ be a positive measure on the measurable space $(\Sigma, \mathcal{B})$, where $\Sigma = N_1 \times  N_2$ and $\mathcal{B} = \mathcal{N}_1 \times \mathcal{N}_2 $ and consider its disintegration $(\{\mu_{\gamma}\} _\gamma, \mu_x )$ along $\mathcal{F}^s$, where $\mu_x = \pi _x {_\ast} \mu$ and $d(\pi _x {_*}\mu) = \phi _ xdm_1$, for some $\phi _ x \in L^1(N_1, m_1)$. We will suppose that the $\sigma$-algebra  $\mathcal{B}$ has a countable generator.

\begin{proposition}
	Suppose that $\mathcal{B}$ has a countable generator, $\Gamma$. If $\{\mu_\gamma\}_\gamma$ and $\{\mu'_\gamma\}_\gamma$ are disintegrations of a positive measure $\mu$ relative to $\mathcal{F}^s$, then $\phi_x(\gamma)\mu_\gamma = \phi_x(\gamma)\mu'_\gamma$ $m_1$-a.e. $\gamma \in N_1$.
	\label{againa}
\end{proposition}
\begin{proof}
	Let $\mathcal{A}$ be the algebra generated by $\Gamma$. $\mathcal{A}$ is countable and $\mathcal{A}$ generates $\mathcal{B}$. For each $A \in \mathcal{A}$ define the sets $$G_A=\{\gamma \in N_1|  \phi_x(\gamma)\mu_\gamma (A) < \phi_x(\gamma)\mu'_\gamma (A) \}$$ and $$R_A=\{\gamma \in N_1|  \phi_x(\gamma)\mu_\gamma(A) > \phi_x(\gamma)\mu'_\gamma (A) \}.$$ If $\gamma \in G_A$ then $\gamma \subset \pi_x^{-1}(G_A)$ and $\mu_\gamma(A)= \mu _\gamma (A \cap \pi_x^{-1}(G_A))$. Otherwise, if $\gamma \notin G_A$ then $\gamma \cap \pi_x^{-1}(G_A) = \emptyset$ and $\mu _\gamma (A \cap \pi_x^{-1}(G_A)) =0 $. The same holds for $\mu'_\gamma$. Then, it holds

	\begin{equation*}
		\mu(A \cap \pi_x^{-1}(G_A))=
		\begin{cases}
			\int{\mu_\gamma (A\cap \pi ^{-1}(Q_A))\phi _x(\gamma)}dm_1 = \int_{Q_A}{\mu_\gamma (A)\phi _x(\gamma)}dm_1 \\ \int{\mu'_\gamma (A\cap \pi ^{-1}(Q_A))\phi _x(\gamma)}dm_1 = \int_{Q_A}{\mu'_\gamma (A)\phi _x(\gamma)}dm_1.
		\end{cases}
	\end{equation*}Since $\phi _x(\gamma)\mu_\gamma (A) <\mu'_\gamma (A)\phi _x(\gamma)$ for all $\gamma \in G_A$, we get $m_1(G_A)=0$. The same holds for $R_A$. Thus $$m_1 \left( \bigcup_{A \in \mathcal{A}} {R_A\cup G_A} \right)=0.$$It means that, $m_1$-a.e. $\gamma \in N_1$ the positive measures $\phi _x(\gamma)\mu_\gamma$ and $\mu'_\gamma \phi _x(\gamma)$ coincides for all measurable set $A$ of an algebra which generates $\mathcal {B}$. Therefore $\phi _x(\gamma)\mu_\gamma = \mu'_\gamma \phi _x(\gamma)$ for $m_1$-a.e. $\gamma \in N_1$.
	
\end{proof}

\begin{proposition}\label{hytre}
	Let $\mu_1,\mu_2 \in \mathcal{AB}$ be two positive measures and denote their marginal densities by $d({\mu_1}_x)= \phi_x dm_1$ and $d({\mu _2} _x) = \psi _x dm_1 $ , where $\phi_x, \psi_x \in L^1(m_1)$ respectively. Then $(\mu_1 + \mu_2)|_\gamma = \mu_1|_\gamma + \mu_2|_\gamma$ $m_1$-a.e. $\gamma \in N_1$.
\end{proposition}
\begin{proof}
	Note that $d(\mu_1 + \mu_2)= (\phi _x + \psi_x)dm_1$. Moreover, consider the disintegration of $\mu _1 + \mu_2$ given by $$(\{(\mu_1 + \mu_2) _\gamma\}_\gamma, (\phi _x + \psi_x) m_1),$$ where
	
	\begin{equation*}
		(\mu_1 + \mu_2) _\gamma =
		\begin{cases}
			\dfrac{\phi _x(\gamma)}{\phi _x(\gamma) + \psi _x(\gamma)} \mu _{1,\gamma}+ \dfrac{\psi _x(\gamma)}{\phi _x(\gamma) + \psi _x(\gamma)} \mu _{2,\gamma}, \ \textnormal{if} \ \phi _x(\gamma) + \psi _x(\gamma) \neq 0  \\ 0,\ \textnormal{if} \ \phi _x(\gamma) + \psi _x(\gamma) = 0.
		\end{cases}
	\end{equation*}Then, by Proposition \ref{againa} for $m_1$-a.e. $\gamma \in N_1$, it holds $$(\phi _x + \psi_x)(\gamma)(\mu_1 + \mu_2) _\gamma = \phi _x(\gamma) \mu _{1,\gamma} + \psi _x(\gamma) \mu _{2,\gamma}.$$Therefore, $(\mu_1 + \mu_2)|_\gamma = \mu_1|_\gamma + \mu_2|_\gamma$ $m_1$-a.e. $\gamma \in N_1$.

\end{proof}

\begin{definition}\label{disj}
	We say that a positive measure $\lambda_1$ is disjoint from a positive measure $\lambda _2$ if $(\lambda_1 - \lambda _2)^+ = \lambda _1$ and $(\lambda_1 - \lambda _2)^- = \lambda _2$.
\end{definition}

\begin{remark}
	A straightforward computations yields that if $\lambda _1 + \lambda _2$ is disjoint from $\lambda _3$, then both $\lambda _1$ and $\lambda _2$ are disjoint from $\lambda _3$, where $\lambda_1, \lambda_2$ and $\lambda_3$ are all positive measures. 
\end{remark}

\begin{lemma}Suppose that $\mu =\mu ^{+}-\mu ^{-}$ and $\nu =\nu ^{+}-\nu ^{-}$ are the Jordan decompositions of the signed measures $\mu$ and $\nu$. Then, there exist positive measures $\mu_1$, $\mu_2$, $\mu ^{++}$, $\mu ^{--}$, $\nu ^{++}$ and $\nu ^{--}$ such that $\mu ^{+}=\mu ^{++}+\mu_{1}$ $\mu ^{-}=\mu ^{--}+\mu_{2}$ and  $\nu ^{+}=\nu ^{++}+\mu_{2}$, $\nu ^{-}=\nu ^{--}+\mu_{1}$.
	\label{fhdj}
\end{lemma}
\begin{proof}
	Suppose $\mu = \nu_1 - \nu_2$ with $\nu_1$ and $\nu_2$ positive measures. Let $\mu^+ $ and $\mu^-$ be the Jordan decomposition of $\mu$.
	Let $\mu' = \nu_1 - \mu^+$, then $\nu_1 = \mu^- + \mu'$. Indeed $\mu^+ - \mu ^- = \nu _1 - \nu _2$ which implies that $\mu^+ - \nu_1 = \mu^- - \nu_2$.
	Thus if $\nu_1, \nu_2$ is a decomposition of $\mu$, then $\nu_1 = \mu^+ + \mu '$ and $\nu_2 = \mu^- + \mu '$ for some positive measure $\mu'$.
	Now, consider $\mu = \mu ^+ - \mu ^-$ and $\nu = \nu ^+ - \nu ^-$. Since the pairs of positive measures $\mu^+, \nu^-$ and $(\mu^+ - \nu^-)^+$, $(\mu ^+ - \nu ^-)^-$ are both decompositions of $\mu^+ - \nu ^-$, by the above comments, we get that $\mu^+ =(\mu^+ - \nu^-)^+ + \mu_1$ and $\nu^- =(\mu^+ - \nu^-)^- + \mu_1$, for some positive measure $\mu_1$. Analogously, since the pairs of positive measures $\mu^-, \nu^+$ and $(\nu^+ - \mu^-)^+$, $(\nu ^+ - \mu ^-)^-$ are both decompositions of $\nu^+ - \mu ^-$, by the above comments, we get that $\nu^+ =(\nu^+ - \mu^-)^+ + \mu_2$ and $\mu^- =(\nu^+ - \mu^-)^- + \mu_2$, for some positive measure $\mu_2$.
	By definition \ref{disj}, $\mu+$ and $\mu^-$ are disjoint, and so are $(\mu ^+ - \nu^-)^+$ and $(\nu ^+ - \mu^-)^-$. Analogously, $\nu+$ and $\nu^-$ are disjoint, and so are $(\mu ^+ - \nu^-)^-$ and $(\nu ^+ - \mu^-)^+$. Moreover, since $(\mu ^+ - \nu^-)^+$ and $(\mu ^+ - \nu^-)^-$ are disjoint, so are $(\nu ^+ - \mu^-)^+$ and $(\nu ^+ - \mu^-)^-$. This gives that, the pair $(\mu ^+ - \nu^-)^+ + (\nu ^+ - \mu^-)^+$, $(\nu ^+ - \mu^-)^- + (\mu ^+ - \nu^-)^-$ is a Jordan decomposition of $\mu + \nu$ and we are done.

\end{proof}

\begin{proposition}
	Let $\mu,\nu \in \mathcal{AB}$ be two signed measures. Then $(\mu + \nu)|_\gamma = \mu|_\gamma + \nu|_\gamma$ $m_1$-a.e. $\gamma \in N_1$.
\end{proposition}
\begin{proof}

	Suppose that $\mu =\mu ^{+}-\mu ^{-}$ and $\nu =\nu ^{+}-\nu ^{-}$ are the Jordan decompositions of $\mu$ and $\nu$ respectively. By definition, $\mu |_{\gamma }=\mu ^{+}|_{\gamma }-\mu ^{-}|_{\gamma }$, $\nu |_{\gamma }=\nu ^{+}|_{\gamma }-\nu ^{-}|_{\gamma }$. 
	
	By Lemma \ref{fhdj}, suppose that $\mu ^{+}=\mu ^{++}+\mu_{1}$, $\mu ^{-}=\mu ^{--}+\mu_{2}$ and  $\nu^{+}=\nu ^{++}+\mu_{2}$, $\nu ^{-}=\nu ^{--}+\mu_{1}$. In a way that $(\mu +\nu )^{+}=\mu ^{++}+\nu ^{++}$ and $(\mu +\nu )^{-}=\mu ^{--}+\nu ^{--}$. By Proposition \ref{hytre}, it holds $\mu ^{+}|_{\gamma }=\mu ^{++}|_{\gamma }+\mu_{1}|_{\gamma }$, $\mu ^{-}|_{\gamma }=\mu ^{--}|_{\gamma }+\mu_{2}|_{\gamma }$, $\nu ^{+}|_{\gamma }=\nu ^{++}|_{\gamma }+\mu_{2}|_{\gamma }$ and $\nu ^{-}|_{\gamma }=\nu ^{--}|_{\gamma }+\mu_{1}|_{\gamma }$.
	
	Moreover,
	
	$(\mu +\nu )^{+}|_{\gamma }=\mu ^{++}|_{\gamma }+\nu ^{++}|_{\gamma }$
	
	$(\mu +\nu )^{-}|_{\gamma }=\mu ^{--}|_{\gamma }+\nu ^{--}|_{\gamma }$
	
	Putting all together, we get:
	
	\begin{eqnarray*}
		(\mu +\nu )|_{\gamma } &=&(\mu +\nu )^{+}|_{\gamma }-(\mu +\nu
		)^{-}|_{\gamma } \\
		&=&\mu ^{++}|_{\gamma }+\nu ^{++}|_{\gamma }-(\mu ^{--}|_{\gamma }+\nu
		^{--}|_{\gamma }) \\
		&=&\mu ^{++}|_{\gamma }+\mu_{1}|_{\gamma }+\nu ^{++}|_{\gamma }+\mu_{2}|_{\gamma
		}-(\mu ^{--}|_{\gamma }+\mu_{2}|_{\gamma }+\nu ^{--}|_{\gamma }+\mu_{1}|_{\gamma
		}) \\
		&=&\mu ^{+}|_{\gamma }-\mu ^{-}|_{\gamma }+\nu ^{+}|_{\gamma }-\nu
		^{-}|_{\gamma } \\
		&=&\mu |_{\gamma }+\nu |_{\gamma }.
	\end{eqnarray*}
\end{proof}We immediately arrive at the following
\begin{corollary}\label{lasttttt}
	Let $\mu \in \mathcal{AB}$ be a signed measure and $\mu = \mu^+ - \mu ^-$ its Jordan decomposition. If $\mu_1$ and $\mu_2$ are positive measures such that $\mu = \mu_1 - \mu _2$, then $ \mu|_\gamma =\mu_1|_\gamma - \mu_2|_\gamma$. It means that, the restriction does not depends on the decomposition of $\mu$.
\end{corollary}

\section{Appendix 3: Uniform Family of Operators}\label{jsdhjfnsd}

In this section, we prove the main results on uniform families of operators stated in Section \ref{unifop}. We
state a general lemma on the stability of fixed points satisfying certain assumptions. Consider two operators $\func{L} _{0}$ and $\func{L}  _{\delta }$ preserving a normed
space of signed measures $\mathcal{B\subseteq }\mathcal{SB}(X)$ with norm $||\cdot ||_{\mathcal{B}%
}$. Suppose that $f_{0},$ $f_{\delta }\in \mathcal{B}$ are fixed
points of $\func{L}  _{0}$ and $\func{L}  _{\delta }$, respectively.

\begin{lemma}
	\label{gen}Suppose that:
	
	\begin{enumerate}
		\item[a)] $||\func{L}  _{\delta }f_{\delta }-\func{L} _{0}f_{\delta }||_{\mathcal{B}}<\infty $;
		
		\item[b)] For all $i\geq 1$, $\func{L} _{0}^{ i}$ is continuous on $\mathcal{B}$: for each $ i \geq 1$, $\exists
		\,C_{i}~s.t.~\forall g\in \mathcal{B},~||\func{L}  _{0}^{ i}g||_{\mathcal{B}}\leq
		C_{i}||g||_{\mathcal{B}}.$
	\end{enumerate}
	
	Then, for each $N \geq 1$, it holds%
	\begin{equation*}
		||f_{\delta }-f_{0}||_{\mathcal{B}}\leq ||\func{L} _{0}^{ N}(f_{\delta }-f_{0})||_{%
			\mathcal{B}}+||\func{L}  _{\delta } f_{\delta }-\func{L}  _{0}f_{\delta }||_{\mathcal{B}%
		}\sum_{i\in \lbrack 0,N-1]}C_{i}.  
	\end{equation*}
\end{lemma}

\begin{proof}
	The proof is a direct computation. First note that,%
	\begin{eqnarray*}
		||f_{\delta }-f_{0}||_{\mathcal{B}} &\leq &||\func{L}_{\delta }^{N}f_{\delta
		}-\func{L}_{0}^{ N}f_{0}||_{\mathcal{B}} \\
		&\leq &||\func{L}_{0}^{ N}f_{0}-\func{L}_{0}^{ N}f_{\delta }||_{\mathcal{B}%
		}+||\func{L}_{0}^{ N}f_{\delta }-\func{L}_{\delta }^{ N}f_{\delta }||_{\mathcal{B}} \\
		&\leq &||\func{L}_{0}^{N}(f_{0}-f_{\delta })||_{\mathcal{B}}+||\func{L}_{0}^{ N}f_{\delta
		}-\func{L}_{\delta }^{N}f_{\delta }||_{\mathcal{B}}.
	\end{eqnarray*}%
	Moreover,%
	\begin{equation*}
		\func{L}_{0}^{ N}-\func{L}_{\delta }^{ N}=\sum_{k=1}^{N}\func{L}_{0}^{(N-k)}(\func{L}_{0}-\func{L} _{\delta
		})\func{L}_{\delta }^{ (k-1)}
	\end{equation*}%
	hence%
	\begin{eqnarray*}
		(\func{L}_{0}^{ N}-\func{L}_{\delta }^{ N})f_{\delta} &=&\sum_{k=1}^{ N}\func{L}_{0}^{ (N-k)}(\func{L}_{0}-\func{L}_{\delta
		})\func{L}_{\delta }^{ (k-1)}f_{\delta } \\
		&=&\sum_{k=1}^{N}\func{L}_{0}^{ (N-k)}(\func{L}_{0}-\func{L}  _{\delta })f_{\delta }
	\end{eqnarray*}%
	by item b), we have%
	\begin{eqnarray*}
		||(\func{L}_{0}^{ N}-\func{L}_{\delta }^{ N})f_{\delta }||_{\mathcal{B}} &\leq
		&\sum_{k=1}^{N}C_{N-k}||(\func{L}_{0}-\func{L} _{\delta })f_{\delta }||_{\mathcal{B}} \\
		&\leq &||(\func{L}_{0}-\func{L}_{\delta })f_{\delta }||_{\mathcal{B}}\sum_{i\in \lbrack
			0,N-1]}C_{i}
	\end{eqnarray*}%
	and then%
	\begin{equation*}
		||f_{\delta }-f_{0}||_{\mathcal{B}}\leq ||\func{L}_{0}^{ N}(f_{0}-f_{\delta })||_{%
			\mathcal{B}}+||(\func{L} _{0}-\func{L}  _{\delta })f_{\delta }||_{\mathcal{B}}\sum_{i\in
			\lbrack 0,N-1]}C_{i}.
	\end{equation*}
\end{proof}

Now, let us apply the statement to our family of operators satisfying
assumptions UF1--UF4, supposing $\ B_{w}=\mathcal{B}$. We have the
following

\begin{proposition}
	\label{dlogd2}Suppose $\{\func{L}_{\delta }\}_{\delta \in \left[0, 1 \right)}$ is a uniform family of operators as in Definition \ref{UF}, where $f_{0}$ is the unique fixed point of $\func{L}_{0}$ in $B_{w}$ and $%
	f_{\delta }$ is a fixed point of $\func{L} _{\delta }$. Then, there is a $\delta _0\in (0,1)$ such that for all $\delta \in (0, \delta_0]$ it holds %
	\begin{equation*}
		||f_{\delta }-f_{0}||_{w}=O(\delta \log \delta ).
	\end{equation*}
\end{proposition}

\begin{proof}
	
	
	First note that, if $\delta \geq 0$ is small enough, then $\delta \leq -\delta \log {\delta} $. Moreover, $x -1 \leq \lfloor  x \rfloor$, for all $x \in \mathbb{R}$.

	By UF2,%
	\begin{equation*}
		||\func{L}_{\delta }f_{\delta }-\func{L}_{0}f_{\delta }||_{w}\leq \delta C
	\end{equation*}%
	(see Lemma \ref{gen}, item a) ) and UF$4$ yields $C_{i}\leq M_{2}.$ 
	
	Hence, by Lemma \ref{gen} we have 
	\begin{equation*}
		||f_{\delta }-f_{0}||_{w}\leq \delta CM_{2}N+||\func{L}_{0}^{ N}(f_{0}-f_{\delta
		})||_{w}.
	\end{equation*}%
	By the exponential convergence to equilibrium of $\func{L}_{0}$ (UF3), there exists $0<\rho _2 <1$ and $C_2 >0$ such that (recalling that
	by UF1 $||(f_{\delta }-f_{0})||_{s}\leq 2M$)
	\begin{eqnarray*}
		||\func{L} _{0}^{ N}(f_{\delta }-f_{0})||_{w} &\leq &C_{2}\rho _{2}^{N}||(f_{\delta
		}-f_{0})||_{s} \\
		&\leq &2C_{2}\rho _{2}^{N}M
	\end{eqnarray*}%
	hence%
	\begin{equation*}
		||f_{\delta }-f_{0}||_{\mathcal{B}}\leq \delta CM_{2}N+2C_{2}\rho _{2}^{N}M.
	\end{equation*}%
	Choosing $N=\left\lfloor \frac{\log \delta }{\log \rho _{2}}\right\rfloor $, we have%
	\begin{eqnarray*}
		||f_{\delta }-f_{0}||_{\mathcal{B}} &\leq &\delta CM_{2}\left\lfloor \frac{%
			\log \delta }{\log \rho _{2}}\right\rfloor +2C_{2}\rho _{2}^{\left\lfloor 
			\frac{\log \delta }{\log \rho _{2}}\right\rfloor }M \\
		&\leq &\delta \log \delta CM_{2}\frac{1}{\log \rho _{2}}+2C_{2}\rho _2 ^ {  \frac{\log \delta }{\log \rho _{2}} -1} M \\
		&\leq & \delta \log \delta CM_{2}\frac{1}{\log \rho _{2}}+\frac{2C_{2}\rho _2 ^ { \frac{\log \delta }{\log \rho _{2}}} M}{\rho _2} \\
		&\leq & \delta \log \delta CM_{2}\frac{1}{\log \rho _{2}}+\frac{2C_{2}\delta M}{\rho _2} \\
		&\leq & \delta \log \delta CM_{2}\frac{1}{\log \rho _{2}}-\frac{2C_{2}\delta \log \delta M}{\rho _2}\\
		&\leq & \delta \log \delta \left( \frac{CM_{2}}{\log \rho _{2}}-\frac{2C_{2} M}{\rho _2} \right).
		\notag
	\end{eqnarray*}
\end{proof}

\section{Appendix 4: On Disintegration of Measures}\label{disint}

In this section, we prove some results on disintegration of absolutely continuous measures with respect to a measure $\mu_0 \in \mathcal{AB}$. Precisely, we are going to prove Lemma \ref{hdgfghddsfg}. 

Let us fix some notations. Denote by $(N_1,m_1)$ and $(N_2,m_2)$ the spaces defined in section \ref{sec2}. For a $\mu_0$-integrable function $f: N_1 \times N_2 \longrightarrow \mathbb{R}$ and a pair $(\gamma,y) \in N_1 \times N_2$ ($\gamma \in N_1$ and $y \in N_2$) we denote by $f_\gamma : N_2 \longrightarrow \mathbb{R}$, the function defined by $f_\gamma(y) = f(\gamma,y)$ and $f|_\gamma$ the restriction of $f$ on the set $\{\gamma\} \times N_2$. Then $f_\gamma = f|_\gamma \circ \pi _{y,\gamma}^{-1}$ and $f_\gamma \circ \pi_{y,\gamma} = f|_\gamma$, where $\pi _{y,\gamma}$ is restriction of the projection $\pi _y(\gamma,y):=y$ on the set $\{\gamma\} \times N_2$. When no confusion can be done, we will denote the leaf $\{\gamma\} \times N_2$, just by $\gamma$.

From now and ahead, for a given positive measure $\mu \in \mathcal{AB}$, on $N_1 \times N_2$, $\widehat{\mu}$ stands for the measure $\pi_x {_*} \mu$. Where $\pi_x$ is the projection on the first coordinate, $\pi_x(x,y)=x$.

For each measurable set $A \subset N_1$, define $g: N_1 \longrightarrow \mathbb{R}$, by $$g(\gamma)= \phi_x(\gamma)\int{\chi _{ \pi _x ^{-1}(A)}|_\gamma (y) f|_\gamma(y)}d\mu_{0,\gamma}(y)$$and note that

\begin{equation*}
	g(\gamma)=
	\begin{cases}
		\phi_x(\gamma) \displaystyle{\int {f|_\gamma (y)}d \mu_{0,\gamma}}, \ \hbox{if} \ \gamma \in A \\
		0, \ \hbox{if} \ \gamma \notin A.
	\end{cases}
\end{equation*}Then, it holds $$g(\gamma) = \chi _A (\gamma) \displaystyle{\phi_x(\gamma)  \int {f|_\gamma (y)}d \mu_{0,\gamma}}.$$

\begin{proof}{(of Lemma \ref{hdgfghddsfg})}

	For each measurable set $A \subset N_1$, we have
	\begin{eqnarray*}
		\int_A{\dfrac{\pi _x ^* (f\mu_0)}{dm_1}}dm_1 &=&\int{\chi _A \circ \pi _x}d(f\mu_0)
		\\&=&\int{\chi _{ \pi _x ^{-1}(A)} f}d\mu_0
		\\&=&\int \left[\int{\chi _{ \pi _x ^{-1}(A)}|_\gamma (y) f|_\gamma(y)}d\mu_{0,\gamma}(y)\right]d(\phi_x m_1)(\gamma)
		\\&=&\int \left[\phi_x(\gamma)\int{\chi _{ \pi _x ^{-1}(A)}|_\gamma (y) f|_\gamma(y)}d\mu_{0,\gamma}(y)\right]d(m_1)(\gamma)
		\\&=&\int{g(\gamma)}d(m_1)(\gamma)
		\\&=&\int_A {\left[\int{f_\gamma(y)}d\mu_{0}{|_\gamma}(y)\right]}d(m_1)(\gamma).
	\end{eqnarray*}Thus, it holds 
	
	\begin{equation*}
		\dfrac{\pi _x {_*} (f\mu_0)}{dm_1} (\gamma) =  \int{f_\gamma(y)}d\mu_{0}{|_\gamma}, \ \hbox{for} \ m_1-\hbox{a.e.} \ \gamma \in N_1.
	\end{equation*}And by a straightforward computation 
	
	\begin{equation}\label{gh}
		\dfrac{\pi _x {_*} (f\mu_0)}{dm_1}(\gamma) = \phi_x (\gamma) \int{f|_\gamma(y)}d\mu_{0,\gamma}, \ \hbox{for} \ m_1-\hbox{a.e.} \ \gamma \in N_1.
	\end{equation}Thus, equation (\ref{fjgh}) is established.
	
	\begin{remark}\label{ghj}
		Setting,
		\begin{equation}\label{tyu}
			\overline{f}:= \dfrac{\pi _x {_*} (f\mu_0)}{dm_1},
		\end{equation}we get, by equation (\ref{gh}), $\overline{f}(\gamma)=0$ iff $\phi_x (\gamma) = 0$ or $\displaystyle{\int{f|_\gamma(y)}d\mu_{0,\gamma} (y)=0}$, for  $m_1$-a.e. $\gamma \in N_1$. 
	\end{remark}

	Now, let us see that, by the $\widehat{\nu}$-uniqueness of the disintegration, equation (\ref{gdfgdgf}) holds. To do it, define, for $m_1$-a.e. $\gamma \in N_1$, de function $h_\gamma : N_2 \longrightarrow \mathbb{R}$, in a way that
	
	\begin{equation}\label{jri}
		h_\gamma (y) =
		\begin{cases}
			\dfrac{f|_\gamma (y)}{\int{f|_\gamma(y)}d\mu_{0,\gamma}(y)} , \ \hbox{if} \ \gamma \in B ^c \\
			0, \ \hbox{if} \ \gamma \in B.
		\end{cases}
	\end{equation}Let us prove equation (\ref{gdfgdgf}) by showing that, for all measurable set $E \subset N_1 \times N_2$, it holds $$f \mu _0 (E)  = \int _{N_1} {\int _{E \cap \gamma} {h_\gamma(y)}}d\mu _{0, \gamma} (y)d (\pi_x {_*}(f \mu_0))(\gamma).$$In fact, by equations (\ref{gh}), (\ref{tyu}), (\ref{jri}) and remark \ref{ghj}, we get
	
	\begin{eqnarray*}
		f\mu_0 (E) &=& \int _E {f} d\mu_0
		\\&=& \int _{N_1} \int _{E\cap \gamma} {f|_\gamma} d\mu_{0, \gamma}d (\phi_x m_1)(\gamma)
		\\&=& \int _{B^c} \int _{E\cap \gamma} {f|_\gamma} d\mu_{0, \gamma}d (\phi_x m_1)(\gamma)
		\\&=& \int _{B^c} \int{f|_\gamma(y)}d\mu_{0,\gamma}(y)\phi_x (\gamma)\left[ \dfrac{1}{\int{f|_\gamma(y)}d\mu_{0,\gamma}(y)}\int _{E\cap \gamma} {f|_\gamma}  d\mu_{0, \gamma}\right]d m_1(\gamma)
		\\&=& \int _{B^c} \overline{f}(\gamma)\left[ \dfrac{1}{\int{f|_\gamma(y)}d\mu_{0,\gamma}(y)}\int _{E\cap \gamma} {f|_\gamma}  d\mu_{0, \gamma}\right]d m_1(\gamma)
		\\&=& \int _{B^c} \left[ \dfrac{1}{\int{f|_\gamma(y)}d\mu_{0,\gamma}(y)}\int _{E\cap \gamma} {f|_\gamma}  d\mu_{0, \gamma}\right]d \overline{f}m_1(\gamma)
		\\&=& \int _{B^c} {\int _{E \cap \gamma} {h_\gamma(y)}}d\mu _{0, \gamma} (y)d (\pi_x {_*}(f \mu_0))(\gamma)
		\\&=& \int _{N_1} {\int _{E \cap \gamma} {h_\gamma(y)}}d\mu _{0, \gamma} (y)d (\pi_x {_*}(f \mu_0))(\gamma).
	\end{eqnarray*}And we are done.
	
\end{proof}


\begin{thebibliography}{99}
	
	
	
	\bibitem{AS} [10.4064/fm224-3-2]
	\newblock  J.  F. Alves and M. Soufi, 
	\newblock Statistical stability of geometric	Lorenz attractors,
	\newblock \emph{Fundamenta Mathematicae}, \textbf{224} (2014), 219--231.
	
	
	\bibitem{AGP} [10.1007/s00209-013-1231-0]
	\newblock  V. Araujo, S. Galatolo and M. Pacifico, 
	\newblock Decay of correlations for maps with uniformily contracting fibers and logarithm law for singular hyperbolic attractors,
	\newblock \emph{Mathematische Zeitschrift}, \textbf{276} (2012).
	
	
	\bibitem{AP} [10.1007/978-3-642-11414-4]
	\newblock V. Araujo and M. Pacifico,
	\newblock \emph{Three-dimensional flows},
	\newblock  Springer-Verlag, New York, 2010. 
	
	
	\bibitem{BR}
	\newblock W. Bahsoun and M. Ruziboev,
	\newblock On the statistical stability of Lorenz attractors with a $C^{1+\alpha}$ stable foliation,
	\newblock preprint, arXiv{1701.08601}. 
	
	\bibitem{B}[10.1007/s10955-016-1663-0]
	\newblock  V. Baladi, 
	\newblock The quest for the ultimate anisotropic Banach space,
	\newblock \emph{Journal of Statistical Physics}, \textbf{166} (2016), 525--557.
	
	
	
	\bibitem{Ba} [10.1142/9789812813633]
	\newblock V. Baladi,
	\newblock \emph{Positive transfer operators and decay of
		correlations},
	\newblock  World Scientific Pub Co Inc, London, 2000. 
	
	
	
	
	\bibitem{BT}
	\newblock  V. Baladi and M. Tsujii, 
	\newblock Anisotropic Holder and Sobolev
	spaces for hyperbolic diffeomorphisms,
	\newblock \emph{Annales de l'institut Fourier}, \textbf{57} (2007), 127--154.
	
	
	\bibitem{BaG} [10.3934/jmd.2010.4.91]
	\newblock V. Baladi and S. Gou\"{e}zel, 
	\newblock Banach spaces for piecewise cone hyperbolic maps,
	\newblock \emph{Journal of Modern Dynamics}, \textbf{4} (2009), 91--137.
	
	
	\bibitem{BaG2} [10.1016/j.anihpc.2009.01.001]
	\newblock V. Baladi and S. Gou\"{e}zel, 
	\newblock Good Banach spaces for piecewise hyperbolic maps via interpolation,
	\newblock \emph{Annales de l'Institut Henri Poincare (C) Non Linear Analysis}, \textbf{26} (2007), 1453--1481. 
	
	\bibitem{BG} [10.1007/978-1-4612-2024-4]
	\newblock A. Boyarsky and P. Gora,
	\newblock \emph{Laws of Chaos - Invariant Measures
		and Dynamical Systems in One Dimension},
	\newblock  Birkhauser, Boston, 1997. 
	
	
	\bibitem{BL} [10.3934/jmd.2007.1.301]
	\newblock O. Butterley and C. Liverani, 
	\newblock Smooth Anosov flows: correlation spectra and stability,
	\newblock \emph{J. Mod. Dyn.}, \textbf{1} (2007), 301--322. 
	
	\bibitem{D} [10.1016/j.chaos.2018.08.028]
	\newblock M. Demers, 
	\newblock A gentle introduction to anisotropic banach spaces,
	\newblock \emph{Chaos, Solitons \& Fractals}, \textbf{116} (2018), 29--42. 
	
	\bibitem{DL} [10.1090/S0002-9947-08-04464-4]
	\newblock M. Demers and C Liverani, 
	\newblock Stability of Statistical Properties in Two-Dimensional Piecewise Hyperbolic Maps,
	\newblock \emph{Transactions of the American Mathematical Society}, \textbf{360} (2006), 4777--4814. 
	
	
	
	\bibitem{BM}
	\newblock O Butterley and I Melbourne,
	\newblock Disintegration of Invariant
	Measures for Hyperbolic Skew Products,
	\newblock preprint, arXiv{1701.08601}. 
	
	\bibitem{DZ} [10.3934/jmd.2011.5.665]
	\newblock M. Demers and HZ. Zhang, 
	\newblock Spectral analysis of the transfer
	operator for the Lorentz gas,
	\newblock \emph{Journal of Modern Dynamics}, \textbf{5} (2011), 665--709. 
	
	\bibitem{DZ2} [10.1007/s00220-013-1820-0]
	\newblock M. Demers and HZ. Zhang, 
	\newblock A functional Analytic approach to
	perturbations of the Lorentz gas,
	\newblock \emph{Communications in Mathematical Physics}, \textbf{324} (2013), 767--830. 
	
	\bibitem{G}
	\newblock  S. Galatolo,
	\newblock Statistical properties of dynamics. Introduction to the functional analytic approach. Lecture notes for the
	Hokkaido-Pisa University summer school 2015,
	\newblock preprint, arXiv{1510.02615}.
	
	
	\bibitem{Gjep} [10.5802/jep.73]
	\newblock S. Galatolo, 
	\newblock Quantitative statistical stability, speed of convergence to equilibrium and partially hyperbolic skew products,
	\newblock \emph{Journal de l'École polytechnique}, \textbf{5} (2018), 377--405. 
	
	\bibitem{GNS} [10.3934/jcd.2015.2.51]
	\newblock S. Galatolo,  I. Nisoli and B. Saussol, 
	\newblock An elementary way to rigorously estimate convergence to equilibrium and escape rates,
	\newblock \emph{Journal of Computational Dynamics}, \textbf{2} (2014), 51--64.
	
	\bibitem{GP} [10.1017/S0143385709000856]
	\newblock S. Galatolo and M. J. Pacifico, 
	\newblock Lorenz-like flows: Exponential decay of correlations for the Poincaré map, logarithm law, quantitative recurrence,
	\newblock \emph{Ergodic Theory and Dynamical Systems}, \textbf{30} (2010), 1703--1737.
	
	\bibitem{GL} [10.1017/S0143385705000374]
	\newblock S Gouezel and C Liverani, 
	\newblock Banach spaces adapted to Anosov systems,
	\newblock \emph{Ergodic Theory and Dynamical Systems}, \textbf{26} (2006), 189--217.
	
	
	\bibitem{Gk} [10.1007/BF00532744]
	\newblock G. Keller, 
	\newblock Generalized bounded variation and applications
	to piecewise monotonic transformations,
	\newblock \emph{Probability Theory and Related Fields}, \textbf{69} (1985), 461--478.
	
	
	\bibitem{KL} 
	\newblock G Keller and C Liverani, 
	\newblock Stability of the spectrum for transfer operators,
	\newblock \emph{Annali della Scuola Normale Superiore di Pisa. Classe di Scienze. Serie IV}, \textbf{28} (1999), 141--152.
	
	\bibitem{IM} [10.2307/1969514]
	\newblock C. Ionescu-Tulcea and Marinescu G, 
	\newblock Theorie ergodique pour des classes d' operateurs non completement continues,
	\newblock \emph{Annals of Mathematics}, \textbf{52} (1950), 140--147.
	
	
	
	
	\bibitem{LY} [10.2307/1996575]
	\newblock A. Lasota and J.Yorke, 
	\newblock On the Existence of Invariant Measures for Piecewise Monotonic Transformations,
	\newblock \emph{Transactions of The American Mathematical Society}, \textbf{186} (1973), 481--488.
	
	\bibitem{L2} 
	\newblock C. Liverani, 
	\newblock Invariant measures and their properties. A functional analytic point of view,
	\newblock \emph{Scuola Normale Superiore in Pisa}, (2004).
	
	\bibitem{L3} [10.2307/2118636]
	\newblock C. Liverani, 
	\newblock Decay Of Correlations,
	\newblock \emph{Annals of Mathematics}, \textbf{142} (1997), 239--301.
	
	
	\bibitem{L}
	\newblock R. Lucena,
	\newblock  \emph{Spectral Gap for Contracting Fiber Systems and Applications},
	\newblock  Ph.D thesis, Universidade Federal do Rio de Janeiro in Brazil, 2015.
	
	
	
	
	
	\bibitem{Kva} 
	\newblock K. Oliveira and M. Viana,
	\newblock \emph{Fudamentos da Teoria Ergódica},
	\newblock  Colec\~ao Fronteiras da Matematica - SBM, Brazil, 2014. 
	
	\bibitem{RE} 
	\newblock J. Rousseau-Egele, 
	\newblock Un Theoreme de la Limite Locale Pour une Classe de Transformations Dilatantes et Monotones par Morceaux, \newblock \emph{The Annals of Probability}, \textbf{11} (1983), 772--788.
	
	\bibitem{V} 
	\newblock M. Viana,
	\newblock \emph{Stochastic dynamics of deterministic systems},
	\newblock Brazillian Math. Colloquium 1997, IMPA, \url{http://w3.impa.br/~viana/out/sdds.pdf}
	
	
	
	
	
	
	
\end{thebibliography}
\end{document}